\documentclass[11pt]{amsart}

\usepackage[utf8]{inputenc} 
\usepackage{graphicx}
\usepackage{amsmath}
\usepackage{amsthm}
\usepackage{amsfonts}
\usepackage{amssymb}
\usepackage{booktabs}
\usepackage{verbatim}
\usepackage{subcaption}
\usepackage{latexsym}
\usepackage{mdframed}
\usepackage{enumerate}
\usepackage{array}
\usepackage{commath}
\usepackage{xcolor}
\newcommand{\dist}{3.0}
\usepackage[utf8]{inputenc} 
\usepackage{graphicx}
\usepackage{mathtools}
\usepackage{hyperref}
\usepackage{multirow}
\usepackage{centernot} 
\usepackage{parcolumns}
\usepackage{multicol}
\usepackage{mathrsfs}
\usepackage{xfrac}
\usepackage{tabularx}
\usepackage[all]{xy}
\usepackage{tabu}
\usepackage{enumitem}
\usepackage{bbm} 
\usepackage{thmtools}
\usepackage{thm-restate}
\usepackage{cleveref}
\usepackage{graphicx}
\usepackage{bm}
\usepackage[many]{tcolorbox}
\usepackage[left=3cm,right=3cm,top=3cm,bottom=3cm]{geometry}
\usepackage{centernot}
\usepackage{algorithm}
\usepackage{algpseudocode}
\usepackage{pgf,tikz}
\usetikzlibrary{graphs, positioning, quotes, arrows, calc}
\usepackage{capt-of}
\usepackage{thmtools}  
\usepackage[normalem]{ulem}
\usepackage{scalerel}

\theoremstyle{plain}
\newtheorem{thm}{Theorem}[section]
\newtheorem{lemma}[thm]{Lemma}
\newtheorem{cor}[thm]{Corollary}
\newtheorem{prop}[thm]{Proposition}
\newtheorem{fact}[thm]{Fact}

\newtheorem*{thm*}{Theorem}

\theoremstyle{remark}
\newtheorem{rmk}[thm]{Remark}

\theoremstyle{definition}
\newtheorem{defn}[thm]{Definition}

\numberwithin{equation}{section}

\def\E{\mathbb{E}}

\def\1{\mathbf{1}}

\newcommand{\0}{\hspace{.01in}}

\DeclareMathOperator{\supp}{supp}

\DeclareMathOperator{\spa}{span}

\DeclareMathOperator{\diag}{diag}

\DeclareMathOperator{\rank}{rank}
\DeclareMathOperator{\row}{row}

 \usepackage{xcolor}

\colorlet{lightviolet}{violet!70}

\newcommand{\sqcell}[1]{\parbox[c][2cm][c]{2cm}{\centering #1}}

\usepackage{natbib}   

\title{Graph Disjointness with Applications to Reversible Markov Chains}

\author{Yang Xiang \and Kevin McGoff \and Andrew B. Nobel}

\begin{document}

\maketitle

\begin{abstract}
The correspondence between weighted undirected graphs and reversible Markov chains via 
vertex random walks is simple and well known.  Leveraging this correspondence and ideas from the theory of dynamical systems, we study the structural 
discordance of graphs and Markov chains by means of graph joinings.
Informally, a joining of graphs $G$ and $H$ is a  
graph on the product of their vertex sets giving rise to a coupling of their random walks.
Graphs $G$ and $H$ are strongly disjoint if their only joining is the tensor product,
and they are weakly disjoint if the degree function of every joining is equal to the degree 
function of the tensor product. 
We establish close connections between graph joinings, disjointness, and graph factors.
Our first principal result characterizes weak disjointness of graphs in terms of the 
spectral overlap of their Markov transition matrices.  The second 
establishes that two graphs without self loops are strongly disjoint if and only if
they are weakly disjoint and exactly one of the graphs is a tree.  The third shows
that the strong or weak disjointness of graphs is essentially determined by their
vertex and edge sets, without regard to edge weights.  Translating these results
into the language of Markov chains yields new insights into the rigidity 
and structure of reversible couplings of reversible Markov chains.



\end{abstract}

\section{Introduction}

Graphs arise in a variety of statistical problems, where they are used to 
represent assumed or inferred relationships 
between pairs of objects of interest, 
or to describe probabilistic structures such as Markov chains,
hidden Markov models, or Markov random fields. 
While the study and statistical analysis of individual graphs in these settings  
is well-established, a number of problems require the analysis and interpretation 
of a \textit{collection} of two or more graphs, potentially of different sizes and topologies. 
In these problems, it is often of interest to align or compare two given graphs.
The existing literature on graph alignment and comparison is focused primarily 
on the study of graph \textit{similarity}, most often in the form of graph isomorphism,
though less rigid notions of similarity have also been studied. 
In contrast, the investigation of 
graph discordance has received little attention.

In this paper we define and study three related notions of 
discordance between undirected graphs with non-negative edge weights. 
As used here, discordance is distinct from notions of dissimilarity that arise   
when two graphs are not isomorphic or are far apart relative to some distance. 
Dissimilarity of this sort is often evident upon inspection, and may be
of limited practical use: for example, knowing that two graphs have a large edit distance
is usually less useful than knowing that their edit distance is small.   
Informally, discordance captures a lack of common structure.
An overview of our results is given below; precise definitions and formal statements
of results can be found in Sections~\ref{Basic_Definitions}--\ref{Section:MC}. 

Our investigations are based on the elementary correspondence between 
reversible Markov chains and undirected graphs with normalized, 
non-negative edge weights.  This Markov embedding approach to graph analysis has
also been considered in the recent papers
\cite{yi2025alignment} and \cite{hoang2025optimal}.  
Let $G$ and $H$ be undirected graphs with non-negative edge weights
and vertex sets $U$ and $V$, respectively. 
{Define the degree of a vertex $u$ to be the sum of the weights 
of all edges incident to $u$.}
Let $X = X_0, X_1, \ldots \in U$ be the standard random walk on $G$:
from a vertex $u \in U$, the walk moves to an adjacent vertex $u'$ with probability 
equal to the weight of the edge $(u,u')$ divided by the degree of $u$. 
As $G$ is undirected, the walk $X$ is a reversible Markov chain.
Here and throughout the paper all random walks are assumed to be stationary.
Let $Y = Y_0, Y_1, \ldots \in V$ be the standard random walk on $H$, which is
defined in the same manner.

Following \cite{hoang2025optimal} we define a graph joining of $G$ and $H$ to be 
a graph $K$ on $U \times V$ whose associated random walk is a
Markovian coupling of the random walks on $G$ and $H$
(see Section \ref{Basic_Definitions} for details).
The definition ensures, in particular, that 
$G$ and $H$ can be recovered from $K$ through appropriate marginalization.
Let $\mathcal{J}(G,H)$ be the set of all graph joinings of $G$ and $H$.
As illustrated below, graph joinings capture structural 
interactions between $G$ and $H$. 
The set $\mathcal{J}(G,H)$ is non-empty, as the tensor product graph 
$K = G \otimes H$ is always a graph joining.

We define and study three types of graph discordance in terms of graph joinings.  
For each notion we use the general term disjointness, which was introduced by \cite{furstenberg1967disjointness} in
the context of ergodic theory and dynamical systems.
We say that graphs $G$ and $H$ are strongly disjoint if the family
$\mathcal{J}(G,H)$ contains only the tensor product $G \otimes H$, and we say that $G$ and $H$ are
weakly disjoint if for each $K \in \mathcal{J}(G,H)$ 
the degree of $(u,v)$ in $K$ is equal to the degree of $u$ in $G$ times the degree
of $v$ in $H$. A third, weaker, form of disjointness arises from the problem of
finding a joining $K \in \mathcal{J}(G,H)$ minimizing the degree-weighted sum of 
a cost function $c: U \times V \to [0,\infty)$ (see Section~\ref{subsec:OGJ}).  
We say that graphs $G$ and $H$ are $c$-disjoint if $G \otimes H$ is a minimizer (not necessarily unique) of the degree weighted cost. 
It is easy to see that strong disjointness implies weak disjointness, and weak disjointness implies $c$-disjointness for any cost function $c$.  Moreover,
two graphs are weakly disjoint if and only if they are $c$-disjoint 
for every cost function $c$.  Each type of disjointness can
be determined in polynomial time with respect to the sizes of $U$ and $V$.
We present a number of examples, including paths and cycles, that illustrate the 
different types of disjointness and show that they are distinct.  

\subsection{Theoretical Results}
Our theoretical analysis begins by examining the connection between graph 
joinings, disjointness, and graph factors.  
Informally, $H$ is a factor of $G$ if the random walk $Y$ on $H$ can be obtained 
(in the sense of distribution) by applying a fixed
function $f: U \to V$ to the random walk $X$ on $G$.  
Factors are closely related to graph joinings: a graph
$K$ with vertex set $U \times V$ is a graph joining of $G$ and $H$ if and only 
if $G$ and $H$ are factors of $K$ under the usual coordinate projections.
Further analysis reveals several analogies between disjoint graphs 
and coprime numbers, paralleling related results in ergodic theory 
(see \cite{furstenberg1967disjointness,glasner1983minimal,de2005introduction, de2023joinings}).
In particular, graphs $G$ and $H$ are strongly disjoint if, whenever both are factors
of a graph $K$, their tensor product $G \otimes H$ is also a factor of $K$.
Also, if $G$ and $H$ are weakly disjoint, then they have no non-trivial common factors,  
but we note that the converse fails to hold in this setting.

We next turn our attention to characterizations of weak and strong disjointness,
which form the primary results of the paper.
Connected graphs $G$ and $H$ are weakly disjoint if and only if the
transition matrices of their random walks share no eigenvalue other than one (see Theorem~\ref{weak_no_share_eigen}).
Moreover, connected graphs with no self loops 
are strongly disjoint if and only if they are weakly disjoint and exactly one of 
the graphs is a tree (see Theorem~\ref{Characterization_connected_no_self}).  
When $G$ and $H$ are different, strong disjointness is the exception 
and weak disjointness is the rule. Moreover, the weak or strong disjointness of
two graphs is preserved for virtually all graphs (in both topological and measure-theoretic senses) having the same vertex and edge sets but potentially different weight 
functions (see Proposition~\ref{prevalence_strong_weak}). 
Finally, we show that the classes of bipartite and connected 
graphs can be characterized in terms of strong disjointness.

\subsection{Scope and Contributions}

As shown in Section \ref{markovchain},
if $G$ and $H$ are graphs with associated random walks $X$ and $Y$, 
the family $\mathcal{J}(G,H)$ of graph joinings is in one-to-one correspondence with 
the family $\mathcal{J}(X,Y)$ of reversible Markovian couplings of $X$ and $Y$.
As a consequence, the results of this paper may be expressed equivalently in 
terms of graphs or Markov chains.  
We have stressed the former, as the
development and results of the paper are expressed most succinctly in terms 
of graphs.  Nevertheless, our results show that the graphical approach 
taken here provides a useful mathematical framework in which to study 
reversible couplings of reversible Markov chains.

While graph joinings and graph factors have, under different names,
received attention in the Markov chain literature, 
the concepts of strong, weak, and $c$-disjointness are new and have not
previously been studied in the context of Markov chains or graphs.
Moreover, factors of Markov chains have primarily been studied 
via lumpability for fixed chains, rather than the more 
algebraic, multi-chain perspective adopted here.
As such, the results of the paper,
expressed in terms of graphs or Markov chains, are new and of potential
interest to researchers interested in either topic, or in the connections 
between them.

The notions of disjointness studied here may have 
implications for graph inference problems that are 
sensitive to structural incompatibility. 
For example, in transfer learning on graph-structured data, where representations learned on a source graph are applied to a target graph that may differ in topology or feature distribution, a mismatch between the source and the target could undermine generalization and give rise to negative transfer \cite{weiss2016survey}. Negative transfer has been linked to structural discrepancies between the source and the target, even when they are semantically similar. Empirical studies show that even small perturbations to the target structure (e.g., minor edge permutations or injected structural noise) can notably degrade transfer performance \cite{wang2024subgraph}. 

\vskip.2in

\subsection{Organization of the Paper}

The definitions and basic properties of weight joinings and graph joinings 
are given in the next section.
Section~\ref{Disjointness} presents the three types of disjointness between graphs 
and provides examples demonstrating their differences.  
Section~\ref{factors} introduces the concept of graph factors, characterizes graph joinings in terms 
of graph factors, and examines connections between graph factors and disjointness. Section~\ref{Characterizing} is devoted to the 
characterizations of weak 
and strong disjointness.
In Section~\ref{Prevalence_of_Disjointness} we investigate a dichotomy in the
persistence of strong and weak disjointness: if the vertices and edges of two 
graphs are fixed, almost all weight assignments yield the same type of 
disjointness.
Section~\ref{characterize_graph_family} is devoted to the characterization of bipartite and 
connected graphs in terms of strong disjointness. 
In Section~\ref{Section:MC} we formulate and discuss the key results of 
the paper in terms of reversible Markovian couplings of reversible Markov chains. 
Sections~\ref{proof_of_disjointness_def}–\ref{proof_for_markovchain} contain 
proofs of the results in Sections~\ref{Disjointness}–\ref{Section:MC}.
The Appendix~\ref{appendix:proofs} contains proofs of several additional results.

\subsection{Related Work}

The graph comparison problem aims to identify and quantify similarities 
between two graphs based on structural or attribute-level information.
Recently, a number of methods based on optimal transport 
(OT) ideas have been developed and applied to graph alignment tasks.
Graph Optimal Transport (GOT) \cite{petric2019got,maretic2022fgot}
derives covariance matrices from the Laplacians of each graph and then uses optimal transport of the resulting Gaussian distributions to align the graphs. 
The Fused Gromov–Wasserstein (FGW) distance combines a Wasserstein component, 
depending on node features, with a Gromov component, based on pairwise 
structure, in an attempt to match node attributes and relational geometry \cite{vayer2020fused}.
Recent work of \cite{o2022optimal} and \cite{yi2025alignment} 
uses Markov optimal transport to compare and align (possibly directed) 
graphs via their random walks.  The Markov embedding approach adopted here
is inspired by this work, though our focus is on graph discordance rather
than graph alignment. 
The notion of graph joining studied here
is introduced in the recent work of \cite{hoang2025optimal}.

Our results have several points of contact with the substantial  
literature on the spectral analysis of reversible Markov chains. 
The spectral gap of Markov chains has been studied in depth, largely because of its central role in deriving quantitative bounds on mixing times and convergence to stationarity 
(see, e.g., \cite{diaconis1996cutoff,boyd2004fastest,boyd2005symmetry}, as well as the monographs \cite{aldous1989lower,montenegro2006mathematical,levin2017markov}).
The great majority of the existing literature focuses on the spectra 
and convergence behavior of a single Markov chain (e.g., \cite{lawler1988bounds,diaconis1991geometric}), in particular 
the comparison of chains with the same transition matrix but 
different initial distributions. 
Couplings of such chains (e.g., \cite{hunter2009coupling}) are most often
used to analyze convergence to stationarity.
Coupling arguments have also been used to bound the distance between 
the distributions of two Markov chains with distinct transition matrices (see, e.g., \cite{diaconis2023markov}).

The Markov chain Monte Carlo (MCMC) literature is concerned with constructing
a Markov chain whose stationary (invariant) distribution is equal to a
prescribed target distribution, typically a posterior distribution arising from Bayesian inference (see, e.g., \cite{tierney1994markov,robert1999monte}). The chains constructed for MCMC are often reversible, as reversibility provides a convenient sufficient condition for invariance of the target distribution \cite{tierney1998note}. Nevertheless, non-reversible chains have also been studied, largely because they can yield improved sampling efficiency \cite{diaconis2000analysis,ottobre2016markov}. 
Coupling techniques are also used in the MCMC context, but with different objectives. For example couplings of multiple copies of the same chain 
are employed to generate exact samples from the stationary distribution,
as in coupling from the past \cite{propp1996exact},
or to bound the distance between marginal distributions in order 
to assess or compare the performance of MCMC algorithms (\cite{biswas2019estimating}). 

The setting and results of this 
paper are different than those in the literature discussed above.  
Expressed in terms of Markov chains, we study
the family of reversible Markovian couplings of reversible Markov
chains.  All chains are assumed to be stationary, and the chains
being coupled will, in general, not be readily comparable in
terms of their state spaces or transition matrices.
Couplings are used not as a means of obtaining convergence
bounds or simulating samples from a desired distribution, but
rather as a means of uncovering structural relationships between the
chains of interest.

The notion of graph factor studied here is closely related to the concept of lumpability 
in the Markov chain literature, which concerns the aggregation of states of a 
single Markov chain in a manner that preserves the Markov property of the induced process \cite{burke1958markovian}. 
Our definition of factor is formally similar to that of lumpability, 
but its purpose and scope differ substantially: 
rather than focusing on the structural reduction of an individual chain, graph factors are used here 
to study the interactions and structural relationships between two or more Markov chains. 
A more detailed discussion of lumpability and its relation to our graph factor notion is provided in Remark~\ref{rmk_lumpability}.
Our notion of graph disjointness is connected to ideas in ergodic
theory, specifically the definition and study of disjointness of dynamical 
systems introduced in \cite{furstenberg1967disjointness}. A detailed discussion of this connection
is provided in Remark~\ref{rmk:disjoint}.
 

\section{Definition and Properties of Graph Joinings}
\label{Basic_Definitions}

In this section we review the notion of weight function and 
and graph joining from \cite{hoang2025optimal}, and then we 
present some examples and some basic properties of graph joinings.

\subsection{Basic Definitions}
\label{Basic}

\begin{defn}[Weight function and degree function]
A \emph{weight function} on a finite set $U$ is a map $\alpha: U \times U \to \mathbb{R}$ satisfying the following conditions: 
\begin{itemize}
\vskip.06in
\item (Nonnegativity) $\alpha(u,u') \geq 0$ for all $u,u' \in U$;
\vskip.06in
\item (Symmetry) $\alpha(u,u') = \alpha(u',u)$ for all $u,u' \in U$;
\vskip.06in
\item (Normalization) $\sum_{u,u' \in U} \alpha(u,u') = 1$.
\end{itemize} 
\vskip.06in
The \emph{degree function} $p : U \to [0,1]$ of a weight function
$\alpha$ is defined for $u \in U$ by
\[
p(u) = \sum_{u'\in U}\alpha(u,u').
\]
\end{defn}

\vskip.05in

The definition above ensures that $p$ 
satisfies $p(u') = \sum_{u \in U} \alpha(u,u')$,
and also $\sum_{u \in U} p(u) = 1$.  In particular, $p$ is necessarily 
a probability mass function on $U$.

\vskip.05in

\begin{defn}[Weighted undirected graphs]
A \emph{weighted, undirected graph} is described by a pair $G = (U,\alpha)$, 
where $U$ is a finite vertex set and $\alpha$ is a weight function on $U$. 
The edge set of $G$ is $E(G) = \{ (u,u')\in U\times U : \alpha(u,u') > 0 \}$.
The \emph{neighborhood} of a vertex $u$ in $G$ is
$N(u) \coloneqq \{u'\in U: \alpha(u,u') > 0\}$.  
\end{defn}

\vskip.06in

Although we refer to $G = (U, \alpha)$ above as being undirected, our analysis is 
simplified by considering edges as ordered pairs, so that $(u,u')$ is a 
directed edge from $u$ to $u'$ with weight $\alpha(u,u')$.  
Self-loops $(u,u)$ are allowed under the definition, but multi-edges are excluded.
Symmetry of $\alpha$ ensures that the weight assigned to a directed edge
$(u,u')$ is the same as the weight assigned to $(u',u)$. 

There is a natural correspondence between an undirected graph 
$G = (U, \alpha)$ specified as
above and the more standard notion of a weighted graph $\tilde{G}$ 
with undirected edges $\{u,u'\}$.  
Under this correspondence the weight of a self loop $\{u,u\}$ is 
equal to $\alpha(u,u)$ and the weight of an edge $\{u,u'\}$ with
$u' \neq u$ is equal to $\alpha(u,u') + \alpha(u',u)$.  
By definition, the graph $G$ is connected if there is a path between 
any pair of distinct vertices in $U$.  
Symmetry of $\alpha$ ensures that connectivity in $G$ is equivalent to 
connectivity in $\tilde{G}$.
We will regard a given undirected, unweighted graph $\tilde{G}$ as a
weighted graph $G = (U, \alpha)$ such that each nonzero value of
$\alpha(u,u')$ is the same, and we will  
refer to such graphs as \emph{uniformly weighted}.

A vertex $u$ in $G = (U, \alpha)$ has zero degree, i.e., $p(u) = 0$, if and only if 
$u$ has no incident edges.
A graph will be called {\em fully supported} if every vertex has strictly positive degree. In particular, every connected graph is fully supported.

\vskip.05in

\begin{defn}[Weight Joining]
\label{Def:weight_joining}
Let $\alpha$ and $\beta$ be weight functions on finite sets $U$ and $V$ 
with degree functions $p$ and $q$, respectively. 
A weight function $\gamma$ on $U \times V$ with degree function 
\[
r (u,v) = \sum_{(u',v') \in U \times V} \gamma((u,v),(u',v'))
\]
is a \emph{weight joining} of $\alpha$ and $\beta$ if it satisfies the following two conditions.
\vskip.1in
\begin{itemize}
\item[(a)] 
Degree coupling: $r$ is a coupling of $p$ and $q$, i.e., 
for all $u \in U$ and $v \in V$, 
\begin{align*}
            \sum\limits_{v \in V} r(u,v) = p(u) \ \ \mbox{ and } \ \ 
            \sum\limits_{u \in U} r(u,v) = q(v).
\end{align*}
\item[(b)] 
Transition coupling: for all $u,u'\in U$ and $v,v'\in V$
\begin{align*}
    &p(u)\sum \limits_{\tilde{v} \in V} \gamma((u,v), (u', \tilde{v})) 
    = {\alpha(u, u')} \, r(u, v), \; \text{ and }\\[.07in]
    &q(v)\sum \limits_{\tilde{u} \in U}\gamma((u, v), (\tilde{u}, v')) 
    = {\beta(v, v')} \, r(u, v).
\end{align*}
\end{itemize}
Let $\mathcal{J}(\alpha,\beta)$ denote the set of all weight joinings of $\alpha$ and $\beta$.
\end{defn}

\vskip.05in

\begin{defn}[Graph Joining]
A graph $K$ is a \emph{graph joining} of $G=(U,\alpha)$ and $H=(V,\beta)$ if it is of the form
$K = (U \times V, \gamma)$ where $\gamma$ is a weight joining
of $\alpha$ and $\beta$.
Let $\mathcal{J}(G,H)$ be the family of all graph joinings of $G$ and $H$.
\end{defn}

\vskip.05in

For any given graphs $G = (U,\alpha)$ and $H = (V,\beta)$ 
it is straightforward to verify that the product weight function, 
defined for all vertices $u,u'\in U$ and $v,v'\in V$ by
\begin{equation}
\label{eqn:prodjoin}
(\alpha \otimes \beta) ((u,v),(u',v')) = \alpha(u,u')\beta(v,v'),
\end{equation}
is a weight joining of $\alpha$ and $\beta$.  Thus
the \emph{tensor product graph}  
$G \otimes H := (U \times V, \alpha \otimes \beta)$ 
is an element of $\mathcal{J}(G,H)$, which is therefore non-empty.
The tensor product plays a critical role in the study of disjointness
in the next section.

Let $\gamma$ be a weight joining of $\alpha$ and $\beta$.  The definition ensures that 
$\alpha$ and $\beta$ can be recovered from $\gamma$ through marginalization, i.e.,
\[
\sum_{v,v' \in V} \gamma((u,v), (u',v')) = \alpha(u,u')
\ \ \mbox{and} \ \ 
\sum_{u,u' \in U} \gamma((u,v), (u',v')) = \beta(v,v').
\]
While they are necessary, it is important to note that these conditions
are not sufficient to guarantee that $\gamma$ is a weight joining. 
In the case where $r(u,v) > 0$, the transition coupling conditions can be 
rewritten in the form
\begin{align*}
    &\frac{1}{r(u,v)} \sum \limits_{\tilde{v} \in V} \gamma((u,v), (u', \tilde{v})) \, = \, \frac{1}{p(u)}\alpha(u,u'),\\
    &\frac{1}{r(u,v)}\sum \limits_{\tilde{u} \in U} \gamma((u,v), (\tilde{u}, v')) \, = \, \frac{1}{q(v)}\beta(v,v').
\end{align*}
The first equation indicates that the fraction of outgoing weight from 
$(u,v)$ to any pair $(u',\tilde{v})$ with $\tilde{v} \in V$ is equal to the fraction 
of outgoing weight from $u$ to $u'$.  The second equation may be interpreted
similarly.
As discussed in \cite{hoang2025optimal} and Section \ref{Section:MC} below,
the definition of graph joining is equivalent to the requirement that 
the random walk on $K \in \mathcal{J}(G,H)$ is a Markovian coupling of the
random walks on $G$ and $H$. 

While it is needed for general undirected graphs, 
the degree coupling condition 
is redundant when both $G$ and $H$ are connected (see \cite{hoang2025optimal}).  A proof of
the following proposition is given in Appendix~\ref{appendix:proofs} for completeness.

\begin{restatable}{prop}{connecteddropcoupling}
    Let $G=(U,\alpha)$ and $H=(V,\beta)$ be graphs with degree functions 
    $p$ and $q$, respectively. 
    If $G$ and $H$ are connected, then any weight function $\gamma$ satisfying
    condition (b) of Definition~\ref{Def:weight_joining} also satisfies
    condition (a).
    \label{connected_drop_coupling}
\end{restatable}

\subsection{Examples of Graph Joinings}

We adopt the following conventions for figures containing graphs.
An undirected edge between 
two nodes $u$ and $u'$ represents the directed edges $(u,u')$ and $(u',u)$;
the number adjacent to the edge indicates the common value of $\alpha(u,u')$ 
and $\alpha(u',u)$. Component graphs are displayed marginally, 
with joinings displayed to the upper right of the component graphs.

\vskip.1in

Figure~\ref{graph_joining_example} (A) and (B) show two different joinings between identical graphs $G$ and $H$. 
Figure~\ref{graph_joining_example} (C) illustrates an example where there is 
only one possible joining of $G$ and $H$, with constant edge weight $1/12$.  
This example is discussed in Section \ref{example_of_disjointness} below.
In examples (A) - (C) the marginal graphs and their joinings are uniformly weighted;
in general, a joining of uniformly weighted graphs need not be uniformly weighted. 
Figure~\ref{graph_joining_example} (D) illustrates a more complicated example, 
which is discussed further in Proposition~\ref{no_factor_not_weak}.
Note that in this example the self-loops of the graph $G$ are reflected in 
the joining as a self-loop and a cycle.

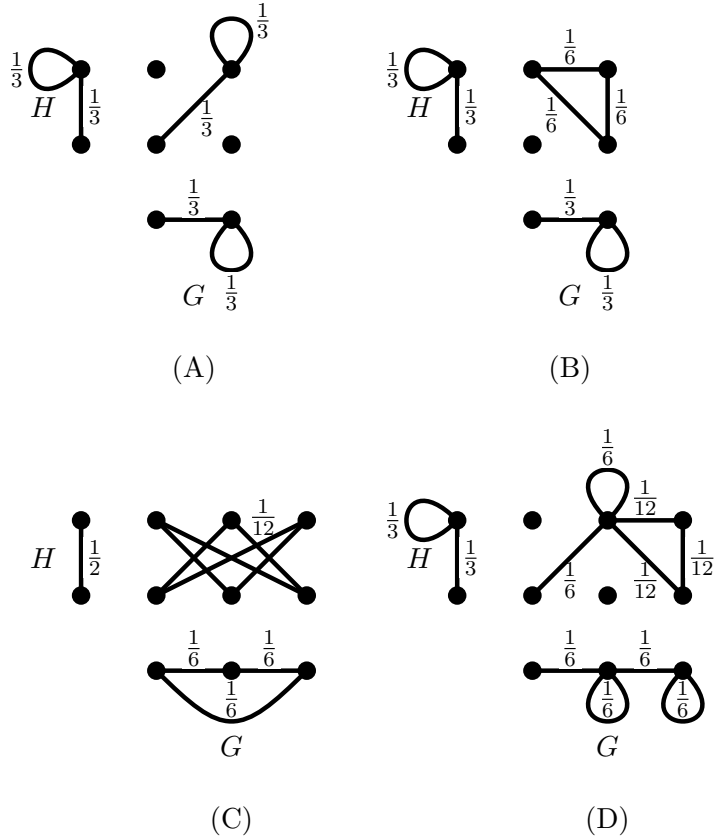
\begin{figure}[h!]
    \centering
\begin{tikzpicture}
    \draw[line width=0.6mm] (0,-1) -- (1,-1) node[midway, above, fill=white, inner sep=1pt] {$\tfrac{1}{3}$}; 
    \draw[line width=0.6mm] (-1,0) -- (-1,1) node[midway, right, fill=white, inner sep=1pt] {$\tfrac{1}{3}$}; 
    \draw[line width=0.6mm] (1,-1) .. controls ++ (0.3*\dist,-0.3*\dist) and ++ (-0.3*\dist,-0.3*\dist) .. (1,-1) node[midway, below, fill=white, inner sep=1pt] {$\tfrac{1}{3}$};
    \draw[line width=0.6mm] (-1,1) .. controls ++ (-0.3*\dist,-0.3*\dist) and ++ (-0.3*\dist,0.3*\dist) .. (-1,1) node[midway, left, fill=white, inner sep=1pt] {$\tfrac{1}{3}$}; 

    \draw[line width=0.6mm] (0,0) -- (1,1) node[pos=0.35, right=4pt, fill=none, inner sep=1pt] {$\tfrac{1}{3}$}; 
    \draw[line width=0.6mm] (1,1) .. controls ++ (-0.3*\dist,0.3*\dist) and ++ (0.3*\dist,0.3*\dist) .. (1,1) node[pos=0.6, right=3pt, fill=none, inner sep=1pt] {$\tfrac{1}{3}$};

    \foreach \x in {0,1} {
        \foreach \y in {-1} {
            \node at (\x,\y) [circle,fill=black,inner sep=0pt,minimum size=7pt] {};
        }
    }
    \foreach \x in {0, 1} {
        \foreach \y in {0, 1} {
            \node at (\x,\y) [circle,fill=black,inner sep=0pt,minimum size=7pt] {};
        }
    }
    \foreach \y in {0, 1} {
        \foreach \x in {-1} {
            \node at (\x,\y) [circle,fill=black,inner sep=0pt,minimum size=7pt] {};
        }
    }
 
    \node at (0.5,-2) {$G$};
    \node at (-1.5,0.5) {$H$};

    \draw[line width=0.6mm] (5,-1) -- (6,-1) node[midway, above, fill=white, inner sep=1pt] {$\tfrac{1}{3}$}; 
    \draw[line width=0.6mm] (4,0) -- (4,1) node[midway, right, fill=white, inner sep=1pt] {$\tfrac{1}{3}$}; 
    \draw[line width=0.6mm] (6,-1) .. controls ++ (0.3*\dist,-0.3*\dist) and ++ (-0.3*\dist,-0.3*\dist) .. (6,-1) node[midway, below, fill=white, inner sep=1pt] {$\tfrac{1}{3}$};
    \draw[line width=0.6mm] (4,1) .. controls ++ (-0.3*\dist,-0.3*\dist) and ++ (-0.3*\dist,0.3*\dist) .. (4,1) node[midway, left, fill=white, inner sep=1pt] {$\tfrac{1}{3}$};

    \draw[line width=0.6mm] (5,1) -- (6,0) node[pos = 0.6, left = 4.7pt, fill=none, inner sep=1pt] {$\tfrac{1}{6}$}; 
    \draw[line width=0.6mm] (6,0) -- (6,1) node[midway, right, fill=white, inner sep=1pt] {$\tfrac{1}{6}$}; 
    \draw[line width=0.6mm] (5,1) -- (6,1) node[midway, above, fill=white, inner sep=1pt] {$\tfrac{1}{6}$}; 

    \foreach \x in {5,6} {
        \foreach \y in {-1} {
            \node at (\x,\y) [circle,fill=black,inner sep=0pt,minimum size=7pt] {};
        }
    }
    \foreach \x in {5, 6} {
        \foreach \y in {0, 1} {
            \node at (\x,\y) [circle,fill=black,inner sep=0pt,minimum size=7pt] {};
        }
    }
    \foreach \y in {0, 1} {
        \foreach \x in {4} {
            \node at (\x,\y) [circle,fill=black,inner sep=0pt,minimum size=7pt] {};
        }
    }
 
    \node at (5.5,-2) {$G$};
    \node at (3.5,0.5) {$H$};    

    \draw[line width=0.6mm] (0,-7) -- (1,-7) node[midway, above, fill=white, inner sep=1pt] {$\tfrac{1}{6}$}; 
    \draw[line width=0.6mm] (2,-7) -- (1,-7) node[midway, above, fill=white, inner sep=1pt] {$\tfrac{1}{6}$}; 
    \draw[line width=0.6mm] (-1,-6) -- (-1,-5) node[midway, right, fill=white, inner sep=1pt] {$\tfrac{1}{2}$}; 
    \draw[line width=0.6mm] (0,-7) .. controls ++ (0.3*\dist,-0.3*\dist) and ++ (-0.3*\dist,-0.3*\dist) .. (2,-7) node[midway, above, fill=none, inner sep=1pt] {$\tfrac{1}{6}$};

    \draw[line width=0.6mm] (0,-5) -- (1,-6); 
    \draw[line width=0.6mm] (1,-6) -- (2,-5); 
    \draw[line width=0.6mm] (2,-5) -- (0,-6) node[midway, anchor = south, xshift = 12pt, yshift = 7pt, fill=none, inner sep=1pt] {$\tfrac{1}{12}$}; 
    \draw[line width=0.6mm] (0,-6) -- (1,-5);
    \draw[line width=0.6mm] (1,-5) -- (2,-6);
    \draw[line width=0.6mm] (2,-6) -- (0,-5);

    \foreach \x in {0,1,2} {
        \foreach \y in {-7} {
            \node at (\x,\y) [circle,fill=black,inner sep=0pt,minimum size=7pt] {};
        }
    }
    \foreach \x in {0, 1, 2} {
        \foreach \y in {-5, -6} {
            \node at (\x,\y) [circle,fill=black,inner sep=0pt,minimum size=7pt] {};
        }
    }
    \foreach \y in {-5, -6} {
        \foreach \x in {-1} {
            \node at (\x,\y) [circle,fill=black,inner sep=0pt,minimum size=7pt] {};
        }
    }

    \draw[line width=0.6mm] (5,-7) -- (6,-7)node[midway, above, fill=white, inner sep=1pt] {$\tfrac{1}{6}$}; 
    \draw[line width=0.6mm] (7,-7) -- (6,-7)node[midway, above, fill=white, inner sep=1pt] {$\tfrac{1}{6}$}; 
    \draw[line width=0.6mm] (4,-6) -- (4,-5)node[midway, right, fill=white, inner sep=1pt] {$\tfrac{1}{3}$}; 
    \draw[line width=0.6mm] (7,-7) .. controls ++ (0.3*\dist,-0.3*\dist) and ++ (-0.3*\dist,-0.3*\dist) .. (7,-7)node[midway, above, fill=none, inner sep=1pt] {$\tfrac{1}{6}$};
    \draw[line width=0.6mm] (6,-7) .. controls ++ (0.3*\dist,-0.3*\dist) and ++ (-0.3*\dist,-0.3*\dist) .. (6,-7)node[midway, above, fill=none, inner sep=1pt] {$\tfrac{1}{6}$};
    \draw[line width=0.6mm] (4,-5) .. controls ++ (-0.3*\dist,-0.3*\dist) and ++ (-0.3*\dist,0.3*\dist) .. (4,-5)node[midway, left, fill=white, inner sep=1pt] {$\tfrac{1}{3}$};

    \draw[line width=0.6mm] (5,-6) -- (6,-5) node[midway, below, fill=none, inner sep=1pt] {$\tfrac{1}{6}$};
    \draw[line width=0.6mm] (6,-5) -- (7,-6) node[midway, below, fill=none, inner sep=1pt] {$\tfrac{1}{12}$};
    \draw[line width=0.6mm] (7,-6) -- (7,-5) node[midway, right, fill=white, inner sep=1pt] {$\tfrac{1}{12}$};
    \draw[line width=0.6mm] (7,-5) -- (6,-5) node[midway, above, fill=white, inner sep=1pt] {$\tfrac{1}{12}$};
    \draw[line width=0.6mm] (6,-5) .. controls ++ (-0.3*\dist,0.3*\dist) and ++ (0.3*\dist,0.3*\dist) .. (6,-5) node[midway, above, fill=white, inner sep=1pt] {$\tfrac{1}{6}$};

    \foreach \x in {5,6,7} {
        \foreach \y in {-7} {
            \node at (\x,\y) [circle,fill=black,inner sep=0pt,minimum size=7pt] {};
        }
    }
    \foreach \x in {5, 6, 7} {
        \foreach \y in {-6, -5} {
            \node at (\x,\y) [circle,fill=black,inner sep=0pt,minimum size=7pt] {};
        }
    }
    \foreach \y in {-6, -5} {
        \foreach \x in {4} {
            \node at (\x,\y) [circle,fill=black,inner sep=0pt,minimum size=7pt] {};
        }
    }
 

    \node at (1,-8) {$G$};
    \node at (-1.5,-5.5) {$H$};    
    \node at (6,-8) {$G$};
    \node at (3.5,-5.5) {$H$};    
    \node at (0.5,-3) {\text{(A)}};
    \node at (5.5,-3) {\text{(B)}};
    \node at (1,-9) {\text{(C)}};
    \node at (6,-9) {\text{(D)}};
\end{tikzpicture}

\caption{ 
(A) and (B) Two distinct graph joinings of the same pair of uniformly weighted graphs (with self-loops); (C) The only graph joining between this pair of graphs is the product joining; (D) An example of a weighted joining between two graphs with self-loops.}
    \label{graph_joining_example}
\end{figure}

\vskip.1in

\subsection{Isomorphic Graphs} 
\label{isomorphic_graphs}

Recall that two graphs $G = (U,\alpha)$ and $H = (V,\beta)$ are isomorphic,
written $G \cong H$, if there is a bijection $f: U \to V$ that preserves 
edge weights in the sense that $\alpha(u,u') = \beta(f(u),f(u'))$
for all $u, u' \in U$.  In particular, $p(u) = q(f(u))$ for all $u \in U$.
In this case there is a \emph{bijective} graph joining 
$K \in \mathcal{J}(G, H)$ with weight function $\gamma$ satisfying
    \begin{align*}
    \gamma((u,v),(u',v')) \, = \, 
    \begin{cases}
        \, \alpha(u,u') &\text{ if } v = f(u) \text{ and } v' = f(u'),\\
        \, 0 &\text{otherwise},
    \end{cases}
\end{align*}
and corresponding degree function $r(u,v) = p(u) \mathbbm{1}(f(u)=v)$.
Indeed, to check that $\gamma$ is a graph joining, we note that the 
first transition coupling condition follows from
\begin{align*}
\sum\limits_{v'\in V} \gamma((u, v), (u', v')) 
\, = \, 
\mathbbm{1}(f(u) = v) \, \alpha(u,u') 
\, = \, \frac{\alpha(u,u')}{p(u)} \, r(u,v),
\end{align*}
and the second follows in a similar manner, making use of the identities 
$\alpha(u,u') = \beta(f(u),f(u'))$ and $p(u) = q(f(u))$.  Verification of
the degree condition is similarly straightforward.

\vskip.2in

\subsection{Vertex and Edge Preservation}

When $\mathcal{X}$ is a finite set, we denote the support of a function $f: \mathcal{X} \to [0,\infty)$ by $\supp(f) = \{x \in \mathcal{X}\mid f(x)>0\}$.

\begin{restatable}{lemma}
{vertexedgepreservation}
\label{vertex_edge_preservation}
Let $G=(U,\alpha)$ and $H=(V,\beta)$ be graphs with degree functions $p$ and $q$. 
Then any graph joining $K = (U \times V, \gamma) \in \mathcal{J}(G,H)$ 
with degree function $r$ satisfies the following properties:
\begin{align*}
    & \supp(r) \, \subseteq \, \supp(p \otimes q) 
      \, = \, \{ (u,v) \in U \times V \mid p(u)>0, q(v)>0 \}, \; \text{ and } \\
    & \supp(\gamma) \, \subseteq \, \supp(\alpha\otimes \beta)
    \, = \, \{((u,v),(u',v'))\in (U\times V)^2\mid (u,u')\in E(G),(v,v')\in E(H)\}.
\end{align*}
We refer to the first property as \emph{vertex preservation}, and the second as \emph{edge preservation}.
\end{restatable}

Vertex preservation says that if vertex $(u,v)$ has positive degree 
in $K$, then $u$ and $v$ must have positive degree in $G$ and $H$, respectively. 
Edge preservation says that if $((u,v),(u',v'))$ is an edge in $K$, 
then $(u,u')$ and $(v,v')$ must be edges in $G$ and $H$, respectively.
In other words, a joining of two graphs cannot create a joint vertex
or a joint edge unless the constituent vertices or edges are present
in the original graphs. Note, however, that $(u,v)$ may have zero degree in
$K$ even when $u$ and $v$ have positive degree in $G$ and $H$,
and similarly for edges (see Figure \ref{graph_joining_example}). The proof 
of the lemma is provided in Appendix~\ref{proof_vertex_edge_preservation}.

\vskip.2in

\subsection{Optimal Graph Joinings}
\label{subsec:OGJ}

In previous work \cite{hoang2025optimal}, we consider a constrained optimal transport problem based 
on graph joinings and study its connection with the graph isomorphism problem.
Let $G=(U,\alpha)$ and $H=(V,\beta)$ be graphs, and let 
$c: U \times V \to [0,\infty)$ be
a cost function relating their vertex sets.
The Optimal Graph Joining (OGJ) problem finds the minimum 
degree-weighted average cost,
\begin{align}
\label{OGJ_distance}
\rho(G,H) 
\, = \, 
\min_{K \in \mathcal{J}(G,H)} 
\sum\limits_{(u,v) \in U \times V} 
c(u,v) \, r_{\scaleto{K}{4pt}}(u,v) ,
\end{align}
where $r_{\scaleto{K}{4pt}}$ is the weight function of $K$. 
As $\mathcal{J}(G,H)$ is non-empty, 
a standard compactness argument shows that the OGJ problem has at least one solution.
In \cite{hoang2025optimal} we consider families $\mathcal{G}$ of 
graphs $G = (U,\alpha)$ with vertex labels, 
and a binary cost function $c(u,v)$ that is zero 
if the labels of $u, v$ agree and one otherwise.  
We identify conditions on the family $\mathcal{G}$ and the 
labeling scheme under which $\rho(G,H) = 0$ implies that (i) $G$
and $H$ are isomorphic, and (ii) the extreme points of 
the set of optimal graph joinings coincide with the bijective graph joinings 
arising from the isomorphisms of $G$ and $H$. 

Loosely speaking, disjointness can be viewed as the \textit{opposite} of 
isomorphism, and as such the results of this paper are complementary 
to those in \cite{hoang2025optimal}.  Indeed, it follows from
Theorem~\ref{weak_no_share_eigen_general} below that for many pairs $G$, $H$ of
{\em non-isomorphic} graphs, the minimum in equation \eqref{OGJ_distance} is
achieved by the tensor product $G \otimes H$.  In this case, 
$\rho(G,H)$ does not capture detailed information about interactions between
$G$ and $H$.

\vskip.25in

\section{Disjointness}
\label{Disjointness}

The family $\mathcal{J}(G,H)$ of joinings between graphs $G=(U,\alpha)$ 
and $H=(V,\beta)$ provides a means to quantify and study similarities and 
compatibilities between the graphs.  A graph joining can be viewed as a
constrained form of graph coupling.  In this sense
the family $\mathcal{J}(G,H)$ is analogous to the family of couplings
of two measures that forms the feasible set of the optimal transport problem.
Unlike \cite{hoang2025optimal} and the usual study of optimal transport, our
interest here is in the structure of the family $\mathcal{J}(G,H)$, without
regard to any particular cost function.

As noted above, there is always one joining of two graphs $G$ and $H$, 
namely the tensor product $G \otimes H = (U \times V, \alpha \otimes \beta)$, 
under which $G$ and $H$ do not interact.   
In process terms, the random walk on $G \otimes H$ is the independent
coupling of the individual random walks on $G$ and $H$.
Other joinings, if they exist, capture non-trivial interactions between $G$ and $H$.
For example, if $G$ and $H$ are isomorphic
then there is a bijective joining in $\mathcal{J}(G,H)$ whose degree function
is concentrated on the graph of an isomorphism between $G$ and $H$.
Of interest here is the definition and analysis of structural
discordance between graphs.  Beginning with the case in which the tensor 
product $G \otimes H$ is the \textit{only} joining of $G$ and $H$, we
explore several forms of discordance under the heading of disjointness.

\subsection{Definitions of graph disjointness}

\newcommand\shortperp{\scalebox{1}[.6]{$\perp$}}

\begin{defn}[Graph disjointness]
Let $G=(U,\alpha)$ and $H=(V,\beta)$ be weighted undirected graphs with degree
functions $p$ and $q$ respectively.  
Let $r_{\scaleto{K}{4pt}}$ denote the degree function of a graph 
joining $K \in \mathcal{J}(G,H)$, and
let $(p \otimes q)(u,v) = p(u) q(v)$ denote
the usual tensor product of $p$ and $q$.
\begin{enumerate}

\vskip.1in

\item
$G$ and $H$ are {\em strongly disjoint} if 
$\mathcal{J}(G,H) = \{G \otimes H \}$.

\vskip.1in

\item 
$G$ and $H$ are {\em weakly disjoint} if every $K \in \mathcal{J}(G,H)$ 
has degree function $r_{\scaleto{K}{4pt}} = p \otimes q$. 

\vskip.1in

\item 
$G$ and $H$ are {\em $c$-disjoint} with respect to a cost function
$c: U \times V \rightarrow [0,\infty)$ if 
\[
\min_{K \in \mathcal{J}(G,H)} 
\sum\limits_{(u,v) \in U \times V} 
c(u,v) \, r_{\scaleto{K}{4pt}}(u,v) 
\, = \, 
\sum\limits_{(u,v)\in U  \times V} c(u,v) \, p(u) \, q(v).
\]

\end{enumerate}
\end{defn}

Note that $G$ and $H$ are $c$-disjoint if 
the product joining $G \otimes H$ is optimal for the
OGJ problem with cost $c$.
It is easy to see that strong disjointness implies weak disjointness, 
and that weak disjointness implies $c$-disjointness for every cost function $c$. 
In general these implications are strict: see Proposition~\ref{circle_property} and Example~\ref{example_weak_c} below.
The following proposition is proved in Section \ref{proof_of_weak_and_c}.

\begin{restatable}{prop}{ctoweak}
\label{c_to_weak}
Graphs $G$ and $H$ are weakly disjoint if and only if they are $c$-disjoint for every cost function $c$.
\end{restatable}

\begin{rmk}
\label{rmk:disjoint}
The term graph joining and the definitions of disjointness above are inspired 
by ideas first introduced in the context of ergodic theory and dynamical systems
by \cite{furstenberg1967disjointness}.  We discuss this briefly
in the language of probability.  Let $X$ and $Y$ be two general 
stationary processes (not necessarily Markov).
Furstenberg defined a joining of $X$ and $Y$ to be any {\em stationary} 
coupling $Z$ of $X$ and $Y$, and he defined $X$ and $Y$ to be disjoint if 
their only joining is the independent coupling (which is always stationary).
Optimal joinings with respect to a predefined cost function 
are the basis of the $\bar{d}$-distance and $\bar{\rho}$-distance 
between processes (see \cite{ornstein1970bernoulli,gray1975generalization}). 
In so far as stochastic processes are concerned, the present work is
restricted to finite state reversible Markov chains.  
As a graph joining corresponds to a stationary Markovian coupling of random
walks, it is a joining in Furstenberg's sense.  
However, the converse does not hold, as a stationary coupling of 
reversible Markov chains need not be reversible or Markovian.
Our notion of strong disjointness is most closely aligned
with Furstenberg's notion of disjointness, as both assert that the 
only possible joining of two given objects is the trivial one.  
Weak- and c-disjointness are natural notions to consider in the context of graphs (or Markov chains), but we are not aware of analogous
definitions in the ergodic theory literature.
\end{rmk}

Each type of disjointness 
introduced above can be characterized by a rank equation involving a suitably 
defined constraint matrix.  As such, disjointness can be verified in polynomial time, which we state formally in the following lemma. The proof appears in Section \ref{proof_of_polynomial_time}.

\begin{restatable}{lemma}{polynomialtime}
\label{polynomial_time}
Given two graphs $G = (U,\alpha)$ and $H = (V,\beta)$ and a cost function $c : U \times V \to [0,\infty)$, one can determine whether 
they are strongly disjoint, weakly disjoint, or $c$-disjoint in time polynomial 
in $|U|$ and $|V|$.
\end{restatable}

\vskip.1in

\subsection{Examples of Disjointness}
\label{example_of_disjointness}


Here and in what follows we make use of standard graph theory notation for 
uniformly weighted cycles and paths (see Section 1.3 in \cite{diestel2025graph}).

\begin{defn}[Cycles, Paths, Bipartite Graphs]
\label{def:circlepath}
\leavevmode 
\begin{enumerate}
\vskip.05in

\item 
For $k \geq 3$, $C_k$ denotes the cycle with $k$ vertices and uniform edge weights $1 /(2k)$.

\vskip.05in

\item
For $k \geq 2$, $P_k$ denotes the path with $k$ vertices and uniform edge weights $1/(2(k-1))$.

\vskip.05in

\item
For $k, l \geq 1$ let $K_{k,l}$ be the (uniformly weighted) complete bipartite graph having vertex groups of
size $k$ and $l$, with all possible edges between groups.

\end{enumerate}
\end{defn}

\vskip.1in

The following propositions demonstrate that the three notions of disjointness 
are non-trivial and distinct from each other.
The proofs are given in Section \ref{proof_of_disjointness_def}.

\begin{restatable}{prop}{linecircle}
\label{line_circle}
Fix integers $m\geq 3$ and $n\geq 2$. The cycle $C_m$ and the path $P_n$ are strongly disjoint if and only if 
$m$ is odd and $\gcd(m,n-1)=1$. In particular, there exists a pair of strongly disjoint graphs.
\end{restatable}




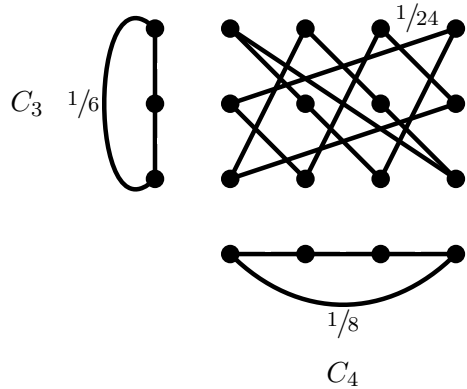
\begin{figure}[ht]
    \centering
    \begin{tikzpicture}
    \draw[line width=0.6mm] (0,0) -- (1,0) node[midway, above, fill=white, inner sep=1pt]{}; 
    \draw[line width=0.6mm] (1,0) -- (2,0) node[midway, above, fill=white, inner sep=1pt]{}; 
    \draw[line width=0.6mm] (2,0) -- (3,0) node[midway, above, fill=white, inner sep=1pt]{}; 
    \draw[line width=0.6mm] (-1,1) -- (-1,2) node[midway, right, fill=white, inner sep=1pt]{};
    \draw[line width=0.6mm] (-1,3) -- (-1,2) node[midway, right, fill=white, inner sep=1pt]{};
    \draw[line width=0.6mm] (0,0) .. controls ++ (0.3*\dist,-0.3*\dist) and ++ (-0.3*\dist,-0.3*\dist) .. (3,0) node[midway, below, fill=white, inner sep=1pt] {$\sfrac{1}{8}$};
    \draw[line width=0.6mm] (-1,1) .. controls ++ (-0.3*\dist,-0.3*\dist) and ++ (-0.3*\dist,0.3*\dist) .. (-1,3) node[midway, left, fill=white, inner sep=1pt]{$\sfrac{1}{6}$};



    \draw[line width=0.6mm] (0,2) -- (3,3) node[midway, anchor = south, xshift = 28pt, yshift = 11pt, fill=none, inner sep=1pt] {$\sfrac{1}{24}$}; 
    \draw[line width=0.6mm] (3,3) -- (2,1);
    \draw[line width=0.6mm] (2,1) -- (1,2);
    \draw[line width=0.6mm] (1,2) -- (0,3);
    \draw[line width=0.6mm] (0,3) -- (3,1);
    \draw[line width=0.6mm] (3,1) -- (2,2);
    \draw[line width=0.6mm] (2,2) -- (1,3);
    \draw[line width=0.6mm] (1,3) -- (0,1);
    \draw[line width=0.6mm] (0,1) -- (3,2);
    \draw[line width=0.6mm] (3,2) -- (2,3);
    \draw[line width=0.6mm] (2,3) -- (1,1);
    \draw[line width=0.6mm] (1,1) -- (0,2);

    \foreach \x in {0,1,2,3} {
        \foreach \y in {0} {
            \node at (\x,\y) [circle,fill=black,inner sep=0pt,minimum size=7pt] {};
        }
    }
    \foreach \x in {0, 1, 2,3} {
        \foreach \y in {1, 2,3} {
            \node at (\x,\y) [circle,fill=black,inner sep=0pt,minimum size=7pt] {};
        }
    }
    \foreach \y in {1, 2,3} {
        \foreach \x in {-1} {
            \node at (\x,\y) [circle,fill=black,inner sep=0pt,minimum size=7pt] {};
        }
    }
 
    \node at (1.5,-1.6) {$C_4$};
    \node at (-2.7,2) {$C_3$};    
\end{tikzpicture}
\caption{
A graph joining of $C_3$ and $C_4$ that is not the tensor product $C_3\otimes C_4$.}
    \label{example_not_strong}
\end{figure}

\begin{restatable}{prop}{circleweakcoprime}
\label{circle_property}
Let $m, n \geq 3$. The cycles $C_m$ and $C_n$ are never strongly disjoint.  Moreover, 
$C_m$ and $C_n$ are weakly disjoint if and only if $\gcd(m,n)=1$. In particular, there exists a pair of graphs that is weakly disjoint but not strongly disjoint.
\end{restatable}

\vskip.1in

Proposition \ref{circle_property} demonstrates that strong and weak disjointness    
are distinct, and that unequal graphs need not be weakly disjoint.
Figure~\ref{graph_joining_example} (D) provides another example of the latter
phenomenon: the graphs $G$ and $H$ in the figure are fully supported, so their degree functions $p$ and $q$ are positive, but the illustrated joining has vertices $(u,v)$ for which 
$0 = r(u,v) \neq (p \otimes q) (u,v) = p(u) q(v) > 0$.  The next example illustrates
that weak and c-disjointness are distinct.

\begin{restatable}{example}{weakc}
\label{example_weak_c} 
The path $P_2$ with vertex set $\{u_0,u_1\}$ and the complete bipartite graph $K_{2,2}$ 
with vertex sets $\{v_0,v_1\}$ and $\{v_2,v_3\}$
are not weakly disjoint, but they are $c$-disjoint with respect to the cost function 
$c(u_i,v_j) = i + j \pmod{2}$.
\end{restatable}


\vskip.2in

\section{Graph Factors}
\label{factors}
In this section we define and explore the basic properties of 
graph factors.  Graph factors provide an alternative
characterization of graph joinings and 
yield insights into weak and strong disjointness.

\begin{defn}
\label{graph_factor}
Let $G=(U,\alpha)$ and $H=(V,\beta)$ be graphs with degree functions
$p$ and $q$, respectively. We say that $H$ is a factor of $G$,
written $H \0 | \0 G$, if there is a {surjective} 
map $\phi: U \rightarrow V$ such that 
\begin{enumerate}

\vskip.12in

\item[(i)]
$q(v) \, = \, \sum\limits_{u \in \phi^{-1}(v)}p(u)$ for all $v \in V$, and 

\vskip.12in

\item[(ii)]
$q(v) \sum\limits_{ u^{\prime} \in \phi^{-1}\left(v^{\prime}\right)} 
\alpha\left(u, u^{\prime}\right) \, = \, p(u) \0 \beta\left(v, v^{\prime}\right)$
for all $v, v^{\prime} \in V$ and 
all $u \in \phi^{-1}(v)$.

\vskip.08in

\end{enumerate}
A function $\phi$ satisfying (i) and (ii) will be called a factor map.
When we wish to emphasize the role of $\phi$, we will 
write $H \0 | \0 G$ under $\phi$.
\end{defn} 

\vskip.1in

Every map $\phi: U \to V$ induces a partition 
$\{ U(v) = \phi^{-1}(v) : v \in V \}$ of the vertices of $G$.
Condition (i) says that $q$ is the push-forward of $p$ under $\phi$.  
Condition (ii) says that for each vertex $u \in U(v)$, the relative 
weight from $u$ to the cell $U(v')$ is equal to the relative weight 
from $v$ to $v'$.  
For connected graphs, condition (i) follows readily from condition (ii) (see Proposition \ref{factor_degree}).
A similar definition of factor for directed, connected graphs was considered in \cite{yi2025alignment}. The definition guarantees that if $\phi$ is applied
to each state in the random walk $X$ on $G$, then the resulting process is equal
in distribution to the random walk $Y$ on $H$.
Connections between graph factors and lumpability of Markov chains are
discussed below, see Remark~\ref{rmk_lumpability}.

The definition of a graph factor considered here is different from 
that most commonly used in the graph theory literature, in which
factor of a graph $G$ is a spanning subgraph, that is, 
a subgraph containing all the vertices of $G$ but 
only a subset of its edges, see, e.g., \cite{tutte1952factors,plummer2007graph}. 
Different types of factors arise from different edge constraints, 
typically involving vertex degrees or the shape of the spanning subgraph. 


\vskip.05in

As studied here, graph factors provide an alternative characterization of
graph joinings. 
Let $\pi_1(u,v) = u$ and $\pi_2(u,v) = v$ be the usual coordinate projections
on $U \times V$.  

\begin{restatable}{prop}{characterizegraphjoining}
A graph $K = (U \times V, \gamma)$ is a graph joining of $G = (U,\alpha)$ and 
$H = (V,\beta)$ if and only if $G \0 | \0 K$ under $\pi_1$ and 
$H \0 | \0 K$ under $\pi_2$.
\label{characterize_graph_joining}
\end{restatable}


Recall that $G \otimes H$ 
denotes the tensor product of $G$ and $H$, and that
$G \cong H$ denotes the fact that $G$ and $H$ are isomorphic.
The following elementary proposition is proved in Section~\ref{proof_of_factors}.

\begin{restatable}{prop}{propfactprop}
\label{prop:factprop}
Let $G$, $H$, and $K$ be graphs.
\begin{enumerate}
\vskip.04in
\item 
$G \0 | \0 (G \otimes H)$ and $H \0 | \0 (G \otimes H)$.
\vskip.04in
\item 
If $G \0 | \0 H$ and $H \0 | \0 K$,  then $G \0 | \0 K$. 
\vskip.04in
\item 
If $G \0 | \0 H$ and $H \0 | \0 G$ then $G \cong H$. 

\end{enumerate}
\end{restatable}

\vskip.1in

\subsection{Factors and Disjointness}

We next explore the relationship between factors and 
disjointness.  
Proofs of the following propositions can be found in Section \ref{proof_of_factors}. 
Markov chain analogs of these results are discussed in Section \ref{markovchain_factor}.

\begin{restatable}{prop}{furstenbergresult}
\label{furstenberg}
Suppose that $G \0 | \0 K$ and $H \0 | \0 L$. If $K$ and $L$ are strongly disjoint, 
then $G$ and $H$ are strongly disjoint.
If $K$ and $L$ are weakly disjoint, then $G$ and $H$ are weakly disjoint.
\end{restatable}

As a corollary of Proposition~\ref{furstenberg}, we obtain the following 
useful lemma.

\begin{restatable}{lemma}{disconnectednotweak}
\label{disconnected_not_weak}
    Two fully supported disconnected graphs are never weakly disjoint. 
\end{restatable}

\begin{restatable}{prop}{factortimesfactorstrong}
\label{factor_times_factor_strong}
Let $G$ and $H$ be graphs. If for every graph $K$, the conditions $G \0 | \0 K$ and $H \0 | \0 K$ together imply $G \otimes H \0 | \0 K$, then $G$ and $H$ are strongly disjoint.
\end{restatable}

\vskip.1in

A graph is said to be \emph{non-trivial} if it contains two or more vertices with positive weights. The following proposition can be proved directly using Proposition~\ref{furstenberg}.

\begin{restatable}{prop}{factortoweak}
\label{factor_to_weak_disjoint}
If there is a non-trivial graph $K$ such that $K \0 | \0 G$ and $K \0 | \0 H$,
then $G$ and $H$ are not weakly disjoint. In other words, weakly disjoint
graphs have no non-trivial common factors.
\end{restatable}


Proposition \ref{factor_to_weak_disjoint} can be used to determine 
when two given graphs are not weakly disjoint. 
For example, two bipartite graphs are never weakly disjoint, 
as they share a common factor $P_2$. 
The converse of Proposition \ref{factor_to_weak_disjoint}
does not hold. 
Figure~\ref{graph_joining_example} (D) provides a counterexample. The detailed 
analysis is provided in the proof section.

\begin{restatable}{prop}{nofactornotweak}
\label{no_factor_not_weak}
There exist 
{connected} graphs $G$ and $H$ such that $G$ and $H$ have no non-trivial common factor and yet $G$ and $H$ are not weakly disjoint.
\end{restatable}

Within restricted graph families 
the absence of a common graph factor may be equivalent to weak 
or strong disjointness.

\begin{restatable}{prop}{remarkoncircle}
\label{remark_on_circle}
Consider cycle graphs $C_m$ and $C_n$ with $m,n \geq 3$, along with a path graph 
$P_n$ with $n \geq 2$. Then
\begin{itemize}

\item $C_m$ and $C_n$ are weakly disjoint if and only if
they share no common graph factor; 

\item $C_m$ and $P_n$ are strongly disjoint if and only if 
they share no common graph factor. 

\end{itemize}
\end{restatable}

\vskip.05in

\begin{rmk}
Propositions \ref{furstenberg} - \ref{no_factor_not_weak} parallel
existing results in ergodic theory 
(see \cite{furstenberg1967disjointness,glasner1983minimal}), 
in which graphs are replaced by measure-preserving systems, and weak and strong 
disjointness are replaced by disjointness of stationary processes, as
defined in Remark \ref{rmk:disjoint}.  Together with Proposition 
\ref{prop:factprop}, these results establish a rough analogy between
coprime numbers and disjoint graphs. Proposition \ref{no_factor_not_weak}
demonstrates that the analogy is not complete; its analog in the ergodic
theory setting is more involved than the result here,
see \cite{glasner1983minimal}. 
Coprime graphs, which connect integers that are coprime,
have been considered in the literature \cite{erdHos1997cycles} but 
are not related to the notion of graph factor considered here.
\end{rmk}

\vskip.1in

\section{Characterizing Weak and Strong Disjointness}
\label{Characterizing}

The first two main results of our investigation are linear algebraic characterizations of 
weak and strong disjointness. 

\subsection{Weak Disjointness}
Weak disjointness of connected graphs is characterized
by the spectral overlap of their Markov transition matrices.

\begin{defn}
The \emph{transition matrix} of a fully supported graph $G = (U, \alpha)$ with 
vertex set $U = \{ u_1, \cdots, u_m \}$ and degree function $p$ is
given by
\[
P \ = \ 
\left[ \frac{\alpha(u_i, u_j)}{p(u_i)} \right]_{i, j=1}^m 
\, \in \, \mathbb{R}^{m \times m}.
\]
In other words, $P$ is the transition probability matrix
of the reversible random walk on the vertices of $G$. 
\end{defn}

As any transition matrix is a stochastic matrix, it must have $1$ as an eigenvalue. 
The following theorem characterizes weak disjointness for connected graphs in terms 
of the rest of their spectra.


\begin{restatable}{thm}{weaknoshareeigen}
Connected graphs $G$ and $H$ are weakly disjoint if and only if their transition matrices share no eigenvalues other than 1.
\label{weak_no_share_eigen}
\end{restatable}

Theorem \ref{weak_no_share_eigen} is established in Section~\ref{proof_of_characterizing}. 
Its proof has points of contact with the 
classical Sylvester equation $AX + XB = C$ from control theory, which has a unique solution for $X$ exactly when there are no common eigenvalues of $A$ and $-B$ (see \cite{bartels1972algorithm,bhatia1997and}). 
To see the connection, let $P$ and $Q$ be the transition matrices 
of $G$ and $H$, respectively.
By summing the transition coupling conditions in the definition
of a weight joining, we obtain a matrix equation of the form $P^T \Gamma = \Gamma Q$, 
where $\Gamma = [\Gamma(i,j)]$ with $\Gamma(i,j) = r(u_i,v_j)$ is the 
degree matrix of a valid graph joining.
If the solution space for $\Gamma$ 
is one-dimensional, then the degree functions of all
valid graph joinings are identical, implying that $G$ and $H$ are weakly disjoint. 
Otherwise, there exists a graph joining between $G$ and $H$ whose degree function is 
not equal to the product of the degree functions of $G$ and $H$.
This leads to a homogeneous Sylvester equation $AX + XB = 0$, with $A = P^T$ 
and $B = -Q$. It is known that this equation has a unique solution (zero)
if and only if $A$ and $-B$ do not share any common eigenvalues. 
In the present case, as $G$ and $H$ are connected the matrices $P$ and $Q$ 
share a common eigenvalue 1. This leads to a solution equal to the tensor product of 
the marginal degree functions, and the question of interest is whether there exists any 
other, nontrivial solutions.
The central idea of the proof is to establish a connection between the 
dimension of the solution space of $P^T \Gamma = \Gamma Q$ and the weak disjointness of $G$ and $H$. In particular, we show that 
the solution space is one-dimensional if and only if $P$ and $Q$ share exactly one eigenvalue (which must be $1$).


One may extend Theorem~\ref{weak_no_share_eigen} to address fully supported disconnected graphs 
$G$ and $H$. The following proposition presents the result in full generality. Recall that for a square matrix, the \emph{algebraic multiplicity} of an eigenvalue is its multiplicity as a root of the characteristic polynomial.  We say that two transition matrices $P$ and $Q$ share an eigenvalue $\lambda$ with \emph{overlapping multiplicity} $k$ if $\lambda$ is an eigenvalue of both matrices and $k$ equals the minimum of its algebraic multiplicities with respect to $P$ and $Q$.

\begin{restatable}{prop}{weaknoshareeigengeneral}
\label{weak_no_share_eigen_general}
    Fully supported graphs $G$ and $H$ are weakly disjoint if and only if their transition matrices share no eigenvalue other than 1 and the eigenvalue 1 has overlapping multiplicity one.
\end{restatable}

\subsection{Strong Disjointness}

\label{discussion_strong_disjointness}

Let $G = (U, \alpha)$ and $H = (V, \beta)$ be graphs. 
Weight joinings of $\alpha$ and $\beta$ are characterized by the transition coupling, 
degree coupling, and weight function conditions, each of which can be expressed as a
linear constraint.  In particular, one may encode the constraints into a matrix
$J = J(G,H)$ such that weight joinings $\gamma$ of $\alpha$ and $\beta$ are in
one-to-one correspondence with vectors $\Vec{\gamma}$ satisfying
\begin{align*}
    J \Vec{\gamma} \, = \, 0, \quad \mathbbm{1}^T \Vec{\gamma} \, = \, 1, \quad \Vec{\gamma} \, \geq \, 0.
\end{align*}
We refer to $J$ as the \emph{joining constraint matrix} of $G$ and $H$. The construction of $J$ is given in Section~\ref{weight_joining_constraint_matrix}.

The following proposition characterizes strong disjointness for arbitrary graphs. Moreover, it implies that strong disjointness can be detected in polynomial time, since computing the null space of a matrix can be done in time polynomial in its dimensions, and the number of rows and columns of the joining constraint matrix $J$ is itself polynomial in $|U|$ and $|V|$.

\begin{restatable}{prop}{rankequivstrong}
\label{rank_equiv_strong}
Graphs $G$ and $H$ are strongly disjoint if and only if the joining 
constraint matrix $J$ for $G$ and $H$ has a null space of dimension 1.
\end{restatable}

\subsection{A Connection between Weak and Strong Disjointness}

For connected graphs $G$ and $H$ without self-loops, there is a simple relationship
between weak and strong disjointness, which may be derived from 
a case-by-case analysis based on the connectivity level of each
graph.  
Let $E(G)$ be the edge set of $G = (U,\alpha)$. We consider
three connectivity levels: 
$|U| = |E(G)| - 1$, $|U| = |E(G)|$, and $|U| > |E(G)|$.
The first level corresponds to trees, the second to graphs containing exactly one cycle, and the third to graphs with two or more cycles. 
For connected graphs $G$ and $H$ without self-loops, Table~\ref{fig:connectivity-table} summarizes  
the absence of strong disjointness (SD) and
the relationship between and strong disjointness and weak disjointness (WD). 

\vskip.15in
\footnotesize
    \begin{table}[h!]
    \centering
    \begin{tabular}{|c|c|c|c|}
    \hline
    \sqcell{Connectivity levels of $G$ and $H$} & \sqcell{$|V|>|E(H)|$} & \sqcell{$|V|=|E(H)|$} & \sqcell{$|V|<|E(H)|$} \\ \hline
    \sqcell{$|U|>|E(G)|$} & \sqcell{\textcircled{\footnotesize{1}} not SD} & \sqcell{\textcircled{\footnotesize{5}} WD $\Leftrightarrow$ SD} & \sqcell{\textcircled{\footnotesize{6}} WD $\Leftrightarrow$ SD} \\ \hline
    \sqcell{$|U|=|E(G)|$} & \sqcell{\textcircled{\footnotesize{5}} WD $\Leftrightarrow$ SD} & \sqcell{\textcircled{\footnotesize{2}} not SD} & \sqcell{\textcircled{\footnotesize{4}} not SD} \\ \hline
    \sqcell{$|U|<|E(G)|$} & \sqcell{\textcircled{\footnotesize{6}} WD $\Leftrightarrow$ SD} & \sqcell{\textcircled{\footnotesize{4}} not SD} & \sqcell{\textcircled{\footnotesize{3}} not SD} \\ \hline
    \end{tabular}
    \normalsize
    \caption{
    Relationship between strong and weak disjointness for connected  graphs $G$ and $H$ without self-loops. Rows correspond to the connectivity levels of $G$ and columns correspond to those of $H$. Circled numbers indicate different connectivity level scenarios.  Note that in scenarios $\textcircled{\small{5}}$ and $\textcircled{\small{6}}$, 
    $G$ and $H$ may be strongly disjoint or not strongly disjoint.  
    }
    \label{fig:connectivity-table}
    \end{table}

\normalsize



As a corollary of the case-by-case analysis summarized in 
Table~\ref{fig:connectivity-table}, we obtain the following simple 
characterization of strong disjointness. 

\begin{restatable}{thm}{characterizationconnectednoself}
Let $G$ and $H$ be connected graphs with no self-loops. Then $G$ and $H$ are strongly disjoint if and only if they are weakly disjoint and exactly one of the graphs is a tree.
\label{Characterization_connected_no_self}
\end{restatable}

The previous result remains valid for fully supported graphs with no self-loops, even when the graph is disconnected.
\begin{restatable}{cor}{characterizationfullysupportednoself}
    Let $G$ and $H$ be fully supported graphs with no self-loops. Then $G$ and $H$ are strongly disjoint if and only if they are weakly disjoint and exactly one of the graphs is a forest.
\label{Characterization_fully_supported_no_self}
\end{restatable}

\vskip.15in

\section{Persistence of Disjointness}
\label{Prevalence_of_Disjointness}

Let $\mathcal{G}$ and $\mathcal{H}$ be families of undirected graphs. 
A pair of graphs $G \in \mathcal{G}$ and $H \in \mathcal{H}$ may be
strongly disjoint, weakly disjoint, or neither. 
In general there is no reason 
to expect that the same relationship should hold 
for all (or most) such pairs.  
The situation is different, however, if each family $\mathcal{G}$ and $\mathcal{H}$ 
consists of weighted graphs with the same vertex and edge sets.
Given $U$ and $E \subseteq U \times U$, let
\begin{align}
W(U,E) \, = \, 
\left\{ \text{weight functions } \alpha \text{ on } U \text{ such that }\supp(\alpha) = E \right\}.
\label{set_skeleton_weight}
\end{align}
Note that $W(U,E)$ is empty if $E$ is not symmetric.
The family $W(U,E)$ may be represented as a subset of $\mathbb{R}^{d}$, where $d = |\{\{u,u'\}\subset U:(u,u')\in E\}|-1$ is the number of undirected edges in $E$, reduced by one due to the normalization constraint that the weights sum to one. 
The normalization constraint ensures that $W(U,E)$ is a bounded set, and 
the support condition ensures that $W(U,E)$ is open in $\mathbb{R}^{d}$. 

One may regard the pair $(U,E)$ underlying the family of weight
functions (and associated graphs) $W(U,E)$ as a skeleton.
Given two skeletal families, we show below that the
disjointness relation between their elements are 
essentially determined by their skeletons. 
Let $B \subseteq \mathbb{R}^d$ be bounded and open. We will call 
a set $A \subseteq B$ \emph{full} if it is open and dense in $B$ and if
$\lambda_d(A) = \lambda_d(B)$, where $\lambda_d$ denotes $d$-dimensional
Lebesgue measure. 
The proof of the following proposition and its corollary 
are given in Section \ref{proof_for_prevalence}.  

\begin{restatable}{prop}{prevalence}
\label{prevalence_strong_weak}
Let $U$ and $V$ be finite sets, and let $E \subseteq U \times U$ and 
$F \subseteq V \times V$ be symmetric. Then {exactly} one of the following holds:
\begin{itemize}

\vskip.05in

\item[(1)]
For all $\alpha \in W(U,E)$ and $\beta \in W(V,F)$, the graphs $G = (U,\alpha)$ and 
$H = (V,\beta)$ are not weakly disjoint.

\vskip.06in

\item[(2)] 
There is a full subset $A \subseteq W(U,E)$ with the property that for 
each $\alpha \in A$ there is a full subset $B_{\alpha} \subseteq W(V,F)$ 
such that $G = (U,\alpha)$ and $H = (V,\beta)$ are weakly
disjoint for every $\beta \in B_{\alpha}$. 

\end{itemize}
The dichotomy above continues to hold if ``weakly disjoint"  is replaced by ``strongly disjoint" in both statements.
\end{restatable}

\begin{restatable}{cor}{zeromeasure}
\label{zero_measure}
In case (2) of Proposition~\ref{prevalence_strong_weak}, the set of weight function pairs 
$\left(\alpha, \beta\right)$ such that $(U,\alpha)$ and $(V,\beta)$ 
are not weakly disjoint has measure zero in 
${W}(U, E) \times {W}(V, F)$.  The same conclusion holds if ``weakly disjoint" is replaced by ``strongly disjoint".
\end{restatable}

\begin{rmk}
Both cases described in Proposition~\ref{prevalence_strong_weak} can occur. Case (1) arises, 
for example, when both families $W(U,E)$ and $W(V,F)$ consist of trees. In this case, 
the families are not weakly disjoint, since they share a common nontrivial graph factor 
$P_2$.
Case (2) occurs when both families $W(U,E)$ and $W(V,F)$ consist of weighted triangles. 
To see this, consider the uniformly weighted triangle and the triangle with edge weights in the ratio $1 : 1 : 2$. By explicitly calculating their eigenvalues and applying Theorem~\ref{weak_no_share_eigen}, we may verify that these two graphs
are weakly disjoint. Hence Statement (1) of the dichotomy in Proposition~\ref{prevalence_strong_weak} cannot hold, and therefore Statement (2) holds.
Even in this simple example, there are
also pairs of weighted triangles that are not weakly disjoint, for example isomorphic pairs, but such pairs are topologically and measure-theoretically exceptional by Statement (2) and Corollary~\ref{zero_measure}.
Similar remarks apply to the case of 
strong disjointness.
\end{rmk}


\section{Characterizing Graph Families Using Disjointness}
\label{characterize_graph_family}

Knowing that a graph $G$ is disjoint from another graph $H$ 
can provide information about the structure of $G$.  
Recall that a graph $G = (U,\alpha)$ is said to be \emph{bipartite} 
if its vertex set $U$ can be partitioned into two disjoint 
subsets $U_1$ and $U_2$ such that all the edges of $G$ 
connect a vertex in $U_1$ to a vertex in $U_2$: there are no
edges between vertices in $U_1$ or between vertices in $U_2$. 
Also recall that $P_2$ denotes the simple path with 2 vertices and no
self-loops.


Note that the results in this section are motivated by and analogous to existing results in ergodic theory (e.g., Theorem 3.2 in \cite{de2023joinings}).

\begin{restatable}{prop}{bipartitestrong}
If $G$ is connected, then the following are equivalent: 
\begin{enumerate}
    \item $G$ is bipartite.
    \item $P_2\0|\0G$.
    \item $G$ is not strongly disjoint from $P_2$.
    \item $G$ is not weakly disjoint from $P_2$.
    \end{enumerate}
\label{bipartite_strongdisjoint}
\end{restatable}

\begin{restatable}{prop}{connectedstrong}
Let $G$ be a fully supported graph, and let $\mathcal{H}$ be
the family of graphs consisting of two isolated vertices, each 
equipped with a self-loop of positive weight.
The following are equivalent: 
\begin{enumerate}
    \item $G$ is connected.
    \item $G$ is strongly disjoint from some $H \in \mathcal{H}$.
    \item $G$ is strongly disjoint from every $H \in \mathcal{H}$.
    \end{enumerate}
\label{connected_self_loop}
\end{restatable}

The two propositions can be combined to conclude that a graph $G$ is connected and 
non-bipartite if and only if it is strongly disjoint from the disjoint 
union of $P_2$ and a single vertex with a self-loop.

\section{Connections with Finite State Markov Chains}
\label{Section:MC}

\subsection{Reversible Markov Chains}
\label{markovchain}


A reversible Markov chain $Z = Z_0, Z_1, \cdots$ with finite state 
space $S$ is described by a transition matrix 
$R(s' \mid s) = \mathbb{P}(Z_{n+1} = s' \mid Z_n = s)$ and a probability mass function $r$
on $S$ that together satisfy the detailed balance equations:
\[
r(s) \, R(s' \mid s) = r(s') \, R(s \mid s') \mbox{ for all } s, s' \in S.
\]
The detailed balance equations ensure that $r$ is a stationary 
distribution for $Z$.
Let $S^{\mathbb{N}}$ denote the set of infinite sequences
$(s_0, s_1, \ldots)$ with $s_i \in S$, with Borel sigma field
$\mathcal{B}(S^{\mathbb{N}})$ 
under the product topology.  
Together, $R$ and $r$ determine $\text{Law}(Z)$, which is a stationary
measure $\mu$ on $(S^{\mathbb{N}}, \mathcal{B}(S^{\mathbb{N}}))$ such that
for all $k \geq 1$ and $u_0, \ldots, u_k \in U$,
\[
\mu(\{s_0\} \times \cdots \times \{s_k\} \times S \times S \times \cdots) 
\ = \ r(s_0) \, \prod_{i=1}^k R(s_i \mid s_{i-1}) .
\]
Let $\mathbb{M}_r(S)$ be the family of reversible Markov chains
with state space $S$, and let 
$\mathbb{L}_r(S) \coloneqq \{ \text{Law}(Z) : Z \in \mathbb{M}_r(S) \}$
be the corresponding family of stationary reversible Markov measures.
 
Every fully supported graph $K = (S, \gamma)$ gives rise
to a measure $\mu \in \mathbb{L}_r(S)$ with transition matrix $R$
and stationary distribution $r$ specified by 
\begin{align*}
R(s' \mid s) \ = \ \frac{\gamma(s,s')}{r(s)} 
\ \text{ and } \ 
r(s) = \sum_{s' \in S} \gamma(s,s').
\end{align*}
Conversely, it is easy to see that every measure $\mu \in \mathbb{L}_r(S)$
can be generated by a pair $(R,r)$ associated with a 
fully supported graph $K = (S, \gamma)$.
In particular, there is a bijective correspondence
between $\mathbb{L}_r(S)$ and the family $\mathbb{G}(S)$ of all fully supported graphs with vertex set $S$.  In the sequel we will
write $Z \sim K$ to denote the fact that $\text{Law}(Z)$ is equal to
the Markov measure generated by $K$.

\subsection{Reversible Markovian Couplings}


Given a process $X = X_0, X_1, \ldots$ with values in $U$ and a function
$f: U \to V$, let $f(X)$ denote the $V$-valued process 
$f(X_0), f(X_1), \ldots$.
In what follows, $X \stackrel{d}{=} Y$ indicates that 
$\text{Law}(X) = \text{Law}(Y)$.
Let $\pi_1(u,v) = u$ and 
$\pi_2(u,v) = v$ be the usual coordinate projections on $U \times V$.

\vskip.05in

\begin{defn}
A reversible Markovian coupling of $X \in \mathbb{M}_r(U)$ and 
$Y \in \mathbb{M}_r(V)$ is a process $Z \in \mathbb{M}_r(U \times V)$ 
such that $X \stackrel{d}{=} \pi_1(Z)$ and $Y \stackrel{d}{=} \pi_2(Z)$. 
Let $\mathbb{J}(X,Y) \subseteq \mathbb{M}_r( U \times V )$
denote the family of reversible Markovian couplings 
of $X$ and $Y$.
\end{defn}

Note that the paired process
$X \otimes Y = (X_0', Y_0'), (X_1',Y_1'), \ldots$ where $X'$ and $Y'$ are
independent copies of $X$ and $Y$ is always a reversible 
Markovian coupling of $X$ and $Y$. Thus $\mathbb{J}(X,Y)$ is non-empty. {The proof of the following proposition is given in Section~\ref{proof_for_markovchain}.}.

\begin{restatable}{prop}{bijectivechaingraph}
Let $X \sim G = (U,\alpha)$ and $Y \sim H = (V,\beta)$.  Then there is 
a bijective correspondence between $\mathcal{J}(G,H)$ and
$\{\mbox{Law}(Z) : Z \in \mathbb{J}(X,Y) \} 
\subseteq \mathbb{L}_r(U \times V)$.
\label{bijective_chain_graph}
\end{restatable}

\vskip.15in

\subsection{Disjoint Markov Chains}
\label{disjoint-MC}

\begin{defn}[Disjointness for Chains]
Let $X \in \mathbb{M}_r(U)$ and $Y \in \mathbb{M}_r(V)$ be reversible
Markov chains with stationary distributions $p$ and $q$, respectively.
\begin{enumerate}

\vskip.08in

\item
$X$ and $Y$ are strongly M-disjoint if for every
$Z \in \mathbb{J}(X,Y)$, $Z \stackrel{d}{=} X \otimes Y$.


\vskip.08in

\item 
$X$ and $Y$ are weakly M-disjoint if for every 
$Z \in \mathbb{J}(X,Y)$, $Z_0 \stackrel{d}{=} p \otimes q$.

\vskip.1in

\item 
$X$ and $Y$ are $c$-M-disjoint with respect to a cost
$c: U \times V \rightarrow [0,\infty)$ if
\[
X \otimes Y
\, \in \, 
\mathop{\mathrm{arg\,min}}_{Z \in \mathbb{J}(X,Y)} 
\E \, c(Z_0).
\]

\end{enumerate}
\end{defn}

\vskip.2in

\subsection{Factors of Markov Chains}
\label{markovchain_factor}

\begin{defn}
Let $X \in \mathbb{M}_r(U)$ and $Y \in \mathbb{M}_r(V)$ be reversible Markov chains. 
The chain $Y$ is a factor of $X$, denoted $Y \0|\0 X$, 
if there is a surjective map $\phi: U \to V$ such that $Y \stackrel{d}{=} \phi(X)$.
\end{defn}

\begin{rmk}
\label{rmk_factors}
In ergodic theory there is a well established notion of factors for
measure-preserving systems, which yields a corresponding notion  
for stationary processes. If a Markov chain $X$ is a factor of a 
Markov chain $Y$ in the sense defined here, then $X$ is a factor
of $Y$ in the more general sense of stationary processes.
In non-Markovian settings, our 
definition of factor is not comparable to the one for stationary 
processes.
In particular, the results here do not imply those in ergodic theory, and conversely.
\end{rmk}

\begin{rmk}
\label{rmk_lumpability}
The idea of aggregating or collapsing the states of a graph was
introduced in Section~\ref{factors} in the context of 
graph factors.  The definition of factor above
is similarly based on aggregating 
the states of a Markov chain. These definitions have close connections with
existing concepts in graph theory and the theory of Markov chains, which we briefly summarize below.

A Markov chain is said to be lumpable 
(see e.g., \cite[Sec. 2.3.1]{levin2017markov}) if its state space can 
be partitioned into disjoint groups such that the transition probability 
from one group to another depends only on the current group, 
and not on the particular state within that group.  In this case, 
the induced process on the groups is itself a Markov chain. 

In graph theory, an equitable partition (\cite[Sec. 9.3]{royle2001algebraic}) is a partition of the vertex set into cells such that every vertex in a given cell has the same number of neighbors in each other cell. Such a partition naturally induces a quotient graph. For unweighted graphs, equitable partitions correspond exactly to lumpable partitions of the associated random walk, and thus can be viewed as a special case of Markov chain lumpability. 

Lumpability is closely related to our definition of factor for 
Markov chains.
In fact, if the chain $Y$ is irreducible, then $Y \0 | \0 X$ under $\phi$ is equivalent to saying 
that $X$ is lumpable with respect to the partition of $U$ induced by $\phi$, 
the corresponding lumped chain being $Y$. 
If $Y$ is reducible, then $Y \0|\0 X$ under $\phi$ still implies that $X$ 
is lumpable with respect to the partition induced by $\phi$, 
but the converse does not necessarily hold.

Lumpable reversible Markov chains and their spectra have been studied 
in the literature.  Orbit partitions arising from graph automorphisms 
provide a special case of equitable partitions \cite{boyd2005symmetry,filliger2008lumping}.
If $X$ is lumpable and $Y$ is the
corresponding reduced chain, then the eigenvalues of 
the transition matrix of $Y$ 
are a subset of those of $X$ \cite{barr1977eigenvector}. 
Since $Y \0 | \0 X$ yields lumpability, this spectral inclusion property 
also holds for factors. Most of the lumpability literature focuses on a 
single Markov chain and the reduced chain obtained by aggregating its states.
There are, however, studies that relate multiple Markov chains to lumpability, but this work considers mutual independence of the Markov chains rather than their interaction \cite{ball1993lumpability}. 

As studied here factors are a means of quantifying the structural 
relationship between two graphs or two chains, rather than a means
of reducing or simplifying them. 
This perspective is fundamentally different from that of lumpability 
in Markov chain theory.
\end{rmk}

The following proposition states several properties of factors of a Markov chain and clarifies their relationship to joinings, as well as to strong and weak M-disjointness.
Its proof follows directly from the results developed in Section~\ref{factors}, 
 since Proposition~\ref{bijective_chain_graph} establishes a one-to-one correspondence between graph joinings and the laws of reversible Markovian couplings. In the following text, we say a chain is \emph{nontrivial} if its stationary distribution
is supported on two or more states.

\begin{prop}
Let $X$ and $Y$ be reversible Markov chains. 
\begin{enumerate}

\vskip.07in

\item 
If $X \sim G$ and $Y \sim H$, then $X \0|\0 Y$ under $\phi$ 
if and only if $G \0|\0 H$ under $\phi$.

\vskip.07in

\item 
For $Z \in \mathbb{M}_r(U \times V)$, we have $Z \in \mathbb{J}(X,Y)$ if and only if 
$X\0|\0 Z$ under $\pi_1$ and $Y \0|\0 Z$ under $\pi_2$.

\vskip.07in

\item 
Suppose that reversible Markov chains $Z$ and $W$ are strongly M-disjoint. 
If $X \0|\0 Z$ and $Y \0|\0 W$, then 
$X$ and $Y$ are strongly M-disjoint. Similarly, if $Z$ and $W$ are weakly M-disjoint and we have $X \0|\0 Z$ and $Y \0|\0 W$, then 
$X$ and $Y$ are weakly M-disjoint.

\vskip.07in

\item 
If, for every reversible Markov chain $Z$, $X \0|\0 Z$ and $Y \0|\0 Z$ 
together imply $X \otimes Y \0|\0 Z$,
then $X$ and $Y$ are strongly M-disjoint.

\vskip.07in

\item 
If there exists a nontrivial reversible Markov chain $W$ such that $W \0|\0 X$ and $W \0|\0 Y$, then $X$ and $Y$ are not weakly M-disjoint; however, the converse does not generally hold.

\end{enumerate}

\end{prop}

\vskip.2in

\subsection{Characterizing Disjointness of MCs}

Section~\ref{Characterizing} contains characterizations of weak and strong disjointness 
for graphs. These results yield corresponding characterizations of weak and strong M-disjointness for Markov chains.

\begin{thm}
Reversible chains $X$ and $Y$ are weakly M-disjoint if and only if 
the only common eigenvalue of their transition matrices is $1$, 
and its overlapping multiplicity is one.
\end{thm}

A reversible Markov chain $X$ is a \emph{forest chain} if its underlying
transition graph contains no cycles.

\begin{thm}
Let $X$ and $Y$ be reversible Markov chains with zeros on the diagonals 
of their transition matrices. Then $X$ and $Y$ are strongly M-disjoint 
if and only if they are weakly M-disjoint and exactly one of $X$ and $Y$ 
is a forest chain.
\end{thm}

\begin{rmk}
    We also note a consequence of weak M-disjointness. Let $X \in \mathbb{M}_r(U)$ and $Y \in \mathbb{M}_r(V)$ be reversible Markov chains with stationary distributions $p$ and $q$, respectively. If $X$ and $Y$ are weakly M-disjoint, then for every joining $Z = (\tilde{X},\tilde{Y} ) \in \mathbb{J}(X,Y)$, and any functions $f:U\to\mathbb{R}$ and $g:V\to \mathbb{R}$ such that $f\in L^1(p)$ and $g \in L^1(q)$, the ergodic theorem gives
    \begin{align*}
        \lim_{n\to\infty} \, \frac{1}{n}\sum_{k=1}^n f(\tilde{X}_k) g(\tilde{Y}_k) \,=\, \mathbb{E}_{p\otimes q}\left[f(\tilde{X}_1)g(\tilde{Y}_1)\right] \,=\,  \left(\sum_{u \in U}f(u)p(u)\right)\left(\sum_{v\in V}g(v)q(v)\right).
    \end{align*}
\end{rmk}

\vskip.2in

\subsection{Persistence of Disjointness for Skeletal Families of MCs}

Here we present an analogue of the result in Section~\ref{Prevalence_of_Disjointness} formulated in the setting of Markov chains. Since the statement is expressed rigorously in terms of graphs, we describe the corresponding version here informally.
Let $\mathbb{L}_r(U,E)$ denote the space of laws of all stationary reversible Markov chains on the vertex set $U$ with stationary distribution and transition matrix having support exactly on $E$. In other words, $(U,E)$ specifies the skeleton of the Markov chains whose laws belongs to $\mathbb{L}_r(U,E)$. It can be verified that $\mathbb{L}_r(U,E)$ is in bijective correspondence with $W(U,E)$ (defined in~\eqref{set_skeleton_weight}), and we endow $\mathbb{L}_r(U,E)$ with the topology that makes this correspondence a homeomorphism. Furthermore, we say that a subset $\mathcal{A} \subseteq \mathbb{L}_r(U,E)$ is full if it corresponds to a full subset of $W(U,E)$ (as defined in \ref{Prevalence_of_Disjointness}). 


The following proposition works for both weak and strong M-disjointness.
\begin{prop}
Let $U$ and $V$ be finite sets, $E \subseteq U \times U$ and 
$F \subseteq V \times V$. Then exactly one of the following holds:
\begin{enumerate}

\item If $X$ and $Y$ are chains with
$\text{Law}(X) \in \mathbb{L}_r(U,E)$ and 
$\text{Law}(Y) \in \mathbb{L}_r(V,F)$, then $X$ and $Y$ are 
not weakly M-disjoint.
        
\item There is a full subset $\mathcal{A} \subseteq \mathbb{L}_r(U,E)$ with 
the property that for each chain $X$ such that $\text{Law}(X) \in \mathcal{A}$, 
there is a full subset $\mathcal{B} \subseteq \mathbb{L}_r(V,F)$ such that 
$X$ is weakly M-disjoint from every chain $Y$ with 
$\text{Law}(Y) \in \mathcal{B}$. 

\end{enumerate}
The dichotomy above continues to hold if ``weakly M-disjoint"  is replaced by ``strongly M-disjoint" in both statements.
\label{Prevalence_MC}
\end{prop}

\begin{cor}
In the case (2) of Proposition~\ref{Prevalence_MC}, the set of pairs
$(\mathbb{P}_1,\mathbb{P}_2) \in \mathbb{L}_r(U,E)\times \mathbb{L}_r(V,F)$ such that the chains $X\sim \mathbb{P}_1$ and $Y \sim \mathbb{P}_2$ are weakly 
M-disjoint has measure zero in the product space 
$\mathbb{L}_r(U,E)\times \mathbb{L}_r(V,F)$.
\end{cor}

\subsection{Characterizing Chains using Disjointness}
In Section~\ref{characterize_graph_family} we discuss the characterization of graph families, and in this section we translate them into the language of Markov chains. Suppose $X$ is an irreducible, reversible stationary Markov chain. Let $W_{\text{flip}}$ be a two-state deterministic flip chain on $\{u_1,u_2\}$, with transition matrix 
$$P = \left(\begin{array}{cc}
    0 & 1 \\
    1 & 0
\end{array}\right).$$ 
Then a reversible Markov chain $X$ has period 2 if and only if $X$ is not strongly M-disjoint from $W_{\text{flip}}$ if and only if $X$ is not weakly M-disjoint from $W_{\text{flip}}$.
Let $W_{\text{fixed}}$ denote the two-state deterministic Markov chain on $\{u_1,u_2\}$ with transition matrix
$$P = \left(\begin{array}{cc}
    1 & 0 \\
    0 & 1
\end{array}\right).$$ 
That is, the chain remains forever at its initial state. For any choice of initial distribution, the law of $W_{\text{fixed}}$ is stationary and reversible. Then a reversible Markov chain $X$ is irreducible if and only if $X$ is strongly M-disjoint from $W_{\text{fixed}}$ for some choice of initial distribution if and only if $X$ is strongly M-disjoint from $W_{\text{fixed}}$ for every choice of initial distribution.

\section{Proofs for Section~\ref{Disjointness}}
\label{proof_of_disjointness_def}

In this section we prove the results stated in Section~\ref{Disjointness}.
Recall that $\mathcal{J}(\alpha,\beta)$ denotes the set of weight joinings of $\alpha$ 
and $\beta$, while $\mathcal{J}(G,H)$ denotes the set of graph joinings of $G$ and $H$.

\subsection{Proposition~\ref{c_to_weak}}
\label{proof_of_weak_and_c}
\ctoweak*

\begin{proof}
If $G = (U,\alpha)$ and $H = (V, \beta)$ are weakly disjoint, then any weight joining 
$\gamma$ of $\alpha$ and $\beta$ has degree function $r = p \otimes q$,
which is equal to the degree of $\alpha \otimes \beta$. 
Since the $c$-optimality of a weight joining depends only on its degree function, we conclude that $\alpha \otimes \beta$ is an optimal weight joining for any cost function $c$, and hence $G$ and $H$ are $c$-disjoint for any $c$.

Suppose conversely that $G$ and $H$ are $c$-disjoint for every cost function 
$c: U \times V \to \mathbb{R}_{\geq 0}$. 
Fix such a cost function $c$. Let $\gamma = \alpha \otimes \beta$ be the usual 
product weight joining, and let $r$ be its degree function. 
Let 
$\Tilde{\gamma}$ be any weight joining of $\alpha$ and $\beta$ with corresponding degree function $\tilde{r}$. 
Note that by the edge-preservation property we have $\supp(\tilde{\gamma}) \subseteq \supp(\gamma)$, and then for all sufficiently small $t >0$,
we see that $(1+t) \gamma - t \tilde{\gamma}$ is nonnegative and therefore a weight joining of $\alpha$ and $\beta$. Thus $c$-disjointness 
yields
\begin{align*}
        \sum\limits_{(u,v)\in U\times V}c(u,v)r(u,v) 
        \,
        \leq 
        \,
        \sum\limits_{(u,v)\in U\times V}c(u,v)((1+t)r(u,v)-t \tilde{r}(u,v)).
\end{align*}
By rearranging this inequality and dividing by $t$ and also by appealing again to $c$-disjointness, we find that
    \begin{align*}
        \sum\limits_{(u,v)\in U\times V}c(u,v)\tilde{r}(u,v) \, \leq \, \sum\limits_{(u,v)\in U\times V}c(u,v)r(u,v)
        \, \leq \, \sum\limits_{(u,v) \in U \times V} c(u,v) \tilde{r}(u,v).
    \end{align*}
 Since $c$ was arbitrary, we conclude that for all cost functions $c$, we have
    \begin{align}
        \sum\limits_{(u,v)\in U\times V}c(u,v)\tilde{r}(u,v) \,
        = \, \sum\limits_{(u,v)\in U\times V}c(u,v)r(u,v).
        \label{equal}
    \end{align}
Plugging the cost function $c(u,v)  =\textbf{1}_{((u,v)=(u_0,v_0))}$ into equation \eqref{equal} for each $(u_0,v_0)\in U\times V$, we obtain $\tilde{r}(u_0,v_0) = r(u_0,v_0)$ for all $(u_0,v_0)\in U\times V$.
Since $\tilde{\gamma}$ is an arbitrary weight joining of $\alpha$ and $\beta$, we conclude that $G$ and $H$ are weakly disjoint. 
\end{proof}

\subsection{Proposition~\ref{line_circle}}
\label{proof_of_cycle_path}
\linecircle*

The proof below relies on Theorems~\ref{weak_no_share_eigen} and \ref{Characterization_connected_no_self}. The result can also be
established by a direct argument, which is substantially longer.

\begin{proof}
\label{second_proof}

The eigenvalues of the transition matrix for  the 
path graph $P_n$ are given by
$$\left\{\cos\left(\frac{k\pi}{n-1}\right)\right\}_{k=0}^{n-1};$$ 
see \cite[Proposition~3.5]{boyd2005symmetry}. 
Note that the transition matrix of $C_m$ is a circulant matrix generated by the vector $(0,1 / 2,0, \ldots, 0,1 / 2)$, so that each row assigns mass $1/2$ to the two neighboring indices modulo the matrix dimension.
Using the known formula for the eigenvalues of symmetric circulant matrices (Equation (3) in \cite{rojo2004some}), 
we find that all the eigenvalues of $C_m$ are given by
\begin{align}
\label{eigenvalue_circulant}
\left\{\cos\left(\frac{2k\pi}{m}\right)\right\}_{k=0}^{m-1}.
\end{align}
It is straightforward to show that these two spectra share no eigenvalue other than 1 
if and only if $\gcd(m, 2(n-1)) = 1$, which is equivalent to $m$ being odd and $\gcd(m,n-1) = 1$.
The result now follows from 
Theorems \ref{weak_no_share_eigen} and \ref{Characterization_connected_no_self}.
\end{proof}



\subsection{Proposition~\ref{circle_property}}
\label{proof_of_circle_property}
\circleweakcoprime*

\begin{proof}
We first prove that $C_m$ and $C_n$ are never strongly disjoint.
Write $C_m = (U,\alpha)$ and $C_n = (V,\beta)$.
    Let $U = (u_1, \cdots, u_m)$ and $V = (v_1, \cdots, v_n)$, where the sequences $u_1 - u_2 - \cdots - u_m - u_1$  and $v_1 - v_2 - \cdots - v_n - v_1$ form  cycles. Extend the indices by setting $u_0 = u_m$, $u_{m+1} = u_1$, $v_0 = v_n$, and $v_{n+1} = v_1$. By the edge preservation property, for any weight joining $\gamma$ of $\alpha$ and $\beta$, the value $\gamma((u_i,v_j),(u',v'))$ can be nonzero only when 
    $$(u',v') \in \{ (u_{i+1},v_{j+1}), (u_{i+1},v_{j-1}), (u_{i-1},v_{j+1}), (u_{i-1},v_{j-1}) \}.$$
    Define 
    \begin{equation}
    \begin{aligned}
        &\gamma((u_i, v_j), (u_{i+1}, v_{j+1})) \, = \, \gamma((u_i, v_j), (u_{i-1}, v_{j-1})) \, = \, \frac{1}{2mn}, \\ 
        &\gamma((u_i, v_j), (u_{i+1}, v_{j-1})) \, = \, \gamma((u_i, v_j), (u_{i-1}, v_{j+1})) \, = \, 0,
        \label{gamma_construction_cycle}
        \end{aligned}
    \end{equation}
    for all $i \in \{ 1, \cdots, m \}$ and $j \in \{ 1, \cdots, n \}$. It is easy to check that $\gamma \neq \alpha \otimes \beta$ (e.g., they differ at  $((u_2,v_2),(u_3,v_1))$).
    Note that under this definition, the degree function $r$ becomes
    \begin{align*}
        r(u_i,v_j) = \gamma((u_i, v_j), (u_{i+1}, v_{j+1})) + \gamma((u_i, v_j), (u_{i-1}, v_{j-1})) = \frac{1}{mn}.
    \end{align*}
    To verify the transition coupling conditions, note that
    \begin{align*}
        \sum\limits_{v'\in V}\gamma((u_i,v_j),(u',v')) \, = \, \begin{cases}
            \frac{1}{2mn} & \text{if } u'\in\{u_{i-1},u_{i+1}\},\\
            0 &\text{otherwise,}
        \end{cases}
    \end{align*}
    and
    \begin{align*}
        \frac{\alpha(u_i,u')}{p(u_i)} \0 r(u_i,v_j) \, = \, \begin{cases}
            \frac{1}{2} \cdot \frac{1}{mn} &\text{if } u'\in\{u_{i-1},u_{i+1}\},\\
            0 &\text{otherwise.}
        \end{cases}
    \end{align*}
    Thus, $\sum_{v'\in V}\gamma((u_i,v_j),(u',v')) = {\alpha(u_i,u')}r(u_i,v_j)/{p(u_i)} $, as desired. An analogous argument verifies the other condition. Since $C_m$ and $C_n$ are connected, Proposition~\ref{connected_drop_coupling} ensures that the coupling condition is automatically satisfied.
    Therefore, $\gamma \in \mathcal{J}(\alpha, \beta)$, and since $\gamma \neq \alpha \otimes \beta$, we conclude that $C_m$ and $C_n$ are not strongly disjoint. 

By Theorem~\ref{weak_no_share_eigen}, $C_m$ and $C_n$ 
are weakly disjoint if and only if their transition matrices 
share no eigenvalues other than $1$. 
From the expression in \eqref{eigenvalue_circulant}, 
$C_m$ and $C_n$ share an eigenvalue other than $1$ if and only if 
there exists integers $k \in \{1,\cdots,m-1\}$ and 
$l \in \{1,\cdots,n-1\}$ such that
\begin{align*}
\cos\left(\frac{2k\pi}{m}\right) \, = \, \cos\left(\frac{2l\pi}{n}\right),
\end{align*}
which is equivalent to 
\begin{align}
        \frac{2k\pi}{m} \, = \, \frac{2l\pi}{n} \, \text{ or } \, \frac{2k\pi}{m} +\frac{2l\pi}{n} \, = \, 2\pi.
        \label{condition_common}
\end{align}
If $m$ and $n$ are co-prime, then no such integers $k$ and $l$ exist, and $C_m$ and $C_n$ are weakly disjoint.
Now suppose $m$ and $n$ are not co-prime, and let $r = \gcd(m,n) > 1$.  Then 
$k = m/r$ and $l = n/r$ are positive integers satisfying the conditions of \eqref{condition_common}, and in this case we see that $C_m$ and $C_n$ share an eigenvalue other than $1$ and are therefore not weakly disjoint. Thus $C_m$ and $C_n$ are weakly disjoint if and only if $\gcd(m,n)=1$.
\end{proof}



\subsection{Example~\ref{example_weak_c}}
\weakc*
\begin{proof}
Since $K_{2,2}$ is bipartite, Proposition~\ref{bipartite_strongdisjoint} gives that $K_{2,2}$ is not weakly disjoint from $P_2$.

Let $\alpha$ and $\beta$ denote the weight functions of $P_2$ and $K_{2,2}$, respectively. Now consider an arbitrary weight joining $\gamma\in\mathcal{J}(\alpha,\beta)$. By the transition coupling conditions, $\gamma$ satisfies
\begin{align*}
    &\gamma((u_0,v_0),(u_1,v_2)) \,=\, \gamma((u_0,v_0),(u_1,v_3)) \,=\, \gamma((u_0,v_1),(u_1,v_2)) \,=\, \gamma((u_0,v_1),(u_1,v_3)),\\
    &\gamma((u_0,v_2),(u_1,v_0)) \,=\, \gamma((u_0,v_2),(u_1,v_1)) \,=\, \gamma((u_0,v_3),(u_1,v_0)) \,=\, \gamma((u_0,v_3),(u_1,v_1)).
\end{align*}
Let $r$ denote the degree function of $\gamma$. 
Summing the equalities above, and noting that $r$ sums to one,
there exists some $x \in [0,1/4]$ such that
\begin{align*}
    &r(u_0,v_0) \,=\, r(u_0,v_1) \,=\, r(u_1,v_2) \,=\, r(u_1,v_3) \,=\, x,\\
    &r(u_0,v_2) \,=\, r(u_0,v_3) \,=\, r(u_1,v_0) \,=\, r(u_1,v_1) \,=\, \frac{1}{4} - x.
\end{align*}
Computing the cost $c(u_i,v_j) = i+j \; (\text{mod } 2)$ on $U\times V$ yields
\begin{align*}
    &c(u_0,v_0) \,=\, c(u_0,v_2) \,=\, c(u_1,v_1) \,=\, c(u_1,v_3) \,=\, 0,\\
    &c(u_0,v_1) \,=\, c(u_0,v_3) \,=\, c(u_1,v_0) \,=\, c(u_1,v_3) \,=\, 1.
\end{align*}
Therefore, the total weighted cost is
\begin{align*}
\sum_{(u,v) \in U \times V} c(u,v) \, r(u,v) 
\, = \, r(u_0,v_1) + r(u_0,v_3) + r(u_1,v_0) + r(u_1,v_3) 
\, = \, \frac{1}{2}.
\end{align*}
As this value is the same for every weight joining 
$\gamma\in\mathcal{J}(\alpha,\beta)$, we conclude that 
$P_2$ and $K_{2,2}$ are $c$-disjoint.
\end{proof}



\section{Proofs for Section~\ref{factors}}
\label{proof_of_factors}

Recall that for any graph $G = (U,\alpha)$, we denote the neighborhood of a vertex $u \in U$ by $N(u) \coloneqq \{u'\in U: \alpha(u,u') > 0\}$.

\subsection{Definition of Graph Factor and Connectedness}

\begin{prop}
    Suppose $H\0 | \0 G$ under $\phi$. 
    If $H$ is connected,  then condition (i) in Definition~\ref{graph_factor} follows directly from condition (ii).
    \label{factor_degree}
\end{prop}

\begin{proof}
    Assume that condition (ii) of Definition~\ref{graph_factor} holds. For $v \in V$, let
$$g(v) := q(v)^{-1} \sum_{u \in \phi^{-1}(v)} p(u).$$  By condition (ii), for any $v \in V$ and $v'\in {N}(v)$, we have
\begin{align*}
    \sum\limits_{u\in \phi^{-1}(v)}p(u) \,
    &= \, \sum\limits_{u\in \phi^{-1}(v)}\sum\limits_{u'\in\phi^{-1}(v')}\frac{\alpha(u,u')}{\beta(v,v')}q(v) \\
    &= \, \sum\limits_{u'\in \phi^{-1}(v')}\sum\limits_{u\in\phi^{-1}(v)}\frac{\alpha(u',u)}{\beta(v',v)}q(v) \\
    &= \, \sum\limits_{u'\in \phi^{-1}(v')} \frac{p(u')}{q(v')}q(v).
\end{align*}
Dividing both sides by $q(v)$ (which is positive since $H$ is connected), we find $g(v) = g(v')$. 
Since {$g(v) = g(v')$ for all neighboring pairs $(v,v')$ and} $H$ is connected, we conclude that $g(v)$ is equal to a constant $C$ 
for all $v \in V$. Thus
\begin{align*}
    1 \, = \, \sum\limits_{v\in V}\sum\limits_{u\in\phi^{-1}(v)}p(u) \, = \, C\sum\limits_{v\in V} q(v) = C,
\end{align*}
and then it follows from the definition of $g$ that $q(v) = \sum_{u \in \phi^{-1}(v)} p(u)$ for all $v \in V$.
\end{proof}

\subsection{Proposition~\ref{characterize_graph_joining}}
\characterizegraphjoining*
\begin{proof}
To establish necessity, suppose that $G \0 | \0 K$ under $\pi_1$ and that
$H \0 | \0 K$ under $\pi_2$.
Let $r$ be the degree function of $K$ and let $u \in U$ and $v \in V$ be such that $r(u,v) > 0$.  Then for $u' \in U$, 
\begin{align*}
\frac{1}{r(u,v)} \sum\limits_{v' \in V} \gamma((u,v),(u',v')) 
\, = \, 
\frac{1}{r(u,v)} \sum\limits_{(u',v')\in \pi_1^{-1}(u')} \gamma((u,v),(u',v')) \, = \, 
\frac{\alpha(u,u')}{p(u)},
\end{align*}
where the first equality is a consequence of the definition of the projection $\pi_1$, and the second equality is a consequence of our assumption that $G \0 | \0 K$ under $\pi_1$.
Similarly,
\begin{align*}
        \frac{1}{r(u,v)} \sum\limits_{u' \in U} \gamma((u,v),(u',v')) \, = \, \frac{1}{r(u,v)} \sum\limits_{(u',v')\in \pi_2^{-1}(v')} \gamma((u,v),(u',v')) \, = \, \frac{\beta(v,v')}{q(v)}.
\end{align*}
This verifies the transition coupling condition. The degree coupling condition is equivalent to constraint (i) in the definition of factors since
\begin{align*}
        &\sum\limits_{v \in V}r(u,v) \, = \,
        \sum\limits_{(u,v) \in \pi_1^{-1}(u)} r(u,v) \, = \, p(u),
\end{align*}
and similarly $\sum_{u \in U} r(u,v) = q(v)$. 
Hence $(U\times V,\gamma) \in \mathcal{J}(G,H)$. 
The proof of sufficiency, which involves reversing the argument above, 
is straightforward.
\end{proof}

\subsection{Proposition~\ref{prop:factprop}}
\propfactprop*
\begin{proof}
Let $G = (U, \alpha)$, $H = (V, \beta)$ and $K = (S, \eta)$ have degree functions $p$, $q$, and 
$z$, respectively.
Since $G \otimes H$ is a graph joining of $G$ and $H$, claim (1) follows from Proposition~\ref{characterize_graph_joining}.
To establish (2), let $f: V \to U$ and $g: S \to V$ be surjective maps such that $G \0 | \0 H$ under $f$ and 
$H \0 | \0 K$ under $g$, and define $h = f \circ g$.
We show that $G \0|\0 K$ under $h$.
First note that $h:S\to U$ is a surjective map since both $f$ and $g$ are surjective. It therefore remains to verify conditions (i) and (ii) in Definition~\ref{graph_factor}. For condition (ii), fix $u,u'\in U$ and $s\in h^{-1}(u)$.  Let us show that
\begin{align*}
    p(u)\sum_{s'\in h^{-1}(u')}\eta(s,s') = z(s) \alpha(u,u').
\end{align*}
By the assumption that $G\0|\0H$ under $f$, for any $v \in f^{-1}(u)$,
\begin{align}
    p(u) \sum\limits_{v' \in f^{-1}(u')}\beta(v,v') \, = \, q(v) \alpha(u,u').
    \label{GH}
\end{align}
Similarly, since $H\0|\0K$ under $g$, for any $s \in g^{-1}(v)$ and for any $v' \in f^{-1}(u')$, 
\begin{align}
    q(v) \sum\limits_{s' \in g^{-1}(v')}\eta(s,s') \, = \, z(s) \beta(v,v').
    \label{HK}
\end{align}
Assuming $z(s) > 0$, we substitute $\beta(v,v')$ in \eqref{GH} using \eqref{HK} to obtain
\begin{equation}
\begin{aligned}
    q(v) \alpha(u,u') \, &= \, p(u) \sum\limits_{v' \in f^{-1}(u')} \frac{q(v)}{z(s)}\sum\limits_{s' \in g^{-1}(v')}\eta(s,s') \\
    &= \, \frac{p(u)q(v)}{z(s)} \sum\limits_{v' \in f^{-1}(u')} \sum\limits_{s' \in g^{-1}(v')}\eta(s,s') \\
    &= \frac{p(u)q(v)}{z(s)} \sum\limits_{s' \in h^{-1}(u')} \eta(s,s').
    \label{GHK}
\end{aligned}
\end{equation}
Moreover, from our assumption that $H\0|\0K$ under $g$ and 
condition (i) in Definition~\ref{graph_factor}
\begin{align}
    q(v) = \sum\limits_{s \in g^{-1}(v)}z(s).
    \label{HK2}
\end{align}
Since we have assumed $z(s) > 0$, it follows that $q(v) > 0$. Dividing both sides of \eqref{GHK} by $q(v)$ yields
\begin{align}
z(s) \0 \alpha(u,u') \, = \, p(u) \sum\limits_{s' \in h^{-1}(u')} \eta(s,s').
\label{factor_i_check}
\end{align}
If instead $z(s) = 0$, then $\eta(s,s')=0$ for all $s'\in S$, and \eqref{factor_i_check} still holds trivially. This verifies condition (ii).
For condition (i), since $G\0|\0H$ under $f$, we have
\begin{align*}
    p(u) = \sum_{v\in f^{-1}(u)}q(v).
\end{align*}
Substituting $q(v)$ from \eqref{HK2}, we obtain
\begin{align*}
    p(u) \, = \, \sum\limits_{v \in f^{-1}(u)}\sum\limits_{s \in g^{-1}(v)} z(s) \, = \, \sum\limits_{s \in h^{-1}(u)} z(s).
\end{align*}
Thus condition (i) is also satisfied for $h$.
We therefore conclude that $G \0 | \0 K$ under $h$.

(3) As $G \0 | \0 H$ and $H \0 | \0 G$ there are surjective maps from $U$ to $V$
and from $V$ to $U$, so $|U| = |V|$. 
Let $f: V \to U$ be a factor map from $H$ to $G$. 
As $|U| = |V|$ and $f$ is surjective, $f$ is in fact bijective. 
The conditions in the definition of graph factor ensure that for all $u,u' \in U$, 
\begin{align*}
     p(u) \, & = \, q(f^{-1}(u)),
\end{align*}
and
\begin{align*}
     p(u) \0 \beta(f^{-1}(u),f^{-1}(u')) \, & = \, q(f^{-1}(u)) \0 \alpha(u,u'). 
\end{align*}
Note that if $p(u) = q(f^{-1}(u)) = 0$, then $\alpha(u,u') = 0$ and $\beta(f^{-1}(u),f^{-1}(u'))=0$ for all $u'\in U$.
Substituting the first identity into the second yields $\alpha(u, u') = \beta(f^{-1}(u),f^{-1}(u'))$ for all $u,u' \in U$.  Hence, $G$ and $H$ are isomorphic.
\end{proof}

\subsection{Proposition~\ref{furstenberg}}
\furstenbergresult*
\begin{proof}
Let $G = (U, \alpha)$, $H = (V, \beta)$, $K = (S, \eta)$ and $L = (T, \delta)$.
Let $f: S \to U$ be a factor map from $K$ to $G$, and let $g: T \to V$ be 
a factor map from $L$ to $H$. 
Let $\gamma$ be a weight
joining of $\alpha$ and $\beta$,
 and let $r_{\gamma}$ be the corresponding degree function.
Using $\gamma$ we may construct a weight joining between $K$ and $L$. For 
$(s,t),(s',t') \in S \times T$, define $\zeta: (S\times T)^2 \to \mathbb{R}$ by
\begin{align*}
        \zeta((s,t),(s',t')) \, = \, \frac{\gamma((f(s),g(t)),(f(s'),g(t')))}{\alpha(f(s),f(s')) \0 \beta(g(t),g(t'))} \0 \eta(s,s') \0 \delta(t,t'),
\end{align*}
with the convention that $0/0=0$. We claim that $\zeta \in \mathcal{J}(\eta,\delta)$.
Clearly $\zeta$ is non-negative and symmetric.
Let $p$, $q$, $k$, $l$ be the degree functions of $\alpha$, $\beta$, $\eta$ 
and $\delta$, respectively. As $H|L$ under $g$ and $\gamma \in \mathcal{J}(\alpha,\beta)$, 
we find that
    \begin{align*}
    \sum\limits_{t'\in T} \zeta((s,t),(s',t')) 
    &= \, \sum\limits_{v'\in V} \sum\limits_{t'\in g^{-1}(v')} \frac{\gamma((f(s),g(t)),(f(s'),v'))}{\alpha(f(s),f(s')) \0 \beta(g(t),v')} \0 \eta(s,s') \0 \delta(t,t')\\[.05in]
    &= \, \sum\limits_{v'\in V}\frac{\gamma((f(s),g(t)),(f(s'),v'))}{\alpha(f(s),f(s')) \0 \beta(g(t),v')} \0 \eta(s,s') \0 \sum\limits_{t'\in g^{-1}(v')} \delta(t,t')\\[.05in]
    &= \, \sum\limits_{v'\in V}\frac{\gamma((f(s),g(t)),(f(s'),v'))}{\alpha(f(s),f(s')) \0 \beta(g(t),v')} \0 \eta(s,s') \0
    \frac{\beta(g(t), v')}{q(g(t))} \0 l(t)\\[.05in]
    &= \, \frac{\eta(s,s') \0 l(t)}{\alpha(f(s),f(s')) \0 q(g(t))}
    \sum\limits_{v'\in V}\gamma((f(s),g(t)),(f(s'),v'))\\[.05in]
    &= \, \frac{\eta(s,s') \0 l(t)}{\alpha(f(s),f(s')) q(g\0 (t))} \frac{\alpha(f(s),f(s'))}{p(f(s))} \0
        r_{\gamma}(f(s),g(t))\\[.05in]
    &= \, \frac{\eta(s,s') \0 l(t)}{q(g(t)) \0 p(f(s)))} \0 r_{\gamma}(f(s),g(t)).
\vspace*{.07in}
\end{align*}
Summing over $s' \in S$ gives that the degree function of $\zeta$ takes the form
\begin{align*}
        r_{\zeta}(s,t) \, = \, \sum\limits_{s'\in S, t'\in T}\zeta((s,t),(s',t')) \, = \, \frac{k(s)l(t)}{q(g(t))p(f(s))}\0 r_{\gamma}(f(s),g(t)) .
\end{align*}
Then substituting this expression into the previous display yields
\begin{align*}
        \sum\limits_{t'\in T} \zeta((s,t),(s',t')) \, = \, \frac{\eta(s,s')}{k(s)}\0 r_{\zeta}(s,t),
    \end{align*}
which verifies the first transition coupling condition. The second may be
verified in a similar manner.

Condition (i) of the definition of a graph factor ensures that
$q(v) \, = \, \sum_{t \in g^{-1}(v)}l(t)$.  From this fact and the degree coupling
condition for $\gamma$, we find 
    \begin{align*}
        \sum\limits_{t \in T}r_{\zeta}(s,t) & 
        \, = \, \sum\limits_{v \in V} \sum\limits_{t \in g^{-1}(v)} 
        \frac{k(s)l(t)}{q(g(t))p(f(s))}\0 \, r_{\gamma}(f(s),g(t)) \\
        &= \, \sum\limits_{v \in V} \frac{k(s)}{q(v)p(f(s))}\0 r_{\gamma}(f(s),v)\sum\limits_{t\in g^{-1}(v)}l(t)\\
        &= \, \sum\limits_{v\in V} \frac{k(s)}{p(f(s))} \0 r_{\gamma}(f(s),v)\\
        &= \, \frac{k(s)}{p(f(s))}\0 p(f(s))
        \, = \, k(s).
    \end{align*}
This verifies the degree coupling and normalization property of $\zeta$, and we conclude
that $\zeta \in \mathcal{J}(\eta,\delta)$.

Now suppose that $G$ and $H$ are not strongly disjoint. Then there exists a weight joining $\gamma \neq \alpha \otimes \beta$ of $G$ and $H$. Thus, by construction, we have $\zeta \neq \eta \otimes \delta$,
and therefore $K$ and $L$ are not strongly disjoint. 

Next suppose that  $G$ and $H$ are not weakly disjoint. Then there exists a degree function $r_{\gamma} \neq p \otimes q$ for some weight joining $\gamma$. Then by the construction above, we have $r_{\zeta} \neq k \otimes l$,
and therefore $K$ and $L$ are not weakly disjoint.
\end{proof}

\subsection{Lemma~\ref{disconnected_not_weak}}
\disconnectednotweak*
\begin{proof}
Any disconnected graph admits a 
factor in which each of its connected components is mapped to one of 
two disconnected vertices equipped with a self-loop. 
The loop weights of each vertex, which we may assume to be positive, 
equal the total weight of the connected components mapped to that
vertex.  By Proposition~\ref{furstenberg}, it suffices to show that 
no pair of graphs, each consisting of two disconnected self-loops, 
is weakly disjoint. Let $(U,\alpha)$ and $(V,\beta)$ be two such graphs
with $U = \{u_1,u_2\}, V = \{v_1,v_2\}$, 
$\alpha(u_1,u_1) = 1-\alpha(u_2,u_2) \in (0,1)$, 
and $\beta(v_1,v_1) = 1- \beta(v_2,v_2) \in (0,1)$. First, one may readily check that
\begin{equation*}
    \max(0,\alpha(u_1,u_1)-\beta(v_2,v_2)) <\min(\alpha(u_1,u_1),\beta(v_1,v_1)).
\end{equation*}
Then we let $x \in (\max(0,\alpha(u_1,u_1)-\beta(v_2,v_2)),\min(\alpha(u_1,u_1),\beta(v_1,v_1))) \setminus \{ p(u_1)q(v_1)\}$.
Now we construct a nontrivial weight joining $\gamma$ of $\alpha$ and $\beta$ as follows:
    \begin{align*}
        &\gamma((u_1,v_1),(u_1,v_1)) = x,\\
        &\gamma((u_1,v_2),(u_1,v_2)) = \alpha(u_1,u_1) - x,\\
        &\gamma((u_2,v_1),(u_2,v_1)) = \beta(v_1,v_1) - x,\\
        &\gamma((u_2,v_2),(u_2,v_2)) = x + \beta(v_2,v_2) - \alpha(u_1,u_1),
    \end{align*}
with all other values of $\gamma$ equal to zero.
Our choice of $x$ ensures that the degree function $r$ of $\gamma$ satisfies
\[
        r(u_1,v_1) = \gamma((u_1,v_1),(u_1,v_1)) = x \neq  p(u_1) q(v_1),
\]
and therefore $(U,\alpha)$ and $(V,\beta)$ are not weakly disjoint. 
\end{proof}

\subsection{Proposition~\ref{factor_times_factor_strong}}
\factortimesfactorstrong*
\begin{proof}

Let $G = (U, \alpha)$ and $H = (V, \beta)$.
For any weight joining $\gamma\in \mathcal{J}(\alpha,\beta)$, by the edge preservation property, we have $\supp(\gamma)\subseteq \supp(\alpha\otimes\beta)$. Define $K = (U\times V,\gamma)$. Then $K \in \mathcal{J}(G,H)$, and by Proposition~\ref{characterize_graph_joining}, we get $G\0 | \0 K$ and $H\0 | \0 K$. By hypothesis, this yields $G \otimes H \0 | \0 K$. 
Since $G \otimes H \0|\0 K$, {there exists a factor map $f:U\times V\to U\times V$ such that $G\otimes H\0|\0 K$ under $f$. Because $f$ is surjective, it must in fact be bijective. Since $f$ is a bijective graph factor, 
$$(\alpha\otimes\beta)((f(x),f(y)),(f(x'),f(y'))) = \gamma((x,y),(x',y')) \quad \forall ((x,y),(x',y')) \in (U\times V)^2.$$
Consequently, $|\supp(\alpha\otimes\beta)| = |\supp(\gamma)|$. In conjunction with the fact that $\supp(\gamma)\subseteq \supp(\alpha\otimes \beta)$,} we obtain $\supp(\alpha\otimes\beta)= \supp(\gamma)$. 
Since $\gamma \in \mathcal{J}(\alpha,\beta)$ was arbitrary, we have shown that  every weight joining $\gamma\in \mathcal{J}(\alpha,\beta)$ has the same support as $\alpha\otimes \beta$.

 Now suppose for contradiction that $G$ and $H$ 
are not strongly disjoint. We will show that this contradicts the same-support property established in the previous paragraph. Since $G$ and $H$ are not strongly disjoint, there exists a weight joining $\gamma \in \mathcal{J}(\alpha,\beta)$ with $\gamma\neq \alpha\otimes \beta$. 
Since $\gamma \neq \alpha\otimes \beta$ and $\supp(\gamma) = \supp(\alpha\otimes \beta)$ (and since the weight functions are normalized), there exists 
$((x,y),(x',y'))\in \supp(\gamma)$ such that $\alpha(x,x')\beta(y,y') > \gamma((x,y),(x',y'))$.
Now consider the function $(1+t)\gamma-t\alpha\otimes\beta$, where 
\begin{align*}
    t\coloneqq \min _{\substack{((x, y),(x^{\prime}, y^{\prime})) \in \supp(\gamma) \\ \alpha(x, x^{\prime}) \beta(y, y^{\prime})>\gamma((x, y),(x^{\prime}, y^{\prime}))}} \frac{\gamma((x, y),(x^{\prime}, y^{\prime}))}{\alpha\left(x, x^{\prime}\right) \beta(y, y^{\prime})-\gamma((x, y),(x^{\prime}, y^{\prime}))} > 0 .
\end{align*}
Let
$((u,v),(u',v'))\in \supp(\gamma)$ be a pair attaining the above minimum. Then
$$t \,=\, \frac{\gamma((u,v),(u',v'))}{\alpha(u,u')\beta(v,v')-\gamma((u,v),(u',v'))},$$ which yields $(1+t)\gamma((u,v),(u',v')) - t\alpha(u,u')\beta(v,v') = 0$.
By the choice of $t$, we have that $(1+t)\gamma-t\alpha\otimes\beta$ is nonnegative, and therefore this function is a valid weight joining of
$\alpha$ and $\beta$. 
Since $((u,v),(u',v'))\in\supp(\gamma) = \supp(\alpha\otimes \beta)$, we get $\supp((1+t)\gamma-t\alpha\otimes\beta) \subsetneq \supp(\alpha\otimes \beta)$, contradicting the conclusion that every weight joining in $\mathcal{J}(\alpha,\beta)$ must have support equal to $\supp(\alpha\otimes\beta)$.
Therefore, $G$ and $H$ are strongly disjoint.
\end{proof}

\subsection{Proposition~\ref{factor_to_weak_disjoint}}
\factortoweak*
\begin{proof}[Proof of Proposition~\ref{factor_to_weak_disjoint}]
Since $K$ is a non-trivial graph and is isomorphic to itself (by the identity map), it is not weakly disjoint from itself (see the weight joining constructed in Section~\ref{isomorphic_graphs}). Under the assumptions $K \0 | \0 G$ and $K \0 | \0 H$, Proposition~\ref{furstenberg} gives that $G$ and $H$ are not weakly disjoint.
\end{proof}

\subsection{Proposition~\ref{no_factor_not_weak}}
\nofactornotweak*
\begin{proof}
Let $G$ and $H$ be the graphs shown in Figure~\ref{graph_joining_example} (D), 
and let $K$ be the graph joining illustrated in the same figure.  As $K$ has vertices with zero weight, but $G$ and $H$ do not, it is clear that 
$G$ and $H$ are not weakly disjoint.

Now suppose for contradiction that $G$ and $H$ share a non-trivial factor $F$. 
As $F$ has no more vertices than $G$ or $H$,
and as single-vertex graphs are trivial, $F$ must have exactly two vertices.
As $F$ and $H$ have the same number of vertices and 
their associated factor map is surjective, $F$ and $H$ are isomorphic, 
and it follows that $H$ must be a factor of $G$.
However, examining the six possible surjective maps
from the three vertices of $G$ to the two vertices of $H$, we find 
that none satisfies the condition in Definition~\ref{graph_factor}.
We therefore conclude by contradiction 
that $G$ and $H$ have no non-trivial common factor. 
\end{proof}
\begin{rmk}
While the example in the proof above has self-loops, 
one may construct similar examples without self-loops.
\end{rmk}

\subsection{Proposition~\ref{remark_on_circle}}
\remarkoncircle*
\begin{proof}
In the proof, let $C_m = (U,\alpha)$ and $C_n = (V,\beta)$ with $U = \{u_1,\cdots,u_m\}$ and $V = \{v_1,\cdots,v_n\}$.

    Suppose $C_m$ and $C_n$ share a non-trivial common graph factor. By Proposition~\ref{factor_to_weak_disjoint}, they are not weakly disjoint.  
    Conversely, suppose $C_m$ and $C_n$ are not weakly disjoint. Then by Proposition~\ref{circle_property}, $\gcd(m,n)\neq 1$.
    Let $r = \gcd(m,n)>1$. Let us now show that $C_r$ is a common factor of both $C_m$ and $C_n$. 
    For notation, suppose $C_r = (S,\eta)$, where $S = \{s_0,\cdots,s_{r-1}\}$. Define a map $f: U\to S$ by $f(u_i) = s_k$ whenever $i\equiv k \mod r$. Similarly, define $g: V\to S$ by $g(v_j)=s_k$ whenever $j\equiv k \mod r$. It is straightforward to verify that $C_r \0 | \0 C_m$ under $f$ and $C_r \0 | \0 C_n$ under $g$. Consequently, $C_m$ and $C_n$ share a non-trivial common graph factor $C_r$. 

    Now suppose $C_m$ and $P_n$ share a common graph factor. Then by Proposition~\ref{factor_to_weak_disjoint} the graphs $C_m$ and $P_n$ are not weakly disjoint, and hence they are not strongly disjoint. 
    Conversely, suppose $C_m$ and $P_n$ are not strongly disjoint. By Proposition~\ref{line_circle} either $m$ is even or $\gcd(m,n-1)>1$. If $m$ is even, then $C_m$ and $P_n$ share the common graph factor $P_2$. Now suppose instead that $m$ is odd and $\gcd(m,n-1) = r>1$. In this case, $r \geq 3$ and $r$ must be odd. We claim that $C_m$ and $P_n$ share a common factor $K = (S,\eta)$, where $S = \{s_1,\cdots,s_{(r+1)/2}\}$ and
    \begin{align*}
        \eta(s_i,s_j) \, = \, \begin{cases}
            \frac{1}{r} &\text{if } |i-j| = 1,\\
            \frac{1}{r} &\text{if }i = j = \frac{r+1}{2},\\
            0 &\text{otherwise.}
        \end{cases}
    \end{align*}
    Define a map $f: U \to S$ by 
    \begin{align*}
        f(u_i) \, = \, \begin{cases}
            s_k &\text{if } i \equiv k \text{ (mod }r), \text{ and }1\leq k \leq \frac{r+1}{2}, \\
            s_{r+2-k} &\text{if } i \equiv k \text{ (mod }r), \text{ and }\frac{r+1}{2}\leq k \leq r-1,\\
            s_2 &\text{if }i\equiv 0 \text{ (mod }r).
        \end{cases}
    \end{align*}
    Define $g: V \to S$ analogously by
    \begin{align*}
        g(v_j) \, = \, \begin{cases}
            s_k &\text{if } j \equiv k \text{ (mod }r), \text{ and }1\leq k \leq \frac{r+1}{2}, \\
            s_{r+2-k} &\text{if } j \equiv k \text{ (mod }r), \text{ and }\frac{r+3}{2}\leq k \leq r-1,\\
            s_2 &\text{if }j\equiv 0 \text{ (mod }r).
        \end{cases}
    \end{align*}
    Since $r \0|\0 m$ and $r \0|\0 (n-1)$, it is straightforward to verify that $K \0|\0 C_m$ under $f$ and $K \0|\0 P_n$ under $g$. Hence, $C_m$ and $P_n$ share a common non-trivial graph factor. 
\end{proof}

\vskip.2in

\section{Proofs for Section~\ref{Characterizing}}
\label{proof_of_characterizing}
\subsection{Theorem~\ref{weak_no_share_eigen}}
\weaknoshareeigen*
\begin{proof}
Let $G = (U,\alpha)$ and $H = (V,\beta)$ be connected. Consider any weight joining $\gamma$ such that $(U\times V,\gamma) \in \mathcal{J}(G,H)$.
By summing over the transition coupling conditions in the definition of a 
weight joining (Definition~\ref{Def:weight_joining}), we find that any weight joining $\gamma$ of $\alpha$ and $\beta$ 
satisfies
    \begin{align*}
        &\sum\limits_{u\in U}\sum\limits_{v'\in V}\gamma((u,v),(u',v')) \, = \, \sum\limits_{u\in U}\frac{\alpha(u,u')}{p(u)}r(u,v), \text{ and }\\
        &\sum\limits_{v'\in V}\sum\limits_{u\in U}\gamma((u,v),(u',v')) \, = \, \sum\limits_{v'\in V}\frac{\beta(v',v)}{q(v')}r(u',v'),
\end{align*}
for all $u' \in U$ and $v \in V$.
Since the left-hand sides of both expressions are equal, it follows that
        \begin{align}
            \sum\limits_{u\in U}\frac{\alpha(u,u')}{p(u)} \0 r(u,v) \, = \, \sum\limits_{v'\in V}\frac{\beta(v',v)}{q(v')} \0 r(u',v'),
        \label{suff_condition}
        \end{align}
for all $u'\in U$ and $v\in V$.
Let $U = \{u_1,\cdots,u_m\}$, $V=\{v_1,\cdots,v_n\}$, and define the following matrices 
\begin{align*}
        & W_1 \, = \,  [\0 \alpha(u_i,u_j) \0]_{i,j=1}^m \in \mathbb{R}^{m\times m}, 
        \quad W_2 \, = \, [\0 \beta(v_i,v_j) \0]_{i,j=1}^n \in \mathbb{R}^{n\times n}, \\[.1in]
        & D_1 \, = \, \diag(p(u_1),\cdots,p(u_m)), \quad D_2 \, = \, \diag(q(v_1),\cdots,q(v_n)), 
        \quad \Gamma \, = \, [\0 r(u_i,v_j) \0]{_{i=1}^m} _{j=1}^n.
\end{align*} 
Since $G$ and $H$ are connected, the diagonal entries of $D_1$ and $D_2$ are strictly positive.  Let
\begin{align*}
P \, = \, D_1^{-1}W_1 
\, = \, \left[\frac{\alpha(u_j,u_i)}{p(u_i)}\right]_{i,j=1}^m, 
\quad Q \, = \, D_2^{-1}W_2 \, = \, \left[\frac{\beta(v_j,v_i)}{q(v_i)}\right]_{i,j=1}^n.
\end{align*}
Note that $P$ and $Q$ are the transition matrices of the random walks on the graphs $G$ and $H$, respectively, and that $P$ and $P^T$ share the same spectrum.
Equation~\eqref{suff_condition} can be therefore be written as $P^T \Gamma = \Gamma Q$.
It is well-known that the number of connected components in an undirected graph is equal
to the multiplicity of the eigenvalue 1 in the normalized adjacency matrix of the graph
(see Lemma 1.17  in \cite{chung1997spectral}). Moreover, the normalized adjacency matrix $D_1^{-1/2}W_1D_1^{-1/2}$ shares the same spectrum as the transition matrix $D_1^{-1}W_1$.
As $G$ and $H$ are connected, $P^T$ and $Q$ each have a unique
eigenvalue equal to 1. 
In particular, $P^T \bar{p} = \bar{p}$ and 
$\bar{q}^T Q = \bar{q}^T$, where 
\begin{align*}
\bar{p} \, = \, \left( p(u_1),\cdots,p(u_m) \right)^T
\ \ \mbox{and} \ \ 
\bar{q} \, = \, \left( q(v_1),\cdots,q(v_n) \right)^T.
\end{align*}
It is immediate that $\Gamma = \bar{p} \0 \bar{q}^T$ satisfies the matrix equation 
\begin{equation} \label{Eqn:Matrices}
    P^T \Gamma = \Gamma Q.
\end{equation}

Suppose now that 1 is the only eigenvalue shared by $P^T$ and $Q$. 
If $y \in \mathbb{R}^n$ is any right eigenvector of $Q$ with eigenvalue 
$\lambda \neq 1$, then (\ref{Eqn:Matrices}) gives
\[
P^T \Gamma y \, = \, \Gamma Q y \, = \, \lambda \Gamma y.
\] 
As $\lambda$ is not an eigenvalue of $P^T$, it follows that $\Gamma y = 0$. 
Thus $\Gamma y = 0$ for all right eigenvectors $y$ of $Q$ except the 
eigenvector corresponding to eigenvalue 1. 
Similarly, $x^T \Gamma = 0$ for all left eigenvectors $x$ of $P^T$, except the  eigenvector corresponding to eigenvalue 1. 

Since the random walks on $G$ and $H$ are reversible, the transition matrices $P$ and $Q$ 
are diagonalizable (see Lemma 12.2 \cite{levin2017markov}), 
and as such their eigenvectors form bases of $\mathbb{R}^m$ and 
$\mathbb{R}^n$, respectively.
The arguments above show that 
$\dim(\mbox{Null}(\Gamma)) \geq n-1$, and therefore
$\rank(\Gamma) = n - \dim(\mbox{Null}(\Gamma)) \leq 1$. 
As $\Gamma$ is not the zero matrix, we must have $\rank(\Gamma) = 1$, 
which implies that $\Gamma = a b^T$ for some nonzero vectors $a,b$. 
As $r$ is a coupling of $p$ and $q$, the definition of $\Gamma$ ensures that 
$\Gamma\mathbbm{1}_n = \bar{p}$ and $\mathbbm{1}_m^T\Gamma = \bar{q}^T$.  
It follows that $a = c \bar{p}$ and $b = {\bar{q}}/{c}$ for some scalar $c \neq 0$. 
Thus $\Gamma = \bar{p} \0 \bar{q}^T$ is the unique solution of 
$P^T \Gamma = \Gamma Q$ {with $\Gamma\mathbbm{1}_n = \bar{p}$ and 
$\mathbbm{1}_m^T\Gamma = \bar{q}^T$}, and we conclude that $G$ and $H$ are weakly disjoint.
    
Now suppose that $P^T$ and $Q$ share an 
eigenvalue $\lambda \neq 1$.  Let $x \in \mathbb{R}^m$ and 
$y \in \mathbb{R}^n$ be nonzero vectors 
such that $P^T x = \lambda x$ and 
$y^T Q = \lambda y^T$. 
Then $\Gamma = x y^T$ satisfies $P^T \Gamma = \Gamma Q$. 
Note that $\mathbbm{1}_m^T P^T = \mathbbm{1}_m^T$, and therefore
\begin{align*}
\mathbbm{1}_m^T x \, = \, \mathbbm{1}_m^T P^T x \, = \, \lambda \mathbbm{1}_m^T x.
\end{align*}
As $\lambda \neq 1$, it follows that $\mathbbm{1}_m^T x = 0$. A similar argument 
shows $y^T \mathbbm{1}_n = 0$. 

First consider the case $\lambda = -1$. Then both $P$ and $Q$ have $-1$ as an eigenvalue. It is well known that for a connected, undirected graph, if the transition matrix has eigenvalue $-1$, then the graph is bipartite (see Lemma 12.1 (iii) in \cite{levin2017markov}, which states that for an irreducible reversible Markov chain, if -1 is an eigenvalue of the transition matrix, then the chain is periodic). Therefore, when $\lambda = -1$, both $G$ and $H$ are bipartite graphs. Then both $G$ and $H$ share a non-trivial graph factor path graph $P_2$.  By Proposition~\ref{factor_to_weak_disjoint}, $G$ and $H$ are not weakly disjoint.

Now suppose $\lambda\neq \pm 1$. Using $\lambda$, $x$, and $y$, let us construct a weight joining of $\alpha$ and $\beta$ whose degree 
function does not equal $p \otimes q$. 
Let ${x}_i$ and ${y}_j$ denote the $i$th and $j$th entries of the vectors $x$ and ${y}$. 
For $i,k \in \{ 1,\cdots,m \}$ and $j,l \in \{ 1,\cdots,n \}$, {and $t \in \mathbb{R}$}, define
\begin{align*}
\gamma_t((u_i,v_j),(u_k,v_l)) \, = \, &\alpha(u_i,u_k) \beta(v_j,v_l)\\
&*\left[1+\frac{t}{1-\lambda^2}\left( \frac{{x}_i {y}_j}{p(u_i)q(v_j)}+\frac{{x}_k {y}_l}{p(u_k)q(v_l)} -\lambda\frac{{x}_k {y}_j}{p(u_k)q(v_j)}-\lambda\frac{{x}_i {y}_l}{p(u_i)q(v_l)} \right)\right].
\end{align*}
Note that there exists $t\neq 0$ such that, for all $i,k\in\{1,\cdots,m\}$ and $j,l\in \{1,\cdots,n\}$,
\begin{align*}
    1+\frac{t}{1-\lambda^2}\left( \frac{{x}_i {y}_j}{p(u_i)q(v_j)}+\frac{{x}_k {y}_l}{p(u_k)q(v_l)} -\lambda\frac{{x}_k {y}_j}{p(u_k)q(v_j)}-\lambda\frac{{x}_i {y}_l}{p(u_i)q(v_l)} \right) > 0.
\end{align*}
We fix such a value of $t$ for the remainder of the proof. Since $\alpha(u_iu_k)\beta(v_j,v_l)\geq 0$, it follows that $\gamma_t$ is nonnegative.
We claim that $\gamma_t$ is a valid weight joining 
of $\alpha$ and $\beta$ with degree function $r_t \neq p\otimes q$. 
It is easy to verify that $\gamma_t$ is symmetric. 
{The normalization property will be checked later using degree functions.} 
Since both $G$ and $H$ are connected, it suffices to verify the transition coupling condition. 
To begin, note that 
\begin{align*}
\sum\limits_{k=1}^m\frac{\alpha(u_i,u_k)}{p(u_k)}{x}_k 
\,  & = \, 
(P^T {x})_i 
\, = \, 
(\lambda {x})_i 
\, = \, \lambda {x}_i, 
 \\
\sum\limits_{l=1}^n\frac{\beta(v_j,v_l)}{q(v_l)}{y}_l 
\, & = \, 
({y}^T Q)_{j} 
\, = \, 
(\lambda {y}^T)_j 
\, = \, 
\lambda {y}_j.
\end{align*}
From the second identity and the definition of $\gamma_t$ we find 
\begin{align*}
\sum_{l=1}^n \gamma_t((u_i,v_j),(u_k,v_l)) \, 
&= \, 
    \alpha(u_i,u_k)\left[q(v_j) + \frac{t}{1-\lambda^2}  \left(\frac{x_iy_j}{p(u_i)} - \lambda\frac{x_ky_j}{p(u_k)} + \left(\frac{x_k}{p(u_k)}-\lambda\frac{x_j}{p(u_i)}\right)\lambda y_j\right)\right] \\
    &= \, \alpha(u_i,u_k)\left[q(v_j) + \frac{t}{1-\lambda^2}(1-\lambda^2) \frac{x_iy_j}{p(u_i)}  \right] \\
    &= \, \alpha(u_i,u_k) \left[q(v_j) + t\frac{x_iy_j}{p(u_i)} \right].
\end{align*}
Similarly, using the first identity, we get
\begin{align*}
    \sum_{k=1}^m \gamma_t((u_i,v_j),(u_k,v_l)) \, = \, \beta(v_j,v_l) \left[p(u_i) + t\frac{x_iy_j}{q(v_j)}\right].
\end{align*}
The degree function of $\gamma_t$ equals
\begin{align*}
    r_t(u_i,v_j) \,=\, \sum_{k=1}^m\sum_{l=1}^n \gamma_t((u_i,v_j),(u_k,v_l)) \,=\, \sum_{k=1}^m \alpha(u_i,u_k) \left[q(v_j) + t\frac{x_iy_j}{p(u_i)} \right] \,=\, p(u_i)q(v_j) + tx_iy_j.
\end{align*}
Therefore
\begin{align*}
    &\sum_{l=1}^n \gamma_t((u_i,v_j),(u_k,v_l)) \, = \, \frac{\alpha(u_i,u_k)}{p(u_i)}r_t(u_i,v_j),\\
    &\sum_{k=1}^m \gamma_t((u_i,v_j),(u_k,v_l)) \, = \, \frac{\beta(v_j,v_l)}{q(v_j)}r_t(u_i,v_j).
\end{align*}
It remains to check the normalization property. 
Since $\mathbbm{1}_m^T x = 0$ and $y^T \mathbbm{1}_n = 0$, 
\begin{align*}
    \sum_{i=1}^m \sum_{j=1}^n r_t(u_i,v_j) \,=\, \sum_{i=1}^m \sum_{j=1}^n \left(p(u_i)q(v_j) + tx_iy_j\right) \,=\, 1 + t\sum_{i=1}^m x_i \sum_{j=1}^n y_j \,=\, 1.
\end{align*}
Now we have shown that $\gamma_t \in \mathcal{J}(\alpha,\beta)$, with the degree function $r_t(u_i,v_j) = p(u_i)q(v_j)+tx_iy_j$. Since $t\neq 0$ and the vectors $x,y\neq 0$, {there exists $i,j$ such that $x_i\neq 0$ and $y_j\neq 0$. For such a pair, we have $r_t(u_i,v_j) \neq p(u_i)q(v_j)$}, and we conclude that $G$ and $H$ are not weakly disjoint.
\end{proof}

\subsection{Proposition~\ref{weak_no_share_eigen_general}}
\weaknoshareeigengeneral*
\begin{proof}
{We distinguish cases according to the connectedness of $G$ and $H$.
In case} both $G$ and $H$ are connected, the result {is immediate from} Theorem~\ref{weak_no_share_eigen}.

{Now consider the case that} both $G$ and $H$ are disconnected. Then
they are not weakly disjoint, as noted in Lemma~\ref{disconnected_not_weak}. 
{For undirected graphs, the algebraic multiplicity of the eigenvalue 1 of the transition matrix equals the number of connected components of the graph (Lemma 1.17 \cite{chung1997spectral}). Since both $G$ and $H$ are disconnected, each has at least two connected components. Consequently, both of their transition matrices have the eigenvalue 1 with algebraic multiplicity at least two. Therefore, the transition matrices share the eigenvalue 1 with overlapping multiplicity at least two.}

On the other hand, suppose that $G$ is connected and $H$ is disconnected. 
We claim that $G$ and $H$ are weakly disjoint
if and only if $G$ is weakly disjoint from each connected component of $H$.

Write $G = (U,\alpha)$ and decompose $H = (V,\beta)$ as $V = V_1\sqcup\cdots\sqcup V_k$, where $\beta(v,v') = 0$ whenever $v\in V_i,v'\in V_j$, and $i \neq j$. For each $i$, let $H_i = (V_i,\beta_i)$ be the induced connected graph with normalized weights $\beta_i(v,v') \coloneqq \frac{\beta(v,v')}{s_i}$ for $v,v'\in V_i$ where $s_i \coloneqq \sum_{x,x'\in V_i}\beta(x,x')$. If $G$ is not weakly disjoint from some $H_i$, then there exists a weight joining $\gamma_i\neq \alpha\otimes\beta_i$ of $\alpha$ and $\beta_i$. Define $\gamma \coloneqq \frac{1}{s_i}\gamma_i + \sum_{j\neq i}\frac{1}{s_j}\alpha\otimes\beta_j$. It is straightforward to check that $\gamma$ is a weight joining of $\alpha$ and $\beta$, and since $\gamma_i\neq \alpha\otimes\beta_i$, we have $\gamma\neq \alpha\otimes\beta$. Hence, $G$ and $H$ are not weakly disjoint. 
Conversely, suppose $G$ and $H$ are not weakly disjoint, so there exists a weight joining $\gamma\neq \alpha\otimes\beta$. For each $i$, define $\gamma_i\coloneqq \frac{1}{t_i}\gamma|_{U\times V_i}$, where $t_i \coloneqq \sum_{u,u'\in U,v,v'\in V_i}\gamma((u,u'),(v,v')) = \sum_{v,v'\in V_i}\beta(v,v') = s_i$. Then each $\gamma_i$ is a weight joining of $\alpha$ and $\beta_i$, and since $\gamma\neq \alpha\otimes\beta$, there exists at least one $i$ such that $\gamma_i\neq \alpha\otimes\beta_i$. Thus, $G$ is not weakly disjoint from that component $H_i$.

Therefore, $G$ and $H$ are not weakly disjoint if and only if $G$ is not weakly disjoint from at least one connected component of $H$. Equivalently, $G$ and $H$ share an eigenvalue other than 1, since the spectrum of each connected component of $H$ is contained in the spectrum of $H$.

An analogous argument applies when $H$ is connected and $G$ is disconnected.
\end{proof}

\vskip.1in

\subsection{Joining Constraint Matrix}
\label{weight_joining_constraint_matrix}


Consider two weighted, undirected graphs $G = (U, \alpha)$ and $H = (V, \beta)$. 
Let $E(G)$ and $E(H)$ be the edge sets of $G$ and $H$, respectively. 
By definition, a weight joining $\gamma$ of $\alpha$ and $\beta$ 
is a function on $(U\times V) \times (U\times V)$ and is therefore determined by $|U|^2|V|^2$ values.
However, many of these values are necessarily zero as a consequence of the edge-preservation property, which implies that $\supp(\gamma) \subseteq \supp(\alpha \otimes \beta)$. 
In the linear programming formulation that follows, variables correspond to the values 
of weight joinings on $\supp(\alpha \otimes \beta)$. 
For each weight joining $\gamma$ of $\alpha$ and $\beta$ 
define a corresponding vector $\Vec{\gamma} \in \mathbb{R}^{k}$, with 
$k = |\supp(\alpha \otimes \beta)| = |E(G)| |E(H)|$, whose 
entries correspond to the values of $\gamma((u,v),(u',v'))$ for pairs 
$((u, v), (u', v')) \in \supp(\alpha \otimes \beta)$. 
The conditions ensuring that a vector $\Vec{\gamma}$ is a valid weight joining may be 
expressed in matrix form as
\begin{align*}
    J\Vec{\gamma} \, = \, 0, \quad \mathbbm{1}^T \Vec{\gamma} \, = \, 1, \quad \Vec{\gamma} \, \geq \, 0,
\end{align*}
where $J$ is a matrix capturing the following structural constraints:
\begin{itemize}

\item symmetry: $\gamma((u,v),(u',v')) = \gamma((u',v'),(u,v))$ for $(u,u') \in E(G)$ and $(v,v') \in E(H)$;

\item degree coupling constraints for all vertices;

\item transition coupling constraints for all vertices.
    \end{itemize}
The additional constraints $\mathbbm{1}^T \Vec{\gamma} = 1$ and $\Vec{\gamma} \geq 0$ ensure that the solution $\Vec{\gamma}$ is normalized and nonnegative.  
There is a one-to-one correspondence between vectors $\Vec{\gamma}$ 
satisfying these constraints and valid weight joinings $\gamma$ of $\alpha$ and $\beta$. 

In more detail, each symmetry constraint is enforced by a row in the constraint matrix 
$J$ such that
$$(0,\cdots,0,1,0,\cdots,0,-1,\cdots,0) \, \Vec{\gamma} \, = \, 0,$$
where $+1$ appears in the column corresponding to $((u,v),(u',v'))$ and 
$-1$ in the column corresponding to $((u',v'),(u,v))$), with all other entries zero.
The degree coupling constraints take the form
$$\sum\limits_{v\in V}r(u,v) \, = \, p(u),$$ 
for each $u \in U$, where $r(u,v) = \sum_{(u',v')\in U\times V} \gamma((u,v),(u',v'))$. Using the normalization condition $\sum_{(u,v)\in U\times V}r(u,v) =1$ imposed by $\mathbbm{1}^T \Vec{\gamma} = 1$, we can express the degree coupling as
    $$\sum\limits_{v\in V}\sum\limits_{(u',v')\in U\times V} \gamma((u,v),(u',v')) - p(u)\sum\limits_{(u,v)\in U\times V}\sum\limits_{(u',v')\in U\times V} \gamma((u,v),(u',v')) \, = \, 0.$$ 
    These equality constraints can be expressed in matrix form. Each such constraint corresponds to a row of matrix $J$. 
    A similar set of constraints applies for each $v \in V$, ensuring that $\sum_{u \in U}r(u,v) = q(v)$.
    
    For the transition coupling constraints, for each fixed $(u,v) \in U \times V$, and for all $u' \in U$ and $v' \in V$, we impose
    \begin{align*}
        &p(u)\sum\limits_{y \in V}\gamma((u,v),(u',y))-{\alpha(u,u')}\sum\limits_{(x,y)\in U\times V} \gamma((u,v),(x,y)) \, = \, 0, \\
        &q(v)\sum\limits_{x \in U}\gamma((u,v),(x,v'))-{\beta(v,v')}\sum\limits_{(x,y)\in U\times V} \gamma((u,v),(x,y)) \, = \, 0.
    \end{align*}
    Each such constraint is again represented by a row of the matrix $J$, {where the entries of $J$ correspond to the coefficients of $\gamma$ in the equations above.} 
    We refer to $J$ as the \textbf{joining constraint matrix} between the graphs $G = (U, \alpha)$ and $H=(V, \beta)$. This matrix encodes all necessary linear constraints that characterize the set $\mathcal{J}(\alpha,\beta)$ of admissible weight joinings. 
    

\subsection{Proposition~\ref{rank_equiv_strong}}
\rankequivstrong*
\begin{proof}
Let $J = J(G,H)$ be the joining constraint matrix between $G$ and $H$.
The problem of identifying the weight joinings in $\mathcal{J}(\alpha,\beta)$ now reduces to solving
\begin{equation}  
\begin{aligned}
\label{feasible}
    J\Vec{\gamma} & \, = \, 0, \\
    \Vec{\gamma} & \, \geq \, 0, \\
    \mathbbm{1}^T \Vec{\gamma} & \, = \, 1.
\end{aligned}
\end{equation}
As established in Section~\ref{weight_joining_constraint_matrix}, there exists a one-to-one correspondence between joinings in $\mathcal{J}(\alpha,\beta)$ and solutions to the system \eqref{feasible}.
Let $\vec{\gamma}_{\text{product}}$ denote the column vector representation of the 
product weight joining $\alpha \otimes \beta$. It is easy to verify that
$\vec{\gamma}_{\text{product}}$ satisfies the constraints in \eqref{feasible}. 
In particular, since each entry of $\vec{\gamma}_{\text{product}}$ is the product of two strictly positive edge weights, $\vec{\gamma}_{\text{product}} > 0$. 

In order to determine whether $G$ and $H$ are strongly disjoint, it suffices to check whether 
there exists a vector $\vec{\gamma}' \notin \spa(\vec{\gamma}_{\text{product}})$ 
satisfying $J\vec{\gamma}'=0$ and $\mathbbm{1}^T\vec{\gamma}'=1$. 
Indeed, if such a vector $\vec{\gamma}'$ exists, then for each $t \in \mathbb{R}$, the vector
\[
{\vec{\gamma}}^* \, = \, (1+t)\vec{\gamma}_{\text{product}}-t\vec{\gamma}'
\]
satisfies $J {\vec{\gamma}}^* = 0$ and $\mathbbm{1}^T {\vec{\gamma}}^* = 0$, and 
as $\vec{\gamma}_{\text{product}} > 0$, the vector $\vec{\gamma}^*$ will have 
strictly positive entries for $t$ sufficiently close to 0, and hence satisfy all conditions in~\eqref{feasible}. Moreover, $\vec{\gamma}^*\neq \vec{\gamma}_{\text{product}}$ 
as $\Vec{\gamma}' \notin \spa({\vec{\gamma}_{\text{product}}})$. 
Thus the existence of $\vec{\gamma}'$ implies that $\vec{\gamma}_{\text{product}}$ 
is not the unique solution to~\eqref{feasible}, and therefore $G$ and $H$ are 
not strongly disjoint. 
Conversely, if $G$ and $H$ are not strongly disjoint, then there exists a solution to \eqref{feasible} distinct from $\vec{\gamma}_{\text{product}}$, 
which implies the existence of a vector $\vec{\gamma}' \notin \spa(\vec{\gamma}_{\text{product}})$ satisfying the stated conditions. Note that that by the normalization condition, $\vec{\gamma}_{\text{product}}$ is the only vector in its span that corresponds to a 
solution of \eqref{feasible}.

To determine whether $G$ and $H$ are strongly disjoint, we analyze the null space $\mbox{Null}(J)$. Since $\vec{\gamma}_{\text{product}}\in \mbox{Null}(J)$ and $\vec{\gamma}_{\text{product}}\neq 0$, it follows that 
 $\dim(\mbox{Null}(J))\geq 1$. Here we consider two cases:
    \begin{itemize}
        \item If $\dim(\mbox{Null}(J)) = 1$, then $\mbox{Null}(J) = \spa(\vec{\gamma}_{\text{product}})$. In this case, there does not exist a nonzero $\vec{\gamma}\in \mbox{Null}(J)$ linearly independent of $\vec{\gamma}_{\text{product}}$, and therefore no alternative solution to~\eqref{feasible} distinct from $\vec{\gamma}_{\text{product}}$ can exist. Hence, $G$ and $H$ are strongly disjoint.
        \item If $\dim(\mbox{Null}(J)) > 1$, then $\spa(\vec{\gamma}_{\text{product}}) \subsetneq \mbox{Null}(J)$, and there exists $\vec{\gamma}'\in \mbox{Null}(J)\setminus \spa(\vec{\gamma}_{\text{product}})$ with $\vec{\gamma}'\neq 0$. Scaling $\vec{\gamma}'$ to ensure
        $\mathbbm{1}^T\vec{\gamma}' =1$, the preceding argument yields the existence of a distinct solution $\vec{\gamma}^*\neq \vec{\gamma}_{\text{product}}$ to \eqref{feasible}. Thus, $G$ and $H$ are not strongly disjoint. 
    \end{itemize}
    In summary, determining whether $G$ and $H$ are strongly disjoint reduces to computing the dimension of $\mbox{Null}(J)$. Specifically, $G$ and $H$ are strongly disjoint if and only if $\dim(\mbox{Null}(J)) = 1$.
\end{proof}

\subsection{Theorem~\ref{Characterization_connected_no_self}}
\characterizationconnectednoself*


To prove Theorem~\ref{Characterization_connected_no_self}, we analyze all possible combinations of connectivity levels of the graphs $G$ and $H$, as discussed in Section~\ref{Characterizing}. The proof proceeds through a sequence of lemmas and propositions, corresponding to the entries listed in Table~\ref{fig:connectivity-table}. For each connectivity regime, we establish whether strong disjointness holds or fails, and we clarify its relationship with weak disjointness.

    In the proof, we will assume
    \begin{itemize}
        \item $G$ is a connected graph containing no self-loops that has a vertex set $U$ of cardinality $m$ and edge set $E(G)$, {where $G$ has $M$ undirected edges}. 
        \item $H$ is a connected graph containing no self-loops that has a vertex set $V$ of cardinality $n$ and an edge set $E(H)$, {where $H$ has $N$ undirected edges.}
    \end{itemize}
    {Since $G$ and $H$ have no self-loops, each undirected edge contributes exactly two directed edges (one in each direction). Hence we have that $|E(G)| = 2M$ and $|E(H)| = 2N$.}
    
    To analyze the joinings of $G$ and $H$, we adopt a linear-algebraic framework in which a weight joining is represented by a vector satisfying a system of linear constraints. The approach parallels the construction of the joining constraint matrix in Section~\ref{weight_joining_constraint_matrix}.
    Since we now impose the additional assumption that both graphs are connected and contain no self-loops, the constraints can be simplified relative to the general joining constraint matrix.  In order to distinguish this setting from the general case, we construct a \emph{simplified joining constraint matrix} $J_s$. Moreover, since the nonzero entries of the weight joining are represented as a vector, we use $\Vec{\gamma}_s$ to represent the simplified weight joining vector. 
    We now describe the construction of the simplified joining constraint matrix $J_s$ and the associated vector of unknowns $\Vec{\gamma}_s$. Let $\kappa$ denote the dimension of $\Vec{\gamma}_s$, which corresponds to the total number of unknown variables. Let $\iota$ denote the number of constraints in the linear system; this equals the number of rows in $J_s$ plus one additional constraint arising from the normalization condition.
    To formulate the unknowns and constraints, we proceed as follows.
    
    First, we identify the unknown variables represented in $\Vec{\gamma}_s$. Because the weight joining is a weight function, it satisfies a symmetry property: for any pairs of edges $(u,u')\in E(G)$ and $(v,v')\in E(H)$, $\gamma((u,v),(u',v')) = \gamma((u',v'),(u,v))$. Accordingly, each symmetric pair is treated as a single unknown variable. Moreover, by the edge preservation property, the weight joining entry at $((u,v),(u',v'))$ can only be positive if $(u,u')\in E(G)$ and $(v,v')\in E(H)$. 
    As $\gamma((u,v),(u',v'))$ may be distinct from $\gamma((u,v'),(u',v))$, the total number of unknowns is
    \begin{align}
        \kappa \,= \frac{|E(G)||E(H)|}{2}\, = \, \frac{(2M)(2N)}{2} \, = \, 2MN,
        \label{number_of_unknowns}
    \end{align}
    where the division by two accounts for the elimination of duplicates arising from symmetric pairs. 

    Next, we turn to the constraints of the linear system. The symmetry constraint has already been incorporated in the definition of $\Vec{\gamma}_s$, while the normalization constraint is given by $\mathbbm{1}^T \vec{\gamma}=1$. Since both graphs are connected, the degree coupling condition can be omitted and it suffices to consider the transition coupling condition. Recall that the first transition coupling condition requires that for each $(u,v) \in U \times V$,
    \begin{align*}
            \sum\limits_{y'\in V}\gamma((u,v),(u',y')) \, = \, \frac{\alpha(u,u')}{p(u)}r(u,v), \text{ for all } u'\in {N}(u).
        \end{align*}
    However, not all of these constraints are independent. Indeed, summing over 
    $u'\in {N}(u)$ gives
        \begin{align*}
            \sum_{u'\in {N}(u)}\sum\limits_{y'\in V}\gamma((u,v),(u',y')) \, = \,  \sum_{u'\in {N}(u)} \frac{\alpha(u,u')}{p(u)}r(u,v) \, = \, r(u,v),       
        \end{align*}
        which holds by definition of $p$. Consequently, for each fixed $u \in U$, one of the constraints is linearly dependent on the others and can be safely removed. Applying the same argument to the second transition coupling condition over $v' \in N(v)$, we remove one redundant constraint per $v$. Therefore, for each $(u,v) \in U \times V$, after removing these redundant constraints, the number of remaining constraints is $(|{N}(u)|-1)+(|{N}(v)|-1)$. 
        Summing over all $(u,v)\in U\times V$ and including the normalization constraint, the total number of remaining 
        constraints is
        $$\iota = {\left(\sum_{(u,v) \in U \times V} (|N(u)|-1) + (|N(v)| - 1) \right) + 1 = } (2 M-m) n+(2 N-n) m+1.$$
        Among the total $\iota$ constraints, $\iota-1$ are transition coupling constraints.
        Let $J_s$ be a matrix with $\iota-1$ 
        rows and $\kappa$ columns such that the entries of each row correspond to the coefficients of one of the remaining transition coupling constraints. 

Based on the equations above, the difference between $\kappa$ (recall that it is defined in \eqref{number_of_unknowns})
and $\iota$ is given by
    \begin{align}
        \kappa - \iota = 2MN-2Mn-2Nm+2mn-1 = 2(M-m)(N-n)-1.
\label{difference}
    \end{align}

    The problem of identifying the weight joining in $\mathcal{J}(\alpha,\beta)$ reduces to solving
    \begin{equation}
    \begin{aligned}
        J_s\Vec{\gamma}_s \,&=\, 0, \\
        \Vec{\gamma}_s \,&\geq \, 0, \\
        \mathbbm{1}^T \Vec{\gamma}_s \, &= \, 1,
        \label{reduced_feasible}
    \end{aligned}
    \end{equation}
    which is identical to \eqref{feasible}, except with $J$ replaced by $J_s$ and $\Vec{\gamma}$ replaced by $\Vec{\gamma}_s$.
    \begin{lemma}
        For connected graphs $G$ and $H$ with no self-loops, let $J_s$ be the simplified joining constraint matrix between them. Then $G$ and $H$ are strongly disjoint if and only if $\dim(\mbox{Null}(J_s)) = 1$.
        \label{lemma:strong_equiv_rank}
    \end{lemma}
    
    \begin{proof}
    We follow the arguments in Proposition~\ref{rank_equiv_strong}.
        By the previous discussion, there exists a one-to-one correspondence between weight joinings in $\mathcal{J}(\alpha,\beta)$ and solutions to the system \eqref{reduced_feasible}. Let $\Vec{\gamma}_{\text{product}}$ denote the column vector representation of the product joining $\alpha\otimes \beta$ in the simplified weight joining vector formulation.
        Then $\Vec{\gamma}_{\text{product}}$ satisfies the constraints in \eqref{reduced_feasible}, and $\Vec{\gamma}_{\text{product}}>0$. If there exists $\Vec{\gamma}_s\neq \Vec{\gamma}_{\text{product}}$ that satisfies \eqref{reduced_feasible}, then $G$ and $H$ are not strongly disjoint. In particular, if there exists $\Vec{\gamma}_s' \notin \text{span}(\Vec{\gamma}_{\text{product}})$ such that $J_s \Vec{\gamma}_s' = 0$ and $\mathbbm{1}^T\Vec{\gamma}_s' = 1$, then $\Vec{\gamma}_s = (1+t)\Vec{\gamma}_{\text{product}}-t\Vec{\gamma}_s'$ for $t$ sufficiently close to 0 satisfies \eqref{reduced_feasible}. Thus, $G$ and $H$ are strongly disjoint if and only if $\dim(\mbox{Null}(J_s)) = 1$. 
    \end{proof}

    We now begin to prove the results summarized in Table~\ref{fig:connectivity-table}.
    Using Lemma~\ref{lemma:strong_equiv_rank} together with our earlier arguments regarding the number of unknowns and constraints, we establish the following proposition. This result shows that in entries $\textcircled{\small{1}}$ and  $\textcircled{\small{3}}$ of Table~\ref{fig:connectivity-table} the corresponding pairs of graphs are not strongly disjoint.
\begin{prop} 
Let $G$ and $H$ be connected graphs with no self-loops, and let $J_s = J_s(G,H)\in\mathbb{R}^{(\iota-1)\times \kappa}$ be the simplified joining constraint matrix between $G$ and $H$.
    \begin{enumerate}
        \item If $\kappa - \iota > 0$, then $G$ and $H$ are not strongly disjoint.
        \item If $\kappa - \iota \leq 0$, then $G$ and $H$ are strongly disjoint if and only if $\rank(J_s) = \kappa - 1$.
        
    \end{enumerate}
    \label{possible_states}
\end{prop}
\begin{proof}
Since the constraint matrix $J_s \in \mathbb{R}^{(\iota-1)\times \kappa}$, 
if $\kappa - \iota > 0$ the Rank-Nullity Theorem gives 
$\dim(\mbox{Null}(J_s)) = \kappa-\rank(J_s) \geq \kappa - (\iota - 1)>1$, and 
therefore $G$ and $H$ are not strongly disjoint by Lemma~\ref{lemma:strong_equiv_rank}.

{We now prove statement (2).}
If the rank
of $J_s$ is $\kappa - 1$, then the dimension of the null space of $J_s$ is one, corresponding to the product joining. {By Lemma~\ref{lemma:strong_equiv_rank}}, this ensures strong disjointness. 
{Since the product joining lies in the null space of $J_s$,  the null space has dimension at least one. Therefore, if the rank of $J_s$ is not equal to $\kappa - 1$, then it must be strictly less than $\kappa - 1$.} 
Thus, in this case, {the Rank-Nullity Theorem gives} $\dim(\mbox{Null}(J_s)) = \kappa - \rank(J_s) > 1$. Therefore, by Lemma~\ref{lemma:strong_equiv_rank}, $G$ and $H$ are not strongly disjoint.
\end{proof}
Note that, by \eqref{difference}, case (1) in Proposition~\ref{possible_states} corresponds either to $M = m-1$ and $N = n-1$, or to $M > m$ and $N > n$. These conditions indicate that the connectivity levels of graphs $G$ and $H$ correspond to entries $\textcircled{\small{1}}$ and $\textcircled{\small{3}}$, respectively. Consequently, in entries $\textcircled{\small{1}}$ and $\textcircled{\small{3}}$, the corresponding pair of graphs are never strongly disjoint.

\vskip.2in

Next, we present several technical results to address the cases where one graph has a number of edges equal to its number of vertices, while the other has at least as many edges as vertices; these correspond to entries $\textcircled{\small{2}}$ and $\textcircled{\small{4}}$ in the table.

Given a vertex set $U$ and an edge set $E \subseteq U \times U$, 
consider the set of weight functions of graphs $G = (U,\alpha)$ with edge set $E$
defined in Section~\ref{Prevalence_of_Disjointness}, 
\begin{align*}
W(U,E) \, = \, 
\left\{ \text{weight functions } \alpha \text{ on } U \text{ such that }\supp(\alpha) = E \right\}.
\end{align*}
Recall that $W(U,E)$ can be viewed as a subset of $\mathbb{R}^{d}$
where $d = |\{\{u,u'\}\subset U:(u,u')\in E\}|-1$, and 
as such it is bounded and open in $\mathbb{R}^d$.

\begin{lemma}
    Let $G=(U,\alpha)$ and $H=(V,\beta)$. 
    Suppose there exists a joining $\gamma\in\mathcal{J}(\alpha,\beta)$ with degree function $r$ such that $\gamma\neq \alpha\otimes \beta$ and $r=p\otimes q$. Then for any $\tilde{\alpha}\in W(U,E(G))$ and $\tilde{\beta}\in {W}(V,E(H))$, the graph $(U,\tilde{\alpha})$ is not strongly disjoint from the graph $(V,\tilde{\beta}).$
    \label{weakly_then_no_strong}
\end{lemma}
\begin{proof}
We first prove that $G$ is not strongly disjoint from $(V,\tilde{\beta})$ for every 
$\tilde{\beta} \in W(V,E(H))$.  Fix such a $\tilde{\beta}$. As $W(V,E(H))$ is open 
in $\mathbb{R}^{d}$,
there exists $t \in (0,1)$ such that 
$$\beta' \, = \, \frac{1}{1-t} \tilde{\beta} - \frac{t}{1-t}\beta \, \in \, {W}(V,E(H)),$$ 
and therefore $\tilde{\beta}$ can be expressed as the convex combination 
$\tilde{\beta} = t \beta+(1-t)\beta'$.  
Let $q'$ and $\tilde{q}$ denote the degree functions of $\beta'$ and $\tilde{\beta}$, respectively.
Define 
$$\tilde{\gamma} \, = \, t \gamma+(1-t)\0\alpha\otimes \beta'.$$ 
Let $\tilde{r}$ denote the degree function of $\tilde{\gamma}.$
We claim that $\tilde{\gamma}\in \mathcal{J}(\alpha,\tilde{\beta})$.  To verify this, we first compute the degree function: using $r = p\otimes q$, for all $(u,v) \in U \times V$, we have
\begin{align*}
    \tilde{r} (u,v) \, = \, t r(u,v) + (1-t)p(u)q'(v) \, = \, p(u) (tq(v)+(1-t)q'(v)) \, = \, p(u)\tilde{q}(v).
\end{align*}
Now we verify the degree coupling conditions: for all $(u,v) \in U \times V$, we have
\begin{align*}
    &\sum\limits_{v'\in V}\tilde{r}(u,v')\, = \, p(u),\\
    & \sum\limits_{u'\in U}\tilde{r}(u',v)\, = \, \tilde{q}(v).
\end{align*}
Next, we check the transition coupling conditions: for all $u,u'\in U$ and $v'\in V$,
\begin{align*}
    \sum\limits_{v'\in V}\tilde{\gamma}((u,v),(u',v')) \, &= \, t \0 \frac{\alpha(u,u')}{p(u)} \0 r(u,v) + (1-t)\alpha(u,u') q'(v) \\
    &= \, \alpha(u,u') (tq(v)+(1-t)q'(v))\\
    &= \, \alpha(u,u') \tilde{q}(v)\\
    &= \, \frac{\alpha(u,u')}{p(u)} \0 \tilde{r}(u,v).
\end{align*}
An analogous argument gives the transition coupling conditions for all $u\in U$ and $v,v'\in V$.
Thus, $\tilde{\gamma}\in\mathcal{J}(\alpha,\tilde{\beta})$.
    Finally, since $\gamma\neq \alpha\otimes \beta$, 
    {there exists $u,u'\in U$ and $v,v'\in V$ such that $\gamma((u,v),(u',v'))\neq \alpha(u,u')\beta(v,v')$. As $t\neq 0$, we have
    \begin{align*}
        \tilde{\gamma}((u,v),(u',v')) &= t\gamma((u,v),(u',v'))+(1-t)\alpha(u,u')\beta'(v,v')\\
        &\neq t\alpha(u,u')\beta(v,v')+(1-t)\alpha(u,u')\beta'(v,v') = \alpha(u,u') \tilde{\beta}(v,v').
    \end{align*}}
    Hence, $\tilde{\gamma} \neq \alpha\otimes\tilde{\beta}$. 
    Therefore, $G$ and $(V,\tilde{\beta})$ are not strongly disjoint for any $\tilde{\beta} \in W(V,E(H))$. 

    Since $\tilde{\gamma} \neq \alpha \otimes \tilde{\beta}$ while $\tilde{r} = p \otimes \tilde{q}$, the condition required for the previous argument is satisfied. Applying that argument once more on $W(U,E(G))$, we conclude that for any $\tilde{\beta} \in W(V,E(H))$ and any $\tilde{\alpha} \in W(U,E(G))$, the pair $(U,\tilde{\alpha})$ is not strongly disjoint from $(V,\tilde{\beta})$.
\end{proof}

Before proceeding with the proof for entry~$\textcircled{\small{2}}$ and entry~$\textcircled{\small{4}}$, we first introduce the definition of a subgraph and present some preliminary results concerning subgraphs.

We say that a graph $G = (U, \alpha)$ is a \textbf{subgraph} of $\tilde{G} = (\tilde{U}, \tilde{\alpha})$, if the following conditions hold:
\begin{align*}
    &\tilde{U} = U \cup U' \quad \text{for some set }U' \text{ disjoint from }U, \text{ and }\\
    &{\alpha}(u,u') \, = \, \frac{\tilde{\alpha}(u,u')}{\sum_{x,x' \in U}\tilde{\alpha}(x,x')} \quad \text{for }u,u'\in U. 
\end{align*}

\begin{prop}
    Let $G=(U, \alpha)$ and $H=(V, \beta)$ be connected graphs with no self-loops, and suppose that $G$ has uniform edge weights. Assume that $G$ and $H$ are not strongly disjoint, and that there exists $\gamma \in \mathcal{J}(\alpha, \beta)$ such that $\gamma\neq \alpha \otimes \beta$ and the associated degree function $r$ satisfies $r = p \otimes q$, where $p$ and $q$ are the degree functions of $\alpha$ and $\beta$, respectively. Then for any connected graphs $\tilde{G} = (\tilde{U},\tilde{\alpha})$ {and $\tilde{H} = (\tilde{V},\tilde{\beta})$ with no self-loops and with uniform edge weights that contain $G$ and $H$ as subgraphs, respectively, the graphs $\tilde{G}$ and $\tilde{H}$ are also not strongly disjoint. Moreover, for any $\alpha'\in W(\tilde{U}, E(\tilde{G}))$ and $\beta' \in W(\tilde{V},E(\tilde{H}))$, $(\tilde{U},\alpha')$ and $(\tilde{V},\beta')$ are not strongly disjoint.}
     \label{subgraph_not_strong}
\end{prop}
\begin{proof}
    Suppose there exists a weight joining $\gamma \neq \alpha \otimes \beta$ of $\alpha$ and $\beta$ with degree function $r$ satisfying $r = p \otimes q$. {We first show that $\tilde{G}$ and $H$ are not strongly disjoint.} Since $G$ is a subgraph of $\tilde{G}$ and $\tilde{G}$ has uniform edge weights, we have $U \subseteq \tilde{U}$, and for all $u,u' \in U,$
    \begin{align*}
        &\tilde{\alpha}(u,u') \,=\, \frac{1}{2|E(\tilde{G})|}\mathbf{1}((u,u')\in E(G)),\\
        &\alpha(u,u') \,=\, \frac{1}{2|E(G)|}\mathbf{1}((u,u')\in E(G)).
    \end{align*}
    
    To prove that $\tilde{G}$ and $H$ are not strongly disjoint, we now construct a new weight joining $\tilde{\gamma}$ between $\tilde{G}$ and $H$ with $\tilde{\gamma}\neq\tilde{\alpha}\otimes\beta$. For any $(u,v),(u',v')\in \tilde{U}\times V$, set
    \begin{equation}
    \begin{aligned}
        \tilde{\gamma}((u,v),(u',v')) = \begin{cases}
            \gamma((u,v),(u',v')) \0 \frac{|E(G)|}{|E(\tilde{G})|}&\text{ if }u,u'\in U,\\
            \tilde{\alpha}(u,u') \beta(v,v')&\text{ otherwise.}
        \end{cases}
    \end{aligned}
    \label{constructed_gamma}
    \end{equation}
    We claim that $\tilde{\gamma} \in \mathcal{J}(\tilde{\alpha}, \beta)$ and $\tilde{\gamma} \neq \tilde{\alpha} \otimes \beta$. Let $\tilde{p}$ denote the degree function of $\tilde{\alpha}$.
    Since both $\tilde{G}$ and $H$ are connected, it suffices to verify the transition coupling conditions. We first examine the degree function $\tilde{r}$ associated with $\tilde{\gamma}$. If $u \in U$, then 
    \begin{align*}
        \tilde{r}(u,v) \,& = \, r(u,v) \0 \frac{|E(G)|}{|E(\tilde{G})|} + \sum_{u' \in \tilde{U}\setminus U}\tilde{\alpha}(u,u') q(v) 
        \, = \, p(u)q(v) \0 \frac{|E(G)|}{|E(\tilde{G})|} + \sum_{u' \in \tilde{U}\setminus U}\tilde{\alpha}(u,u') q(v) \\
        &= \, \sum_{u' \in U} \tilde{\alpha}(u,u')q(v) +\sum_{u' \in \tilde{U}\setminus U}\tilde{\alpha}(u,u') q(v) \, = \, \tilde{p}(u)q(v).
    \end{align*}
    On the other hand, if $u \in \tilde{U}\setminus U$, then
    \begin{align*}
        \tilde{r}(u,v) \, = \, \sum_{u' \in \tilde{U}} \tilde{\alpha}(u,u')q(v) \, = \, \tilde{p}(u) q(v).
    \end{align*}
    We now verify the transition coupling conditions. For the first transition coupling condition, if $u,u' \in U$, then for all $v \in V$,
    \begin{align*}
        \sum_{v' \in V}\tilde{\gamma}((u,v),(u',v')) \, = \, \alpha(u,u')q(v) \0 \frac{|E(G)|}{|E(\tilde{G})|} \, = \, 
        \tilde{\alpha}(u,u') q(v) \, = \, \frac{\tilde{\alpha}(u,u')}{\tilde{p}(u)}\tilde{r}(u,v),
    \end{align*}
    and otherwise, for all $v \in V$,
    \begin{align*}
        \sum_{v' \in V}\tilde{\gamma}((u,v),(u',v')) \, = \, \tilde{\alpha}(u,u') q(v) \, = \, \frac{\tilde{\alpha}(u,u')}{\tilde{p}(u)}\tilde{r}(u,v).
    \end{align*}
    For the second transition coupling condition, if $u \in U$, then for all $v,v' \in V$,
    \begin{align*}
        &\sum_{u' \in \tilde{U}} \tilde{\gamma}((u,v),(u',v')) \, = \, p(u)\beta(v,v') \0 \frac{|E(G)|}{|E(\tilde{G})|} + \sum_{u' \in \tilde{U} \setminus U} \tilde{\alpha}(u,u') \beta(v,v') \\ &\quad = \, \sum_{u' \in U}\tilde{\alpha}(u,u') \beta(v,v')+ \sum_{u' \in \tilde{U} \setminus U} \tilde{\alpha}(u,u') \beta(v,v') = \, \tilde{p}(u)\beta(v,v') \, = \, \frac{\beta(v,v')}{q(v)}\tilde{r}(u,v),
    \end{align*}
    and otherwise if $u \in \tilde{U} \setminus U$, then for all $v, v' \in V$,
    \begin{align*}
        \sum_{u' \in \tilde{U}} \tilde{\gamma}((u,v),(u',v')) \, = \, \tilde{p}(u) \beta(v,v') \, = \, \frac{\beta(v,v')}{q(v)}\tilde{r}(u,v). 
    \end{align*}
    Thus, $\tilde{\gamma} \in \mathcal{J}(\tilde{\alpha},\beta)$. Moreover, since $\gamma\neq \alpha\otimes \beta$, there exists $u,u' \in U$ and $v,v' \in V$, such that $\tilde{\gamma}((u,v),(u',v') \neq \tilde{\alpha}(u,u')\beta(v,v')$.
    Thus, $\tilde{G}$ and $H$ are not strongly disjoint.

    {Since $\tilde{G}$ and $H$ are not strongly disjoint, and the weight joining $\tilde{\gamma} \in \mathcal{J}(\tilde{\alpha},\beta)$ constructed in \eqref{constructed_gamma} satisfies $\tilde{\gamma} \neq \tilde{\alpha}\otimes\beta$ while the associated degree function satisfies $\tilde{r} = \tilde{p}\otimes q$, we may apply the previous argument again on $\tilde{G}$ and $H$, with $H$ being the subgraph of $\tilde{H}$. Hence $\tilde{G}$ and $\tilde{H}$ are not strongly disjoint. Note that the constructed weight joining of $\tilde{\alpha}$ and $\tilde{\beta}$ also satisfies the condition that the weight joining does not equal to $\tilde{\alpha}\otimes\tilde{\beta}$ while the associated degree function equals to $\tilde{p}\otimes\tilde{q}$, where $\tilde{q}$ is the degree function of $\tilde{\beta}$.}

    {Moreover, since the constructed weight joining of $\tilde{\alpha}$ and $\tilde{\beta}$ satisfies the condition of Lemma~\ref{weakly_then_no_strong}, we conclude that for any $\alpha'\in W(\tilde{U}, E(\tilde{G}))$ and $\beta' \in W(\tilde{V},E(\tilde{H}))$, $(\tilde{U},\alpha')$ and $(\tilde{V},\beta')$ are not strongly disjoint.}
\end{proof}


In the following proposition, we establish the result corresponding to entries $\textcircled{\small{2}}$ and $\textcircled{\small{4}}$ in Table~\ref{fig:connectivity-table}.
\begin{prop}
    Let $G$ and $H$ be connected graphs with no self-loops. Suppose $G$ satisfies $M = m$, i.e., the number of vertices equals the number of edges, and $H$ satisfies $N \geq n$, i.e., the number of vertices is less than or equal to the number of edges. In this case, the graphs $G$ and $H$ are not strongly disjoint.
    \label{moderate_connectivity}
\end{prop}
\begin{proof}

    Since $M=m$ and $N\geq n$, and $G$ and $H$ are both connected graphs without self-loops, it follows that $G$ contains exactly one cycle, while $H$ contains at least one. 
    
    We begin with the case where both graphs have uniform edge weights. Let $G = (U,\alpha)$ and $H = (V,\beta)$.
    By Proposition~\ref{circle_property}, the single-cycle subgraphs of $G$ and $H$ with uniform weights are not strongly disjoint, since any two uniformly weighted cycles are not strongly disjoint. Moreover, the weight joining $\gamma$ constructed for uniform weight cycles in \eqref{gamma_construction_cycle} satisfies the subgraph condition of Proposition~\ref{subgraph_not_strong}. 
    Therefore, Proposition~\ref{subgraph_not_strong} gives that $G$ and $H$ are not strongly disjoint, and for any ${\alpha'}\in W(U,E(G))$ and ${\beta'}\in W(V,E(H))$, the graphs $(U,{\alpha'})$ and $(V,{\beta'})$ are not strongly disjoint. 
    Therefore, any graphs $G$ and $H$ with $M = m$ and $N\geq n$ are not strongly disjoint, regardless of the edge weights.
\end{proof}


The following proposition establishes that, in the remaining entries~\textcircled{\small{5}} and~\textcircled{\small{6}} of Table~\ref{fig:connectivity-table}, weak disjointness necessarily implies strong disjointness. In other words, when one of the graphs is a tree, the pair is either strongly disjoint or not weakly disjoint. 

\begin{prop}
   Let $G$ and $H$ be connected graphs with no self-loops. {Assume that exactly one of $G$ and $H$ is a tree}. 
   If $G$ and $H$ are weakly disjoint, then they are necessarily strongly disjoint.  
   \label{one_tree_one_not_tree}
\end{prop}
{In the proof, a vertex is called a \textbf{leaf} if it is incident to exactly one edge. Let $H = (V,\beta)$. For any graph $G = (U,\alpha)$, we say that an edge $(u,u')$ of $G$ is \textbf{$H$-marked} if, for every $\gamma\in\mathcal{J}(\alpha,\beta)$, $\gamma((u,v),(u',v')) = \alpha(u,u')\beta(v,v')$ for all $v,v'\in V$. 
Observe that if every edge of $G$ is $H$-marked, then $G$ and $H$ are strongly disjoint.
Before proceeding with the proof, we establish the following claim.}

\vspace{2mm}
\noindent
\textit{Claim.} Suppose $G$ is a tree, and that $G$ and $H$ are weakly disjoint. 
Then 
\begin{enumerate}[label=(\roman*)]
    \item 
    every edge incident to a leaf of $G$ is $H$-marked;
    \item if, at a given vertex, all but one of the incident edges are $H$-marked, then the remaining edge is also $H$-marked.
\end{enumerate}
\begin{proof}[Proof of the Claim]
Let $G = (U,\alpha)$ and $H = (V,\beta)$ be as above.
    For (i), consider a leaf node $u\in U$ in $G$, which by definition has exactly one neighbor, which we denote by $u'$. 
{By the edge preservation property, the assumption of weak disjointness and the fact that $u$ is a leaf
node with only one neighbor $u'$,} for any $v,v'\in V$, and for any weight joining $\gamma \in \mathcal{J}(\alpha,\beta)$ with degree function $r$, we have
\begin{align*}
    \gamma((u,v),(u',v')) \,=\, \sum\limits_{x'\in U}\gamma((u,v),(x',v')) \,=\,\frac{\beta(v,v')}{q(v)}\0 r(u,v) \,=\, \beta(v,v')p(u)\,=\,\beta(v,v')\alpha(u,u').
\end{align*}
{We have shown that for any leaf node $u \in U$, for $u'\in N(u)$ and for any $v,v'\in V$, we have $\gamma((u,v),(u',v'))=\alpha(u,u')\beta(v,v')$. Thus (i) holds.}

To prove (ii), suppose for vertex $u \in U$, all the incident edges except $(u,u')\in E(G)$ are $H$-marked, that is, for all $x \in {N}(u)\setminus \{u'\}$,  we have $\gamma((u,v),(x,v')) = \alpha(u,x)\beta(v,v')$ for all $v,v' \in V$. Then for all $v,v'\in V$,
 \begin{align*}
\gamma((u,v),(u',v')) \, &= \,  \sum\limits_{x\in U}\gamma((u,v),(x,v')) - \sum\limits_{x\in U\setminus \{u'\}}\gamma((u,v),(x,v')) \\
&= \, p(u)\beta(v,v') -   \sum\limits_{x\in U\setminus \{u'\}}\alpha(u,x)\beta(v,v') \, = \, \alpha(u,u')\beta(v,v').
\end{align*}
Therefore the edge $(u,u')$ is also $H$-marked, and statement (ii) holds.
\end{proof}

\begin{proof}[Proof of Proposition~\ref{one_tree_one_not_tree}]
Let $G = (U,\alpha)$ be a tree and let $H=(V,\beta)$ be a connected graph with no self-loops that is not a tree, in particular, $|V|\leq|E(H)|$. The assumption that $H$ is not a tree is necessary, since if both graphs are trees, then they are not weakly disjoint. 
Suppose that $G$ and $H$ are weakly disjoint. 
We prove that $G$ and $H$ are strongly disjoint by induction along the tree structure of $G$. 

Fix an arbitrary root $r\in U$ of $G$. For each vertex $u \in U$, define $\text{depth}(u)$ to be the graph distance from $r$ to $u$, and let $D = \max_{u\in U}\text{depth}(u)$. We show by backwards induction on depth that every edge of $G$ is $H$-marked. Vertices of depth $D$ are leaves of the rooted tree. Hence by (i) of the preceding Claim, every edge incident to a vertex of depth $D$ is $H$-marked. Now fix $k<D$ and assume as the inductive hypothesis that every edge incident to a vertex of depth strictly greater than $k$ is $H$-marked. Let $u$ be a vertex with $\text{depth}(u) = k$. A \textbf{child} of $u$ is any neighbor $u'$ with $\text{depth}(u') = k+1$, and the \textbf{parent} of $u$ (if $u\neq r$) is its unique neighbor with depth $k-1$. Each child of $u$ has depth greater than $k$; therefore, by the inductive hypothesis, the edge between $u$ and each of its children is $H$-marked. Consequently, all edges incident to $u$ are $H$-marked, except possibly the edge between $u$ and its parent. Then by (ii) of the preceding Claim, this remaining parent edge must also be $H$-marked. This completes the induction from depth $D$ down to depth $0$. Hence every edge of the rooted tree $G$ is $H$-marked. Therefore, $G$ and $H$ are strongly disjoint.
\end{proof}
Note that the above argument establishes that when one of the graphs is a tree, weak disjointness implies strong disjointness.  However, in region \textcircled{\small{1}}, the graphs are never weakly disjoint, as all trees share the common factor $P_2$.

\begin{proof}[Proof of Theorem~\ref{Characterization_connected_no_self}]
    Let $G$ and $H$ be two connected graphs with no self-loops. We have shown that the result in Table~\ref{fig:connectivity-table} is true. 
    
    To prove the equivalence stated in the theorem, first suppose that $G$ and $H$ are strongly disjoint. By Table~\ref{fig:connectivity-table}, this can occur only in entries \textcircled{\small{5}} or \textcircled{\small{6}}, that is, when exactly one of the graphs is a tree. Moreover, since $G$ and $H$ are strongly disjoint, they are weakly disjoint. 
    
    Conversely, suppose that $G$ and $H$ are weakly disjoint and exactly one of the graphs is a tree. Then $G$ and $H$ fall into entries \textcircled{\small{5}} or \textcircled{\small{6}} of Table~\ref{fig:connectivity-table}, where weak disjointness coincides with strong disjointness. Therefore, $G$ and $H$ are strongly disjoint.
\end{proof}

\subsection{Corollary~\ref{Characterization_fully_supported_no_self}}
\characterizationfullysupportednoself*
\begin{proof}
    If both $G$ and $H$ are disconnected graphs with no self-loops, then by Lemma~\ref{disconnected_not_weak} they are not weakly disjoint, and hence not strongly disjoint. If $G$ is connected and $H$ is disconnected, then $G$ and $H$ are strongly disjoint if and only if $G$ is strongly disjoint from each connected component of $H$. By Theorem~\ref{Characterization_connected_no_self}, this is equivalent to requiring that $G$ be weakly disjoint from every connected component of $H$, and that either $G$ is a tree or every connected component of $H$ is a tree—that is, $H$ is a forest.
\end{proof}

\section{Proofs for Section~\ref{Prevalence_of_Disjointness}}
\label{proof_for_prevalence}

\subsection{Proposition~\ref{prevalence_strong_weak}}
\prevalence*
We first prove the proposition in the case of strong disjointness.
Before proceeding with the proof, we establish the following lemma.
\begin{lemma}
\label{lemma:strong_disjoint}
Let $U$ and $V$ be finite sets, and let $E\subseteq U\times U$ and $F \subseteq V\times V$. 
    Fix $\beta\in W(V,F)$. Suppose there exists $\alpha\in W(U,E)$ such that $(U,\alpha)$ and $(V,\beta)$ are strongly disjoint. Define
    \begin{align*}
     {A}_{\beta}\coloneqq \left\{\tilde{\alpha} \in {W}(U,E) \mid (U,\tilde\alpha) \text{ and }(V,{\beta}) \text{ are strongly disjoint} \right\}.
\end{align*}
    Then $A_{\beta} \subseteq W(U,E)$ is full. An analogous statement holds with the roles of $\alpha$ and $\beta$ interchanged, fixing $\alpha$ and considering the corresponding subset of $W(V,F)$.
\end{lemma}

\begin{proof}

By assumption, ${A}_{\beta}$ is nonempty since $\alpha \in A_{\beta}$.
For any $\alpha' \in W(U,E)$, let $J'\in\mathbb{R}^{k\times l}$ denote the joining constraint matrix associated with $(U,\alpha')$ and $(V,\beta)$. By Proposition~\ref{rank_equiv_strong}, we have $\alpha'\notin A_{\beta}$ if and only if $\dim(\mbox{Null}(J')) > 1$. By the Rank-Nullity Theorem, this is equivalent to $\rank(J') = l-\dim(\mbox{Null}(J')) < l-1$. Equivalently, all $(l-1)\times (l-1)$ minors of $J'$ must have zero determinant. Since $J'$ has only finitely many such minors, this condition defines a finite system of polynomial equations. Moreover, the entries in $J'$ depend analytically on the weight functions $\alpha'$ and $\beta$. As $\beta$ is fixed, the coefficients of this system depend only on $\alpha'$. It is known that a real analytic function on $\mathbb{R}^d$ is either identically zero or has a zero set of Lebesgue measure zero in $\mathbb{R}^d$ (\cite{Mityagin2020ZeroSet}). Since we consider only finitely many real analytic functions on $\mathbb{R}^d$ with $d = |\{\{u,u'\}\subset U:(u,u')\in E\}|-1$, the resulting system is either identically zero or has a zero set of Lebesgue measure zero in $\mathbb{R}^d$. 

Since $\alpha \in A_{\beta}$, the joint constraint matrix associated with $(U,\alpha)$ and $(V,\beta)$ has nullity exactly 1. Hence, the above analytic functions does not vanish at $\alpha$, and thus the system is not identically zero. Consequently, $\lambda_d(W(U,E)\setminus A_{\beta}) = 0$.
This gives $\lambda_d(A_{\beta}) = \lambda_d(W(U,E))$. 

Finally, we show that $A_{\beta}$ is open and dense. As shown above,  $W(U,E)\setminus A_{\beta}$ is the common zero set of finitely many polynomials, and hence is closed. Therefore $A_{\beta}$ is open. Moreover, since $W(U,E)\setminus A_{\beta}$ has Lebesgue measure zero, it cannot contain any nonempty open set. Thus it has empty interior, and $A_{\beta}$ is dense in $W(U,E)$.
\end{proof}

\begin{proof}[Proof of Proposition \ref{prevalence_strong_weak} for strong disjointness case]



   
We establish the proposition by showing that if (1) does not hold, then (2) holds. To that end, suppose (1) does not hold, \textit{i.e.},
there exists $\alpha \in W(U,E)$ and $\beta \in W(V,F)$ such that $(U,\alpha)$ and $(V,\beta)$ are strongly disjoint.
Applying Lemma~\ref{lemma:strong_disjoint} with $\beta$ fixed, we obtain a full subset $A_{\beta}\subseteq W(U,E)$. Set $A = A_{\beta}$. For each $\tilde{\alpha}\in A$, applying Lemma~\ref{lemma:strong_disjoint} again with $\tilde{\alpha}$ fixed yields a full subset $B_{\tilde{\alpha}}\subseteq W(U,E)$ such that, for every $\tilde{\beta}\in B_{\tilde{\alpha}}$, $(U,\tilde{\alpha})$ and $(V,\tilde{\beta})$ are not strongly disjoint.  
\end{proof}

Before proceeding with the proof for weak disjointness case, we establish the following lemma.

\begin{lemma}
    \label{lemma:weak_disjoint}
    Fix $\beta\in W(V,F)$. Suppose there exists $\alpha\in W(U,E)$ such that $(U,\alpha)$ and $(V,\beta)$ are weakly disjoint. Define
    \begin{align*}
     {A}_{\beta}\coloneqq \left\{\tilde{\alpha} \in {W}(U,E) \mid (U,\tilde\alpha) \text{ and }(V,{\beta}) \text{ are weakly disjoint} \right\}.
\end{align*}
    Then $A_{\beta} \subseteq W(U,E)$ is full. An analogous statement holds with the roles of $\alpha$ and $\beta$ interchanged, fixing $\alpha$ and considering the corresponding subset of $W(V,F)$.
\end{lemma}
\begin{proof}
By assumption, ${A}_{\beta}$ is nonempty since $\alpha \in A_{\beta}$. Let $Q \in \mathbb{R}^{n\times n}$ be the transition matrix associated with the graph $(V,\beta)$. For any $\alpha' \in W(U,E)$, let $P' \in \mathbb{R}^{m\times m}$ denote the transition matrix associated with $(U,\alpha')$. 
We first consider the case where at least one of the edge sets $E$ and $F$ are connected.
By Proposition~\ref{weak_no_share_eigen_general}, the condition $\alpha' \notin A_{\beta}$ is equivalent to the existence of an eigenvalue of $P'$, different from 1, that coincides with an eigenvalue of $Q$. Since $\beta$ is fixed, the transition matrix $Q$ and its spectrum are fixed. Let the eigenvalues of $Q$ other than 1 be denoted by $\{\lambda_j\}_{j=1}^k$ where $k\leq n-1$. Therefore, $\alpha' \notin A_{\beta}$ if and only if there exists $j\in\{1,\cdots,k\}$ such that $\lambda_j$ is an eigenvalue of $P'$, equivalently, $\det(P' - \lambda_j I_{m\times m}) = 0$. Since the entries of $P'$ depend linearly on $\alpha'$, each equation $\det(P' - \lambda_j I_{m\times m}) = 0$ defines a polynomial in $\alpha'$. Hence, the set $W(U,E)\setminus A_{\beta}$ is a finite union of zero sets of polynomials in $\mathbb{R}^d$, where $d = |\{\{u,u'\}\subset U:(u,u')\in E\}|-1$. By the fact that a real analytic function on $\mathbb{R}^d$ is either identically zero or has a zero set of Lebesgue measure zero in $\mathbb{R}^d$ (\cite{Mityagin2020ZeroSet}), this finite union of zero sets is either all of $W(U,E)$ or has Lebesgue measure zero.  If both edge sets $E$ and $F$ are disconnected, then by Lemma~\ref{disconnected_not_weak}, the graphs with skeletons $(U,E)$ and $(V,F)$ are not weakly disjoint. This contradicts our assumption that there already exists a pair that is weakly disjoint, and therefore this case cannot occur.

Since $\alpha\in A_{\beta}$, none of the above polynomials vanish at $\alpha$, and thus the system is not identically zero. Consequently, $\lambda_d(W(U,E)\setminus A_{\beta}) = 0$.

Finally, we show that $A_{\beta}$ is open and dense. As shown above,  $W(U,E)\setminus A_{\beta}$ is the union of the zero set of finitely many polynomials, and hence it is closed. Therefore, $A_{\beta}$ is open. Moreover, since $W(U,E)\setminus A_{\beta}$ has Lebesgue measure zero, it cannot contain any nonempty open set. Thus it has empty interior, and $A_{\beta}$ is dense in $W(U,E)$.
\end{proof}

\begin{proof}[Proof of Proposition \ref{prevalence_strong_weak} for weak disjointness case]



We establish the proposition by showing that if (1) does not hold, then (2) holds. To that
end, suppose (1) does not hold, that is, there exist $\alpha \in {W}(U,E)$ and $\beta \in {W}(V,F)$ such that $(U,\alpha)$ and $(V,\beta)$ are weakly disjoint. 
Applying Lemma~\ref{lemma:weak_disjoint} with $\beta$ fixed, we obtain a full subset $A_{\beta}\subseteq W(U,E)$. Set $A = A_{\beta}$. For each $\tilde{\alpha}\in A$, applying Lemma~\ref{lemma:weak_disjoint} again with $\tilde{\alpha}$ fixed yields a full subset $B_{\tilde{\alpha}}\subseteq W(U,E)$ such that, for every $\tilde{\beta}\in B_{\tilde{\alpha}}$, $(U,\tilde{\alpha})$ and $(V,\tilde{\beta})$ are not weakly disjoint.  
\end{proof}

\subsection{Corollary~\ref{zero_measure}}
\zeromeasure*
\begin{proof}

Consider the product weight space ${W}(U,E)\times {W}(V,F)$, equipped with the product measure $\lambda_{d_E} \otimes \lambda_{d_F}$, where $d_E = |\{\{u,u'\}\subset U:(u,u')\in E\}|-1$ and $d_F = |\{\{v,v'\}\subset V:(v,v')\in F\}|-1$.
If the second statement holds, then the weight function pairs $(\alpha,\beta)$ that fail to give weak disjointness between $(U,\alpha)$ and $(V,\beta)$ is contained in
\begin{equation} \label{Eqn:TwoSets}
    \left(({W}(U,E)\setminus {A})\times {W}(V,F)\right) \, \bigcup \, \left(\bigcup_{\alpha \in {A}}\left\{ \alpha \right\}\times \left({W}(V,F)\setminus {B}_{\alpha}\right)\right).
\end{equation}
Since $A$ is a full subset in $W(U,E)$, we have $\lambda_{d_E}({W}(U,E)\setminus{A}) = 0$, and the first set in (\ref{Eqn:TwoSets}) has measure zero in the product space. For the second set in (\ref{Eqn:TwoSets}), applying Fubini’s theorem yields
\begin{align*}
    (\lambda_{d_E}\otimes \lambda_{d_F}) \left(\bigcup_{\alpha \in {A}}\left\{ \alpha \right\}\times \left({W}(V,F)\setminus{B}_{\alpha}\right)\right) \, = \, \int_{ {A}}\lambda_{d_F} \left({W}(V,F)\setminus {B}_{\alpha}\right)\mathrm{d}\lambda_{d_E}(\alpha) \, = \, 0,
\end{align*}
since $\lambda_{d_F}(W(V,F)\setminus B_{\alpha}) = 0$ for any $\alpha$.
\end{proof}

\section{Proofs for Section~\ref{characterize_graph_family}}


    

\subsection{Proposition~\ref{bipartite_strongdisjoint}}
\bipartitestrong*
Before presenting the proof, we first state a result that will be used in the argument.
\begin{prop}
    Let $G$ be fully supported. Then the following statements are equivalent:
    \begin{itemize}
        \item $G$ is bipartite.
        \item $P_2 \0 | \0 G$.
    \end{itemize}
    \label{bipartite_factor}
\end{prop}
\begin{proof}
    Assume first that $G$ is bipartite. Then it is straightforward to verify that $P_2 \0 | \0 G$.

    Conversely, suppose $P_2 \0 | \0 G$ and let $G = (U,\alpha)$ and $P_2  = (V,\beta)$ where $V = \{v_1, v_2\}$. Then by definition, $\beta(v_1,v_1) = \beta(v_2,v_2) = 0$, and $\beta(v_1,v_2) = \beta(v_2,v_1) = 1/2$.  Let $f:U\to V$ be a factor map such that $P_2 \0 | \0 G$ under $f$. For any $u\in f^{-1}(v_1),$ the factor definition gives
    \begin{align*}
    & q(v_1) \sum\limits_{u' \in f^{-1}(v_1)}{\alpha(u,u')} \, = \, p(u) \beta(v_1,v_1), \\
    & q(v_1) \sum\limits_{u' \in f^{-1}(v_2)}\alpha(u,u') \, = \, p(u) \beta(v_1,v_2).
    \end{align*}
    Since $\beta(v_1,v_1) = 0$ and $\beta(v_1,v_2) = q(v_1)$, it follows that $\alpha(u,u')=0$ for all $u'\in f^{-1}(v_1)$, and the entire mass of the neighborhood of $u$ is supported on $f^{-1}(v_2)$. The same argument holds for $u \in f^{-1}(v_2)$, implying that there are no edges within either fiber $f^{-1}(v_1)$ or $f^{-1}(v_2)$, only between them. Since $G$ is fully supported, $p(u) > 0$ for all $u \in U$, and then there must exist $u' \in U$ such that $\alpha(u,u')>0$. Therefore, every edge must connect a vertex in $f^{-1}(v_1)$ to one in $f^{-1}(v_2)$, and thus $G$ is a bipartite graph.
\end{proof}

The following fact is a classical theorem in graph theory (see Theorem 1.2.18 in \cite{west2001introduction}). 
\begin{fact}
    A graph is bipartite if and only if it has no odd-length cycle.
    \label{bipartite_no_odd}
\end{fact}
With this result in hand, we establish the following lemma, which will be used in the proof of Proposition~\ref{bipartite_strongdisjoint}.
\begin{lemma}
     Let $G$ be a connected graph. If $G$ contains an odd-length cycle, then $G$ is strongly disjoint from $P_2$. Otherwise, $G$ is not weakly disjoint from $P_2$.
    \label{odd_even}
\end{lemma}    
\begin{proof}

    Let $G = (U,\alpha)$ and $H = (V,\beta)$, where $V = \{v_1,v_2\}$, and $\beta(v_1,v_2) = \beta(v_2,v_1) = 1/2$, $\beta(v_1,v_1) = \beta(v_2,v_2) = 0$. 
    
    Suppose first that $G$ contains no odd-length cycles. Then, by Fact~\ref{bipartite_no_odd}, $G$ is bipartite. By Proposition~\ref{bipartite_factor}, $P_2 \0 | \0 G$, and therefore $G$ and $P_2$ share a common nontrivial factor $P_2$. By Proposition~\ref{factor_to_weak_disjoint}, $G$ and $P_2$ are not weakly disjoint.

    Now suppose $G$ contains a cycle of odd length. Consider any weight joining $\gamma \in \mathcal{J}(\alpha,\beta)$ with degree function $r$. 
    By the degree coupling condition that $r(u,v_1) + r(u,v_2) = p(u) >0$, for each $u\in U$, at least one of $r(u,v_1)$ or $r(u,v_2)$ must be strictly positive. 
    Suppose $r(u,v_1)>0$ for some $u \in U$. Then for any neighbor $u'\in {N}(u),$ we have $r(u',v_2)>0$ since
    \begin{equation}
    \begin{aligned}
        \frac{\alpha(u,u')}{p(u)}r(u,v_1) \,
        &= \, \sum\limits_{v'\in V} \gamma((u,v_1),(u',v')) \, = \, \gamma((u,v_1),(u',v_2))\\
        &= \, \gamma((u',v_2),(u,v_1)) \, = \, \sum\limits_{v\in V} \gamma((u',v_2),(u,v)) \, = \, \frac{\alpha(u',u)}{p(u')}r(u',v_2).
        \label{adjacent_marginal}
    \end{aligned}
    \end{equation}
    A symmetric argument holds if $r(u,v_2)>0$, implying $r(u',v_1)>0$ for all $u'\in {N}(u)$.
    Now consider an odd-length cycle $u_1\to u_2\to \cdots,\to u_{2m-1}\to u_1$ in $G$, with $2m-1$ edges. Without loss of generality, assume $r(u_1,v_1)>0.$ Then, by the argument above, we have 
    $$r(u_2,v_2)>0, r(u_3,v_1)>0, \cdots, r(u_{2m-1},v_1)>0.$$ 
    Since $u_1\in {N}(u_{2m-1})$, this leads to $r(u_1,v_2)>0$,  and hence both $r(u,v_1)$ and $r(u,v_2)$ are positive. As $G$ is connected, this positivity propagates, and we obtain $r(u,v_1)>0$ and $r(u,v_2)>0$ for all $u\in U$, implying $\supp(r)=U \times V$ for all $\gamma\in \mathcal{J}(\alpha,\beta).$
    From Equation~\eqref{adjacent_marginal},
    for adjacent $u,u' \in U$,
    \begin{align*}
        \frac{r(u,v_1)}{r(u',v_2)} = \frac{p(u)}{p(u')}.
    \end{align*}
    Now consider a walk from $u$ that traverses the previous odd-length cycle and returns to $u$, e.g., 
    $$u\to u_{k_1}\cdots\to u_{k_s}\to u_{j_1}\to u_{j_{2}}\to\cdots u_{j_{2m-1}}\to u_{j_1}\to u_{k_s}\to\cdots\to u_{k_1}\to u,$$
    which also has an odd length. Applying the ratio chain from above, we find
    \begin{align*}
        \frac{r(u,v_1)}{r(u,v_2)} = \frac{r(u,v_1)}{r(u_{k_1},v_2)} \frac{r(u_{k_1},v_2)}{r(u_{k_2},v_1)} \cdots \frac{r(u_{k_2},v_2)}{r(u_{k_1},v_1)}\frac{r(u_{k_1},v_1)}{r(u,v_2)} = \frac{p(u)}{p(u_{k_1})} \cdots \frac{p(u_{k_1})}{p(u)} = 1.
    \end{align*} 
    Since $r(u,v_1)+r(u,v_2) = p(u)$, this gives $r(u,v_1)=r(u,v_2)={p(u)}/{2}=p(u)q(v_1)=p(u)q(v_2)$ for all $u \in U$. Thus, $r = p \otimes q$, and so $G$ and $H$ are weakly disjoint.
    Finally, for adjacent $u,u'$, we compute 
    \begin{align*}
        \gamma((u,v_1),(u',v_2)) \,
        = \, \frac{\alpha(u,u')}{p(u)}r(u,v_1) \, = \, \frac{\alpha(u,u')}{p(u)}\frac{p(u)}{2} \, = \, \frac{\alpha(u,u')}{2} \,         = \, \alpha(u,u') \beta(v_1,v_2),
    \end{align*}
    showing that $\gamma = \alpha\otimes\beta$.
    Therefore, $G$ and $H$ are strongly disjoint.
\end{proof}

\begin{proof}[Proof for Proposition~\ref{bipartite_strongdisjoint}]
    If $G$ is bipartite, then by Fact \ref{bipartite_no_odd}, it contains no odd-length cycles. Consequently, by Lemma~\ref{odd_even}, $G$ is not weakly disjoint from $P_2$, and therefore also not strongly disjoint from $P_2$.

    Conversely, if $G$ is not strongly disjoint from $P_2$, then by Lemma \ref{odd_even}, it contains no odd-length cycles. By Fact~\ref{bipartite_no_odd}, it follows that $G$ is bipartite.
\end{proof}

\subsection{Proposition~\ref{connected_self_loop}}
\connectedstrong*
\begin{proof}

    Let $G = (U,\alpha)$ and $H = (V,\beta)$ with $V = \{v_1,v_2\}$, and $\beta(v_1,v_2) = \beta(v_2,v_1) = 0$, $\beta(v_1,v_1) = 1 - \beta(v_2,v_2) \in (0,1)$.

    Assume that $G$ is connected. We prove $G$ is strongly disjoint from $H$. For any weight joining $\gamma \in \mathcal{J}(\alpha,\beta)$ with degree function $r$, the edge preservation property yields 
    \begin{equation}
    \begin{aligned}
        &\gamma((u,v_1),(u',v_1)) \, = \, \sum_{v' \in V}\gamma((u,v_1),(u',v')) \, =  \, \frac{\alpha(u,u')}{p(u)}\0 r(u,v_1),\\
        &\gamma((u,v_2),(u',v_2)) \, = \, \sum_{v' \in V} \gamma((u,v_2),(u',v')) \, = \,  \frac{\alpha(u,u')}{p(u)}\0 r(u,v_2).
        \label{edge_pres_result}
    \end{aligned}
    \end{equation}
    Therefore, for any adjacent pair $u,u'\in U$, we obtain
    \begin{align*}
        \frac{r(u,v_1)}{r(u,v_2)} \, = \, \frac{\gamma((u,v_1),(u',v_1))}{\gamma((u,v_2),(u',v_2))} \, = \, \frac{r(u',v_1)}{r(u',v_2)}.
    \end{align*}
    By connectedness of $G$, the ratio $r(u,v_1)/r(u,v_2)$ must be constant across all $u \in U$. Let this constant be $C$. Then 
    \begin{align*}
        q(v_1) \, = \, \sum\limits_{u\in U}r(u,v_1) \, = \, \sum\limits_{u\in U}Cr(u,v_2) \, = \, Cq(v_2),
    \end{align*}
    which yields $C = q(v_1)/q(v_2)$. Hence, 
    $$r(u,v_1) \, = \, \frac{r(u,v_1)}{r(u,v_1) + r(u,v_2)} \0 p(u) \, = \, \frac{C}{C+1} \0 p(u) = \frac{q(v_1)}{q(v_1) + q(v_2)}\0 p(u) \, = \, q(v_1)p(u),$$ 
    where the last equality follows from $q(v_1)+q(v_2) = 1$. Using \eqref{edge_pres_result}, for any $u,u' \in U$,
    \begin{align*}
        &\gamma((u,v_1),(u',v_1)) \, = \, \frac{\alpha(u,u')}{p(u)} \0 q(v_1)p(u) \, = \, \alpha(u,u')q(v_1) \, = \, \alpha(u,u')\beta(v_1,v_1),\\
        &\gamma((u,v_2),(u',v_2)) \, = \, \frac{\alpha(u,u')}{p(u)}\0 q(v_2)p(u) \, = \, \alpha(u,u')q(v_2) \, = \, \alpha(u,u')\beta(v_2,v_2).
    \end{align*}
    This shows that we must have $\gamma = \alpha \otimes \beta$, and thus $G$ and $H$ are strongly disjoint.

    Now suppose $G = (U, \alpha)$ is not connected. 
    We will construct a weight joining $\gamma \in \mathcal{J}(\alpha,\beta)$ such that $\gamma \neq \alpha \otimes \beta$.
    Since $G$ is disconnected, there exist disjoint nonempty subsets $U_1,U_2 \subseteq U$ such that $U = U_1 \sqcup U_2$ and $\alpha(u,u') = 0$ for all $u \in U_1$ and $u' \in U_2$. Define $C = \sum_{u,u' \in U_1}\alpha(u,u') \in (0,1).$ Let $t \in (0,1)$ be an arbitrary constant.
    We now define $\gamma$ as follows:
    \begin{align*}
        &\gamma((u,v_1),(u',v_1)) \, = \, \begin{cases}
            t \alpha(u,u')q(v_1) &\text{if }u,u' \in U_1,\\
            \frac{1-tC}{1-C}\0 \alpha(u,u') q(v_1)&\text{if } u,u' \in U_2,
        \end{cases}\\
        &\gamma((u,v_2),(u',v_2)) \, = \, \begin{cases}
            (1-tq(v_1))\alpha(u,u') &\text{if }u,u' \in U_1, \\
            \left(1- \frac{1-tC}{1-C}\0 q(v_1)\right)\alpha(u,u') &\text{if }u,u' \in U_2,
        \end{cases}\\
        & \gamma((u,v_1),(u',v_2)) \, = \gamma((u,v_2),(u',v_1)) \, = \, 0 \quad \text{for all }u,u' \in U.
    \end{align*}
    It is easy to see that $\gamma$ is symmetric and nonnegative.
    The resulting degree function is
    \begin{align*}
        &r(u,v_1) \, = \, \begin{cases}
            tp(u)q(v_1) &\text{if }u \in U_1, \\
            \frac{1-tC}{1-C}p(u) q(v_1)&\text{if }u\in U_2,
        \end{cases}\\
        &r(u,v_2) \, = \, \begin{cases}
            (1-tq(v_1))p(u) &\text{if }u \in U_1,\\
            \left(1- \frac{1-tC}{1-C}\0 q(v_1)\right) p(u) &\text{if } u\in U_2.
        \end{cases}
    \end{align*}
    Observe that for all $u,u' \in U$,
    \begin{align*}
        \sum_{v' \in V}\gamma((u,v_1),(u',v')) \, = \, \gamma((u,v_1),(u',v_1)) \, = \,  \frac{\alpha(u,u')}{p(u)}r(u,v_1),\\
        \sum_{v' \in V}\gamma((u,v_2),(u',v')) \, = \, \gamma((u,v_2),(u',v_2)) \, = \,  \frac{\alpha(u,u')}{p(u)}r(u,v_2).
    \end{align*}
    Moreover, for all $u \in U$,
    \begin{align*}
        \sum_{u' \in U} \gamma((u,v_1),(u',v_1)) \, = \, r(u,v_1),\\
        \sum_{u' \in U} \gamma((u,v_2),(u',v_2)) \, = \, r(u,v_2).
    \end{align*}
    For the degree coupling condition, for all $u \in U$,
    \begin{align*}
        \sum_{v\in V}r(u,v) \, = \, r(u,v_1)+r(u,v_2) \, = \, p(u).
    \end{align*}
    In addition,
    \begin{align*}
        &\sum_{u \in U}r(u,v_1) \, = \, \sum_{u \in U_1}r(u,v_1)+\sum_{u \in U_2}r(u,v_1) \, = \, tCq(v_1)+\frac{1-tC}{1-C}\0 (1-C)q(v_1) \, = \, q(v_1),\\
        &\sum_{u \in U}r(u,v_2) \, = \, \sum_{u \in U_1}r(u,v_2)+\sum_{u \in U_2}r(u,v_2) \\
        &\quad= \, (1-tq(v_1))C+\left(1-\frac{1-tC}{1-C}q(v_1)\right)(1-C) \, = \, 1-q(v_1) \, = \, q(v_2).
    \end{align*}
    Thus, we have shown that $\gamma \in \mathcal{J}(\alpha,\beta)$. However, since $t$ is an arbitrary constant in $(0,1)$ by construction, it follows that $\gamma\neq \alpha\otimes\beta$. Therefore,
    $G$ and $H$ are not strongly disjoint.
\end{proof}

\section{Proof for Section~\ref{Section:MC}}
\label{proof_for_markovchain}
\subsection{Proposition~\ref{bijective_chain_graph}}
\bijectivechaingraph*

\begin{proof}
Let $(U \times V, \gamma) \in \mathcal{J}(G,H)$ be a joining
of $G$ and $H$ with degree function $r$. 
For pairs $(u,v),(u',v') \in U \times V$, define
\[
R((u',v') \mid (u,v)) \,\coloneqq\, \frac{\gamma((u,v),(u',v'))}{r(u,v)},
\]
with the convention that $0/0 = 0$.
One may readily verify that $(R,r)$ 
determines a reversible Markov measure 
$\mu_{\gamma} \in \mathbb{L}_r(U \times V)$. 
Let $Z \in \mathbb{M}_r(U \times V)$ have measure $\mu_\gamma$.
The definition of weight joining ensures that 
$X \stackrel{d}{=} \pi_1(Z)$ and $Y \stackrel{d}{=} \pi_2(Z)$, 
and it follows that $\mu_{\gamma}$ is the law of a reversible
Markovian coupling of $X$ and $Y$. 

{Conversely, for $Z\in\mathbb{J}(X,Y)$, let $\mu = \mathrm{Law}(Z)$. Since $Z \in \mathbb{M}_r(U\times V)$, its law $\mu$ is contained in $\mathbb{L}_r(U\times V)$. 
Let $R$ be the transition matrix of $\mu$ and let $r$ be its stationary distribution. 
Define
\begin{align*}
    \gamma_{\mu}((u,v),(u',v')) \coloneqq R((u',v')\mid (u,v))r(u,v).
\end{align*}
Since $Z$ is reversible, $\gamma_{\mu}$ is symmetric and has degree function $r$.
Using the assumption that $Z\in \mathbb{J}(X,Y)$, we have by definition that $X \stackrel{d}{=} \pi_1(Z)$ and $Y \stackrel{d}{=} \pi_2(Z)$, and one can verify that $\gamma_{\mu} \in \mathcal{J}(\alpha,\beta)$. Consequently, $(U\times V,\gamma_{\mu})\in \mathcal{J}(G,H)$. 
}

{Finally, these two constructions are inverse to each other. Starting from $\gamma$, forming $\mu_\gamma$ and then reconstructing $\gamma_{\mu_\gamma}$ recovers $\gamma$. Likewise, beginning with $\mu$, forming $\gamma_\mu$ and then constructing $\mu_{\gamma_\mu}$ recovers $\mu$. Therefore, there is a bijection between $\mathcal{J}(G,H)$ and
$\{\text{Law}(Z) : Z \in \mathbb{J}(X,Y) \}$.}
\end{proof}

\bibliographystyle{plain}

\bibliography{references}

\newpage
\appendix

\section{Additional Proofs}\label{appendix:proofs}

\subsection{Proposition~\ref{connected_drop_coupling}}
\connecteddropcoupling*
\begin{proof}
Suppose $\gamma$ satisfies condition (b) of Definition~\ref{Def:weight_joining}. Let $r$ denote the degree function of $\gamma$. 
To prove that $r$ is a coupling of $p$ and $q$, it suffices to show $\sum_{v\in V}r(u,v)=p(u)$ and $\sum_{u\in U}r(u,v)=q(v)$.
    Since $r(u,v)=\sum_{(u',v')\in U\times V} \gamma((u,v),(u',v'))$, by the symmetry of $\gamma$ and the first transition coupling condition in Definition~\ref{Def:weight_joining}, we obtain
    \begin{align*}
        \sum\limits_{v\in V}r(u,v)\, & = \, \sum\limits_{v\in V} \sum\limits_{(u',v')\in U\times V} \gamma((u,v),(u',v'))\\
        &= \, \sum\limits_{(u',v')\in U\times V} \sum\limits_{v\in V} \gamma((u',v'),(u,v)) \\
        &= \, \sum\limits_{(u',v')\in U\times V} \frac{\alpha(u',u)}{p(u')}r(u',v')\\
        &= \, \sum\limits_{u'\in U} \frac{\alpha(u',u)}{p(u')} \sum\limits_{v'\in V} r(u',v').
    \end{align*}
    Define $z(u) \coloneqq \sum\limits_{v\in V}r(u,v)$. Then by the above display and the normalization property of $r$,
    \begin{align*}
        z(u) \, = \, \sum\limits_{u'\in U}\frac{\alpha(u',u)}{p(u')}\0 z(u'), \; \text{ and } \;
        \sum\limits_{u\in U}z(u) \, = \, 1.
    \end{align*}
    Interpreting ${\alpha(u',u)}/{p(u')}$ as the transition probability from $u'$ to $u$ in the reversible Markov chain associated with $G$, note that if $G$ is connected, then this chain is irreducible and hence admits a unique stationary distribution (see Corollary 1.17 \cite{levin2017markov}). Since $p$ is the stationary distribution of the Markov chain associated with $G$, the above equations therefore have the unique solution $z(u)=p(u)$ for all $u \in U$. An analogous argument applies to the second marginal yields $\sum_{u\in U}r(u,v)=q(v)$.
    Hence, $r$ is a coupling of $p$ and $q.$
\end{proof}

\subsection{Lemma~\ref{vertex_edge_preservation}}
\label{proof_vertex_edge_preservation}
\vertexedgepreservation*
\begin{proof}
It follows from the degree coupling condition in the definition of a 
weight joining that if a vertex $u \in U$ has degree $p(u) = 0$ then $r(u,v)=0$
for all $K \in \mathcal{J}(G,H)$ and all $v \in V$.  A similar 
conclusion holds if $q(v) = 0$ for some vertex $v \in V$.
This establishes the vertex preservation property.

Let $K \in \mathcal{J}(G,H)$.  If $p(u) = 0$, then for any vertex $v \in V$ 
vertex preservation yields that $r(u,v) = 0$.  Suppose then that $p(u) > 0$ and that $(u,u') \notin E(G)$.  
The joining condition yields that for all $v \in V$
\[
\sum_{v'\in V} \gamma((u,v),(u',v')) 
\, = \, \frac{0}{p(u)} \, r(u,v) \, = \, 0,
\]
which yields that $\gamma((u,v),(u',v')) = 0$ for all $v,v' \in V$. By an analogous argument for $(v,v') \notin E(H)$, we conclude that the edge preservation property holds.
\end{proof}

\subsection{Lemma~\ref{polynomial_time}}
\label{proof_of_polynomial_time}
\polynomialtime*
The detection of strong disjointness follows directly from Proposition~\ref{rank_equiv_strong}, since the rank of a weight joining constraint matrix can be computed in time polynomial in the number of its rows and columns, which in turn are polynomial in $|U|$ and $|V|$.
\begin{proof}[Proof for weak disjointness]
    We use the vector representation of a weight joining $\vec{\gamma}\in \mathbb{R}^k$ where $k = |E(G)||E(H)|$, introduced in Section~\ref{weight_joining_constraint_matrix}. Recall that there is a one-to-one correspondence between $\vec{\gamma}$ and a weight function $\gamma$ defined on $(U\times V)\times (U\times V)$.
    We now define a matrix $J_w$ that encodes, for each $(u,v)\in U\times V$, the constraint
    \begin{align*}
        \sum\limits_{(y',z')\in U\times V}\gamma((u,v),(y',z')) - p(u)q(v)\sum\limits_{(y,z)\in U\times V}\sum\limits_{(y',z')\in U\times V}\gamma((y,z),(y',z')) \, = \, 0.
    \end{align*}
    Each such constraint is represented by a row of the matrix $J_w$.
    If $J_w\Vec{\gamma} = 0$ and $\mathbbm{1}^T \Vec{\gamma} = 1$, then the associated weight function $\gamma$ with degree function $r$ satisfies $r = p\otimes q$.
    Let $J$ be the joining constraint matrix between $G$ and $H$.
    We claim that $G$ and $H$  are weakly disjoint if and only if
    \begin{align*}
        \rank(J) \, = \, \rank\left(\left[\begin{array}{c}
             J  \\
             J_w
        \end{array}\right]\right).
    \end{align*}

    If $\rank(J) = \rank(\left[\begin{array}{c}
             J  \\
             J_w
        \end{array}\right]),$ then there exists a matrix $D$ such that $J_w=DJ.$ For any $\vec{x}\in \mbox{Null}(J)$, we have $J_w\vec{x}=DJ\vec{x}=0$, thus $\vec{x}\in \mbox{Null}(J_w)$ and gives $\mbox{Null}(J)\subseteq \mbox{Null}(J_w).$ Since any weight joining $\gamma$ of $\alpha$ and $\beta$ has a corresponding vector representation $\vec{\gamma}$ satisfying $J\vec{\gamma} = 0$, we have $\vec{\gamma}\in \mbox{Null}(J)\subseteq \mbox{Null}(J_w)$. Together with the constraint $\mathbbm{1}^T \vec{\gamma} = 0$, which is also required for $\gamma$ to be a weight joining, our previous discussion gives that the degree function $r$ of $\gamma$ satisfies $r(u,v)=p(u)q(v)$. Hence, $G$ and $H$ are weakly disjoint.

    Now suppose $G$ and $H$ are weakly disjoint. For any $\Vec{x}\in \mbox{Null}(J)$ with $\Vec{x}\geq 0$, $\Vec{x}\neq 0$, we normalize it to satisfy $\mathbbm{1}^T \vec{x} = 1$, then the corresponding weight function $x$ satisfies $x \in \mathcal{J}(\alpha,\beta).$ Since $G$ and $H$ are weakly disjoint, it follows that $J_w\vec{x}=0$. Recall that $\vec{\gamma}_{\text{product}}$ is the vector representation of the product weight joining $\alpha\otimes\beta$. Now we consider two cases:
    \begin{itemize}
        \item $\dim((\mbox{Null}(J))=1$:  Then $\mbox{Null}(J)=\{\lambda \vec{\gamma}_{\text{product}}\}$. Since $\vec{\gamma}_{\text{product}}\subseteq \mbox{Null}(J_w)$, we have $\mbox{Null}(J)\subseteq \mbox{Null}(J_w)$.
        \item $\dim((\mbox{Null}(J))>1$: Suppose there exists $\vec{\gamma}\in \mbox{Null}(J)$ such that $\vec{\gamma}\notin\mbox{Null}(J_w)$, then $\vec{\gamma}\neq \lambda \vec{\gamma}_{\text{product}}$ for any $\lambda$. We normalize $\vec{\gamma}$ to satisfy $\mathbbm{1}^T \vec{\gamma} = 1$. Define $\vec{\gamma}'=t \vec{\gamma}+(1-t)\vec{\gamma}_{\text{product}}.$ Since $\vec{\gamma}_{\text{product}}>0$, there exists $t \neq 0$ such that $\vec{\gamma}' \geq 0.$ This yields that the corresponding weight function $\gamma'$ satisfies $\gamma'\in\mathcal{J}(\alpha,\beta)$. Since $G$ and $H$ are weakly disjoint, $\vec{\gamma}'\in \mbox{Null}(J_w)$. Together with $\vec{\gamma}_{\text{product}}\in\mbox{Null}(J_w)$, we must have $\vec{\gamma}\in \mbox{Null}(J_w)$, contradicting the assumption. Thus, such a $\vec{\gamma}$ cannot exist, and we conclude that $\mbox{Null}(J)\subseteq \mbox{Null}(J_w).$
    \end{itemize}
    From $\mbox{Null}(J)\subseteq \mbox{Null}(J_w)$, it follows that $\row(J_w)\subseteq \row(J)$. Therefore,  $\rank(J) = \rank(\left[\begin{array}{c}
             J  \\
             J_w
        \end{array}\right]).$ 

        In conclusion, to check weak disjointness between $G$ and $H$, it suffices to check whether
        \begin{align*}
            \rank(J) \, = \, \rank\left(\left[\begin{array}{c}
                 J  \\
                 J_w 
            \end{array}\right]\right).
        \end{align*}
        Since the rank of a matrix can be computed in time polynomial to the number of rows and the number of columns, which are polynomial in $|U|$ and $|V|$, we can check the weak disjointness in polynomial time.
\end{proof}

\begin{proof}[Proof for $c$-disjointness]
    We introduce a row vector $J_c$  encoding the single constraint
    \begin{align*}
        \sum\limits_{(u,v)\in U\times V}\sum\limits_{(u',v')\in U\times V}&c(u,v) \gamma((u,v),(u',v'))  \\
        &-\left(\sum\limits_{(y,z)\in U\times V}c(y,z)p(y)q(z)\right)\sum\limits_{(u,v)\in U\times V}\sum\limits_{(u',v')\in U\times V}\gamma((u,v),(u',v')) \, = \, 0.
    \end{align*}
    Let $J$ be the joining constraint matrix associated with $G$ and $H$.
    We claim that $G$ and $H$ are $c$-disjoint if and only if
    \begin{align*}
        \rank\left(\left[\begin{array}{c}
             J  \\
             J_c 
        \end{array}\right]\right) \, = \, \rank(J).
    \end{align*}
    If this rank equality holds, then 
    \begin{align*}
        \mbox{Null}\left(\left[\begin{array}{c}
             J  \\
             J_c 
        \end{array}\right]\right) \, = \, \mbox{Null}(J),
    \end{align*}
    which yields that $G$ and $H$ are $c$-disjoint. Conversely, if two graphs are $c$-disjoint, then by our previous argument, for any $\gamma \in \mathcal{J}(\alpha,\beta)$, the corresponding degree function $r$ satisfies $$\sum\limits_{(u,v)\in U\times V}c(u,v)r(u,v) \, = \, \sum\limits_{(y,z)\in U\times V}c(y,z)p(y)q(z).$$

    Thus, by verifying the rank condition above, we can determine whether $G$ and $H$ are $c$-disjoint.
    As in the case of weak disjointness, $c$-disjointness can also be checked in time polynomial in $|U|$ and $|V|$.
\end{proof}

\fontsize{9.5}{12}\selectfont\textsc{Yang Xiang, Department of Statistics and Operations Research, UNC Chapel Hill}

\textit{Email address}: \href{mailto:xya@unc.edu}{xya@unc.edu}
\vskip 0.2cm

\fontsize{9.5}{12}\selectfont\textsc{Kevin McGoff, Department of Mathematics and Statistics, UNC Charlotte}

\textit{Email address}: \href{mailto:kmcgoff1@charlotte.edu}{kmcgoff1@charlotte.edu}
\vskip 0.2cm

\fontsize{9.5}{12}\selectfont\textsc{Andrew B. Nobel, Department of Statistics and Operations Research, UNC Chapel Hill}

\textit{Email address}: \href{mailto:nobel@email.unc.edu}{nobel@email.unc.edu}

\end{document}